\documentclass[10pt,reqno]{amsart}
\usepackage{amsmath,amssymb,mathrsfs,graphicx}
\usepackage{ifthen}
\usepackage{a4wide,color}
\usepackage{caption}
\usepackage{sidecap}
\usepackage{rotating}
\usepackage{colortbl}

\usepackage{babel}

\numberwithin{equation}{section}

\numberwithin{figure}{section}

\theoremstyle{plain}

\theoremstyle{plain}

\theoremstyle{remark}

\theoremstyle{plain}

\allowdisplaybreaks
\let \pa \partial

\let \eps \varepsilon

\newtheorem{theorem}{Theorem}[section]
\newtheorem{lemma}{Lemma}[section]

\newtheorem{proposition}{Proposition}[section]
\newtheorem{remark}{Remark}[section]

\let \de=\delta
\let \eps=\varepsilon

\let \la=\lambda

\let \pa=\partial

\let \om=\omega

\makeatother
\providecommand{\lemmaname}{Lemma}
\providecommand{\propositionname}{Proposition}
\providecommand{\remarkname}{Remark}
\providecommand{\theoremname}{Theorem}

\begin{document}

\title[Boltzmann equation with hard potential]{
Space-time structure and particle-fluid duality of solutions for Boltzmann equation with hard potentials}
\author{Yu-Chu Lin}
\address{Yu-Chu Lin, Department of Mathematics, National Cheng Kung
University, Tainan, Taiwan}
\email{yuchu@mail.ncku.edu.tw }
\author{Haitao Wang}
\address{Haitao Wang, Institute of Natural Sciences and School of
Mathematical Sciences, MOE-LSA, CMA-Shanghai, Shanghai Jiao Tong University, Shanghai,
China}
\email{haitallica@sjtu.edu.cn}
\author{Kung-Chien Wu}
\address{Kung-Chien Wu, Department of Mathematics, National Cheng Kung
University, Tainan, Taiwan and National Center for Theoretical Sciences,
National Taiwan University, Taipei, Taiwan}
\email{kungchienwu@gmail.com}

\begin{abstract}
	We study the quantitative pointwise behavior of solutions to the Boltzmann equation for hard potentials and Maxwellian molecules, which generalize the hard sphere case introduced by Liu-Yu in 2004  \cite{[LiuYu]}. The large time behavior of the solution is dominated by fluid structures, similar to the hard sphere case \cite{[LinWangWu3], [LiuYu]}. However, unlike hard sphere, the spatial decay here depends on the potential power  $\gamma$ and the initial velocity weight.  A key challenge in this problem is the loss of velocity weight in linear estimates, which makes standard nonlinear iteration infeasible.
	
	To address this, we develop an Enhanced Mixture Lemma, demonstrating that mixing the transport and gain parts of the linearized collision operator can generate arbitrary-order regularity and decay in both space and velocity variables. This allows us to decompose the linearized solution into fluid (arbitrary regularity and velocity decay) and particle (rapid space-time decay, but with loss of velocity decay) parts, making it possible to solve the nonlinear problem through this particle-fluid duality.
\end{abstract}

\date{\today }

\maketitle

\section{Introduction}

\subsection{The model}

Consider the  Boltzmann equation%
\begin{equation}
\left\{
\begin{array}{l}
\partial _{t}F+\xi \cdot \nabla _{x}F=Q(F,F)\text{,} \\[4mm]
F(0,x,\xi )=F_{0}(x,\xi )\text{,}%
\end{array}%
\right. \quad (t,x,\xi )\in {\mathbb{R}}^{+}\times {\mathbb{R}}^{3}\times {%
\mathbb{R}}^{3}\text{,}  \label{bot.1.a}
\end{equation}%
where $F(t,x,\xi )$ is the velocity distribution function for the particles
at time $t\geq0$, position $x=(x_{1},x_{2},x_{3})\in {\mathbb{R}}^{3}$ and
microscopic velocity $\xi =(\xi _{1},\xi _{2},\xi _{3})\in {\mathbb{R}}^{3}$%
. The bilinear collision operator is given by
\begin{equation*}
Q(F,G)=\frac{1}{2}\int_{{\mathbb{R}}^{3}\times S^{2}}B(|\xi _{\ast }-\xi
|,\Omega )\left\{ F_{\ast }^{\prime }G^{\prime }+G_{\ast }^{\prime
}F^{\prime }-F_{\ast }G-FG_{\ast }\right\} d\xi _{\ast }d\Omega \text{,}
\end{equation*}%
with collision kernel
\begin{equation*}
B(\vartheta )=|\xi -\xi _{\ast }|^{\gamma }\beta (\vartheta ).
\end{equation*}%
Here the usual convention, i.e., $F=F\left( t,x,\xi \right) $, $F_{\ast}=F_{\ast
}\left( t,x,\xi _{\ast }\right) $, $F^{\prime }=F(t,x,\xi )$ and $F_{\ast
}^{\prime }=F\left( t,x,\xi _{\ast }^{\prime }\right) $, is used; the
post-collisional velocities of particles satisfy
\begin{equation*}
\xi ^{\prime }=\xi -[(\xi -\xi _{\ast })\cdot \Omega ]\Omega \text{,}\quad
\xi _{\ast }^{\prime }=\xi +[(\xi -\xi _{\ast })\cdot \Omega ]\Omega \text{,
\ \ }\Omega \in S^{2}\text{.}
\end{equation*}%

Throughout this paper, we consider the \textit{Maxwellian molecules} $\left( \gamma
=0\right) $ and \textit{hard potentials} ($0<\gamma< 1$); and $\beta (\vartheta )$
satisfies the Grad's angular cutoff assumption
\begin{equation*}
0<\beta (\vartheta )\leq C\left\vert \cos \vartheta \right\vert \text{,}
\end{equation*}%
for some constant $C>0$, where $\vartheta $ is defined by
\begin{equation*}
\cos \vartheta =\frac{|(\xi -\xi _{\ast })\cdot \Omega |}{|\xi -\xi _{\ast }|%
}\text{.}
\end{equation*}
The case that  $\gamma=1$ and $\beta (\vartheta )=\lvert \cos \vartheta\rvert$ is called hard sphere model.

It is known that the global Maxwellians are equilibriums to the
Boltzmann equation. In the perturbation regime near the Maxwellian, we look
for the solution in the form of
\begin{equation}
F=\mathcal{M}+\sqrt{\mathcal{M}}f\text{, \ \ \ }F_{0}(x,\xi )=\mathcal{M}%
+\varepsilon \sqrt{\mathcal{M}}f_{0}\text{,}  \label{F}
\end{equation}%
where $\varepsilon >0$ is sufficiently small, with a perturbation $f$ to $%
\mathcal{M}$, where the global Maxwellian $\mathcal{M}$ is normalized as
\begin{equation*}\displaystyle
\mathcal{M}= (2\pi)^{-3/2}e^{-\frac{\lvert \xi \rvert^2}{2}}\text{.}
\end{equation*}%
Substituting \eqref{F} into \eqref{bot.1.a},
the perturbation $f$ satisfies the following equation
\begin{equation}
\left\{
\begin{array}{l}
\partial _{t}f+\xi \cdot \nabla _{x}f=Lf+\Gamma \left( f,f\right)
\vspace {3mm}
\\
f|_{t=0}=\varepsilon f_{0}\text{, \ \ }\left( x,\xi \right) \in \mathbb{R}%
^{3}\times \mathbb{R}^{3}\text{,}%
\end{array}%
\right.  \label{Linearized}
\end{equation}%
where%
\begin{equation*}
Lf=2  \mathcal{M}^{-1/2}Q(  \mathcal{M}, \mathcal{M}^{1/2}f),\qquad \Gamma \left(
f,f\right) = \mathcal{M}^{-1/2}Q(  \mathcal{M}^{1/2}f, \mathcal{M}^{1/2}f) \text{.}
\end{equation*}%
The null space of $L$ is a five-dimensional vector space spanned by an
orthonormal basis $\{\chi _{i}\}_{i=0}^{4}$, where
\begin{equation*}
\mathrm{Ker}(L)=\left\{ \chi _{0},\chi _{i},\chi _{4}\right\} =\left\{ \mathcal{M}%
^{1/2},\ \xi _{i}\mathcal{M}^{1/2},\ \frac{1}{\sqrt{6}}(|\xi |^{2}-3)%
\mathcal{M}^{1/2}\right\} \,\text{,}\quad i=1\text{,}\ 2\text{,}\ 3\text{.}
\end{equation*}%

Based on this property, one can introduce the macro-micro decomposition: let $%
\mathrm{P}_{0}$ be the orthogonal projection with respect to the $L_{\xi
}^{2}$ inner product onto $\mathrm{Ker}(L)$, and $\mathrm{P}_{1}\equiv
\mathrm{Id}-\mathrm{P}_{0}$.

The goal of this paper is to study the explicit space-time pointwise structure of (\ref{Linearized}).

\subsection{Notations}
In order to proceed, we introduce some notations that will be frequently used throughout the paper. We denote $\left\langle \xi \right\rangle ^{\beta }=(1+|\xi
|^{2})^{\beta /2}$ and $\left\langle \xi \right\rangle_{D} ^{\beta }=(D^{2}+|\xi
|^{2})^{\beta /2}$ with $D>0$, $\beta \in {\mathbb{R}}$. For the microscopic variable $%
\xi $, we denote
\begin{equation*}
|g|_{L_{\xi }^{q}}=\Big(\int_{{\mathbb{R}}^{3}}|g|^{q}d\xi \Big)^{1/q}%
\hbox{
if }1\leq q<\infty \hbox{,}\quad \quad |g|_{L_{\xi }^{\infty }}=\sup_{\xi
\in {\mathbb{R}}^{3}}|g(\xi )|\hbox{,}
\end{equation*}%
and the weighted norms can be defined by
\begin{equation*}
|g|_{L_{\xi ,\beta }^{q}}=\Big(\int_{{\mathbb{R}}^{3}}\left\vert
\left\langle \xi \right\rangle ^{\beta }g\right\vert ^{q}d\xi \Big)^{1/q}%
\hbox{
if }1\leq q<\infty \hbox{,}\quad \quad |g|_{L_{\xi ,\beta }^{\infty
}}=\sup_{\xi \in {\mathbb{R}}^{3}}\left\vert \left\langle \xi \right\rangle
^{\beta }g(\xi )\right\vert \hbox{.}
\end{equation*}%
The $L_{\xi }^{2}$ inner product in ${\mathbb{R}}^{3}$ will be denoted by $%
\left\langle \cdot ,\cdot \right\rangle _{\xi }$, i.e.,
\begin{equation*}
\left\langle f,g\right\rangle _{\xi }=\int_{\mathbb{R}^3} f(\xi )\overline{g(\xi )}d\xi %
\hbox{.}
\end{equation*}%
For the Boltzmann equation, the natural norm in $\xi $ is $|\cdot
|_{L_{\sigma }^{2}}$, which is defined as
\begin{equation*}
|g|_{L_{\sigma }^{2}}^{2}=\left|\left \langle \xi \right \rangle ^{\frac{%
\gamma }{2}}g\right|_{L_{\xi }^{2}}^{2}\text{.}
\end{equation*}%

For the space variable $x$, we have similar notations, namely,
\begin{equation*}
|g|_{L_{x}^{q}}=\left( \int_{{\mathbb{R}^{3}}}|g|^{q}dx\right) ^{1/q}%
\hbox{
if }1\leq q<\infty \hbox{,}\quad \quad |g|_{L_{x}^{\infty }}=\sup_{x\in {%
\mathbb{R}^{3}}}|g(x)|\hbox{.}
\end{equation*}

Finally, with $\mathcal{X}$ and $\mathcal{Y}$ being two normed spaces, we
define%
\begin{equation*}
\left\Vert g\right\Vert _{\mathcal{XY}}=\left\vert \left\vert g\right\vert _{%
\mathcal{Y}}\right\vert _{\mathcal{X}}\hbox{,}
\end{equation*}%
and for simplicity, we denote
\begin{equation*}
\left\Vert g\right\Vert _{L^{2}}=\left\Vert g\right\Vert _{L_{\xi
}^{2}L_{x}^{2}}=\left( \int_{\mathbb{R}^{3}}\left\vert g\right\vert
_{L_{x}^{2}}^{2}d\xi \right) ^{1/2}\hbox{.}
\end{equation*}%
For any two functions $f(x,t)$ and $g(x,t)$, we define the space-time
convolution as
\begin{equation*}
f(x,t)\ast _{x,t}g(x,t)=\int_{0}^{t}\int_{{\mathbb{R}}^{3}}f(x-y,t-\tau
)g(y,\tau )dyd\tau .
\end{equation*}%
For simplicity of notations, hereafter, we abbreviate $\leq C$ to $\lesssim $%
, where $C$ is a constant depending only on fixed numbers.

\subsection{Review of previous works and main difficulties}

The perturbation equation \eqref{Linearized} has been studied extensively using spectral analysis and energy methods; see the classical works \cite{[Guo-04],[LiuYu-04],[LiuYangYu], [Nishida-78],[Ukai-74]} and references therein. These studies assume smooth initial data with high Sobolev regularity and obtain decay of the perturbation in a global norm.

To better understand the solution’s quantitative aspects, such as its singularity structure and wave behavior, \cite{[LiuYu]} initiated research on the Green’s function and space-time pointwise behavior for the 1D Boltzmann equation with hard sphere. This approach has proven crucial for the study of the initial-boundary value problem \cite{[LiuYu-07]}, pointwise shock stability \cite{[Yu-10]}, and invariant manifold theory \cite{[LiuYu-13]}. Pointwise behaviors have also been studied for the 3D linearized Boltzmann equation \cite{[LiuYu1],[LiuYu2]}, the Boltzmann equation coupled with an electromagnetic field \cite{[LiYangZhong-20],[LiYangZhong-23]}, and the Boltzmann equation with polynomial tail perturbation \cite{[LinWangWu2]}. Notably, all of these pointwise results are limited to the hard sphere model.

Therefore, it is natural to study the pointwise behavior for other potentials, such as hard, Maxwell gas, and soft potentials, with or without an angular cutoff. In particular, it would be very interesting to understand how solutions behave differently across these models.  Some attempts have been made to study the pointwise structures of Boltzmann equation with hard potentials by the authors and others \cite{[LeeLiuYu-07], [LinWangWu],[LinWangWu3]}, but only at the linearized level. These studies encounter a weight loss in microscopic velocity in the linear estimates, which prevents the nonlinear iteration and closure that has been achieved for the hard sphere case. Incidentally, if one focuses only on spatial asymptotic behavior of the solution, without detailed descriptions inside the finite Mach region, the weighted energy method can be used to obtain the desired results for the nonlinear problem \cite{[LinWangLyuWu]}.

Let us examine the challenges in more detail. It is known that for the angular cutoff model,  the linearized collision operator $L$ in \eqref{Linearized} consists of a multiplicative operator $\nu(\xi)$ and an integral operator $K$:
\[
Lh=-\nu(\xi) h+Kh,
\]
where $\nu(\xi)\sim (1+\lvert \xi \rvert)^{\gamma}$ depends on the potential. The linearized Boltzmann equation can be expressed as
\begin{equation}\label{eq:linear}
	\partial_t h+\xi \cdot\nabla h + \nu(\xi)h=Kh,\qquad  h\big|_{t=0}=h_0.
\end{equation}
Let $\mathbb{S}^t$ denote the solution operator for the damped transport equation on the left-hand side of \eqref{eq:linear}, which is explicitly given by
\begin{equation}\notag
	\mathbb{S}^th_0=e^{-\nu(\xi)t}h_0(x-\xi t,\xi),
\end{equation}
satisfying
\begin{equation}\label{eq:trans}
	\lvert \mathbb{S}^th_0\rvert =\lvert e^{-\nu(\xi)t}h_0(x-\xi t,\xi)\rvert  \lesssim e^{-\nu_0 (1+\lvert \xi \rvert)^\gamma t} \lvert h_0(x-\xi t,\xi)\rvert.
\end{equation}

In the soft potential case, where $\gamma$ is negative, the damping in the exponent becomes degenerate as $\xi$ grows large. This indicates that to achieve time decay, one must sacrifice some weight in the microscopic velocity, as discussed in \cite{[Caflisch],[LinWangLyuWu], [LinWangWu4]}, etc.

In the hard potential case, letting $x-\xi t=y$ and assuming that the initial data $h_0$ is compactly supported in space, say $h_0(x,\cdot)=0$ for $\lvert x\rvert >1$, it readily follows from \eqref{eq:trans} that
\begin{equation}\label{eq:tr-decay}
	\lvert \mathbb{S}^t h_0 \rvert \lesssim \sup\limits_{y\in\mathbb{R}^3} e^{-\nu_0 (t+\lvert x-y\rvert)^{\gamma} t^{1-\gamma}} \lvert h_0(y,\cdot) \rvert_{L^\infty_\xi}\lesssim e^{-\frac{t+\lvert x \rvert^\gamma t^{1-\gamma}}{C}} \lVert h_0 \rVert_{L^\infty_{x,\xi}}.
\end{equation}
When $\gamma=1$, i.e., the hard sphere case, we have
\[
\lVert \mathbb{S}^t h_0 \rVert_{L^\infty_{x,\xi}} \lesssim e^{-\frac{t+\lvert x\rvert}{C}} \lVert h_0 \rVert_{L^\infty_{x,\xi}}.
\]
Note that the solution and initial data are in the same weighted space.

By Duhamel's principle, the solution $h$ to \eqref{eq:linear} solves the integral equation
\begin{equation}\notag
	h(t)=\mathbb{S}^t h_0 +\int_{0}^{t} \mathbb{S}^{t-\tau} K h(\tau) d\tau.
\end{equation}
Then, a Picard-type iteration can be designed to extract the first several terms while \textit{maintaining the space-time exponential decay structure}.  These terms represent the most singular part of the solution and are referred as \textit{particle-like waves}. The interplay between operators $\mathbb{S}^t$ and $K$ provides a regularization mechanism for the remainder, known as the Mixture Lemma \cite{[LiuYu]}, which we will return to shortly.

However, in the hard potential and Maxwell gas cases where $0\leq \gamma<1$, one observes from \eqref{eq:tr-decay} that the spatial variable $x$ can become very large  while the damping exponent remains only of order $O(1)$. This implies that, although the initial data is compactly supported near the origin, it can quickly spread throughout the whole space without decay as time progresses due to the large microscopic velocity. To achieve spatial decay, we need to sacrifice some velocity weight, similarly to the trade-off required in the soft potential case to obtain time decay.  Moreover, velocity weight is sacrificed at each step in the Picard-type iteration, leading to a cumulative loss of velocity weight in the linear estimate, see \cite{[LinWangWu],[LinWangWu3]}. Thus, the nonlinear closure procedure for the hard sphere case cannot be applied.

On the other hand, classical spectral analysis \cite{[EP],[LiuYu], [Ukai-76]} shows that
for values of $\gamma$ within the range $0\leq \gamma\leq 1$, the discrete spectrum of the semigroup for the linearized Boltzmann equation,  $\displaystyle e^{(-\xi\cdot\nabla + L)t}$, exhibit similar behavior, corresponding to the \textit{fluid-like  component} of the solution and dominating its large-time behavior. Intuitively, the macroscopic fluid behavior is associated with thermodynamic equilibrium, close to a Maxwellian, and should be irrelevant to the loss of velocity weight.

Therefore, to overcome the loss of velocity weight while preserving the space-time structure of the solution, we need to develop a more complete separation of the particle and fluid parts in the linear estimates and implement a corresponding decomposition at the nonlinear level.

\subsection{Key ingredients and ideas} Now we introduce the key ingredients and ideas to realize the complete particle-fluid decomposition.

\subsubsection{Enhanced Mixture Lemma}
Motivated by \eqref{eq:linear}, a Picard-type iteration is designed by
\begin{align}
		\partial _{t}h^{(0)}+\xi \cdot \nabla _{x}h^{(0)}+\nu (\xi )h^{(0)}=0,\qquad &
	h^{(0)}\big|_{t=0}=h_{0}, \notag\\
		\partial _{t}h^{(m)}+\xi \cdot \nabla _{x}h^{(m)}+\nu (\xi
	)h^{(m)}=Kh^{(m-1)},\qquad  &h^{(m)}\big|_{t=0}=0,\mbox{ for }m\geq 1. \notag 
\end{align}
The solution can be written down explicitly
\begin{equation*}
	h^{(m)}\left( t,x,\xi \right) =\mathbf{M}_{m}^{t}h_{0}\text{, }\quad m\geq 0,
\end{equation*}
where \textit{mixture operator} $\mathbf{M}_{m}^{t}$ is defined by
\begin{equation*}
	\mathbf{M}_{m}^{t}h_{0}(x,\xi)=\int_{0}^{t}\int_{0}^{t_{1}}\cdots \int_{0}^{t_{m-1}}%
	\mathbb{S}^{t-t_{1}}K\mathbb{S}^{t_{1}-t_{2}}K\cdots \mathbb{S}%
	^{t_{m-1}-t_{m}}K\mathbb{S}^{t_{m}}h_{0}dt_{1}dt_{2}\cdots dt_{m}.
\end{equation*}

Heuristically, we can consider $K$ as a collision with ambient particles that obey a Maxwellian distribution, and regard the number of mixing as the number of times a particle collides with background during traveling. The more collisions the particles undergo, the closer they approach thermo-equilibrium.

As shown in \cite{[LiuYu]}, the mixture operator $\mathbf{M}_m^t$ can transfer regularity from the velocity variable to the spatial variable. Specifically, we have the following lemma:
\begin{lemma}[Mixture Lemma, \cite{[LiuYu]}] For each given $k\geq 0$, there exists a positive constant $C_k$ such that
\begin{equation*}
	\lVert \partial_x^k \mathbf{M}_{2k}^t h_0 \rVert_{L^2_x(L^2_\xi)} \leq C_k e^{-\nu_0 t} \left(		\lVert h_0 	\rVert_{L^2_x(L^2_\xi)}+ 	\lVert \partial_\xi^k h_0 	\rVert_{L^2_x(L^2_\xi)}		\right).
\end{equation*}
\end{lemma}
In the hard sphere case, as previously mentioned, particle-like waves decay exponentially in both space and time while retaining velocity weight, so the primary concern here is regularity. The above mixture lemma is sufficient for this purpose.

However, for hard potentials, additional challenges arise, including the loss of velocity weight and spatial decay. We therefore develop an enhanced version of the mixture lemma, which demonstrates  that mixing not only transfers regularity from velocity to space, but also generates regularity and decay in both space and velocity up to an arbitrary order as one desires.
\begin{lemma}[Enhanced Mixture Lemma] \hfill
	\begin{enumerate}
		\item For any $\alpha ,\beta \in \mathbb{N}_{0}^{3}$, any $P\geq 0$, when
		\begin{equation*}
			N\geq \max \left\{ P+\frac{|\alpha |+|\beta |+|\alpha |\gamma }{2},3\left(
			|\alpha |+|\beta |\right) \right\} \text{,}
		\end{equation*}%
		we have
		\begin{equation*}
			\left\Vert \left( 1+|\xi |\right) ^{P}\partial _{x}^{\alpha }\partial _{\xi
			}^{\beta }\mathbf{M}_{N}^{t}h_{0}\right\Vert _{L_{\xi }^{2}L_{x}^{2}}\leq
			Ce^{-ct}\left\Vert h_{0}\right\Vert _{L_{\xi }^{2}L_{x}^{2}}\text{.}
		\end{equation*}	
		\item Suppose $h_{0}\left( x,\xi \right) =0$ for $|x|>1$, for any $%
		\alpha ,\beta \in \mathbb{N}_{0}^{3}$, any $M, P\geq 0$, any $Q\in \mathbb{R}$,
		when
		\begin{equation*}
			N\geq \max \left\{ \frac{P+\left( |\alpha |+|\beta |+|\alpha |\gamma \right)
				+M\left( 1-\gamma \right) -Q}{2}\text{, }3\left( |\alpha |+|\beta |\right)
			\right\} \text{,}
		\end{equation*}%
		we have
		\begin{equation*}
			\left\vert \partial _{x}^{\alpha }\partial _{\xi }^{\beta }\mathbf{M}%
			_{N}^{t}h_{0}\left( t,x,\xi \right) \right\vert \leq Ce^{-c_{0}\nu (\xi
				)t}\left( 1+|x|\right) ^{-M}\left( 1+|\xi |\right) ^{-P}\left\Vert \left(
			1+|\xi |\right) ^{Q}h_{0}\right\Vert _{L_{x,\xi }^{\infty }}\text{.}
		\end{equation*}
	\end{enumerate}
\end{lemma}

\subsubsection{Particle-fluid duality} With the Enhanced Mixture Lemma, we decompose the solution into a particle-like part, which corresponds to those experiencing insufficient collisions, and a fluid-like part, which constitutes the remaining part. Specifically, we define:
\begin{equation*}
 \mathbb{G}^t_{\ast} h_0 \equiv \sum_{k=0}^{N} h^{(k)},\qquad  \mathbb{G}^t_{\sharp} h_0 \equiv h-\mathbb{G}^t_{\ast} h_0,
\end{equation*}
where  $h$ is the solution to linearized Boltzmann equation \eqref{eq:linear}, and $N$ is an integer chosen based on our needs. The fluid-like part $\mathbb{G}^t_{\sharp} h_0$, denoted by $\mathcal{R}^{(N)}$ for simplicity, satisfies
\begin{equation*}
\partial _{t}\mathcal{R}^{(N)}+\xi \cdot \nabla _{x}\mathcal{R}^{(N)}-L\mathcal{R}^{(N)}=Kh^{(N)}\text{,}\qquad \mathcal{R}^{(N)}\big|_{t=0}=0\text{.}
\end{equation*}%
Thanks to the Enhanced Mixture Lemma, the source term $Kh^{(N)}$  provides sufficient velocity decay and spatial decay. However, this decomposition has a drawback:  the solution operator $\mathbb{G}^t$ gains an extra decay of $(1+t)^{-1/2}$ when acting on purely microscopic data, but the decomposition above breaks the macro-micro structure. By combining the decompositions from both the spectral and particle-fluid perspectives, we  can reassemble the solution operator to restore the separation-of-scale property (see Theorem \ref{PF-estimate}).

The result reveals a duality between the particle and fluid aspects in the Boltzmann equation. The particle-like part inherits the singularity of the initial data, decays quickly in time, and its spatial decay rate depends on the intermolecular potential and initial velocity weight. Conversely, the fluid-like part is regular, governs the large-time behavior, decays rapidly in velocity, and propagates at essentially finite speed in space-time.

\subsubsection{Decomposition of the nonlinear solution}
Due to nonlinear interactions, we cannot expect a fully  separation of the particle and fluid components for the nonlinear problem, as is possible for the linearized case. Nevertheless, given their distinct structures, we use $\mathbb{G}_{\ast}^t$ and $\mathbb{G}_{\sharp}^t$ to decompose the solution to the nonlinear problem \eqref{Linearized} through the following coupled integral system:
\begin{equation}
\left\{
\begin{aligned}
&f_{\ast }=\eps\mathbb{G}_{\ast }^{t}f_{0}+\int_{0}^{t}\mathbb{G}_{\ast
}^{t-\tau }\Gamma \left( f_{\ast }+f_{\sharp },f_{\ast }+f_{\sharp }\right)
\left( \tau \right) d\tau \text{,}\\
&f_{\sharp }=\eps\mathbb{G}_{\sharp }^{t}f_{0}+\int_{0}^{t}\mathbb{G}_{\sharp
}^{t-\tau }\Gamma \left( f_{\ast }+f_{\sharp },f_{\ast }+f_{\sharp }\right)
\left( \tau \right) d\tau \text{.}
\end{aligned}
\right. \notag
\end{equation}

For $f_{\ast}$, the loss of velocity weight is intrinsic; we account for this by sacrificing some velocity weight and applying interpolation techniques to obtain the desired structure. In the case of $f_{\sharp}$, no velocity weight is lost, and we carry out the nonlinear wave interaction to reveal an explicit space-time structure, where the scale separation property of $\mathbb{G}^t_{\sharp}$ plays a crucial role.

To close the nonlinear problem, we propose a polynomial-type ansatz to facilitate estimates. However, since the initial data has compact spatial support and the fluid part should propagate at a finite speed, we expect the solution to exhibit faster decay outside the sound wave cone. We further employ a weighted energy method to improve the spatial asymptotic behavior of the solution.

\subsection{Main result}
Finally, we state the main theorem of this paper, which is  a consequence of Theorem \ref{thm:main} and Theorem \ref{thm:space-asy}, as below:
\begin{theorem}
Let $0\leq \gamma <1$, $\beta >3/2+\gamma $ and $\theta
_{0}\geq 42-34\gamma$. Assume that  $f_{0}$ is compactly
supported in the variable $x$ for all $\xi $, namely,
\begin{equation*}
f_{0}\left( x,\xi \right) \equiv 0\text{ for }\left\vert x\right\vert \geq 1%
\text{, }\xi \in \mathbb{R}^{3}\text{,}
\end{equation*}%
and $\displaystyle f_{0}\left(
x,\cdot \right) \in L_{\xi ,\hat{\beta}}^{\infty }$, where $\hat{\beta}=\beta +\theta _{0}$. For any fixed $\de>0$, there exists a sufficiently small $\eps>0$ such that the solution $f$ of the
nonlinear problem $(\ref{Linearized})$ satisfies
\begin{align*}
\left\vert f\right\vert _{L_{\xi ,\beta }^{\infty }}& \lesssim \eps\left[
\Phi_{m}(x,t)+(1+t)^{\frac{1}{2}}(1+t+|x|)^{-\frac{\theta _{0}}{1-\gamma }}\right] \\
& \quad +\eps^{2}\mathbf{1}_{\{ \left \vert x\right \vert \leq (%
\mathbf{c}+\de)t\}}\left[ \left( 1+t\right) ^{-2}B_{3/2}(|x|,t)+\left(
1+t\right) ^{-\frac{5}{2}}B_{1}(|x|-\mathbf{c}t,t)\right] \text{\thinspace ,}
\end{align*}%
where $\Phi_{m}$ and $B_k$ are  defined in \eqref{fluid-structure}, and $m$ can be any large positive number.  Here, $\mathbf{c}=\sqrt{5/3}$ is the sound speed associated with the global Maxwellian $\mathcal{M}$.
\end{theorem}

\begin{remark} The estimate of solution includes terms of orders $\varepsilon$ and $\varepsilon^2$. The $\varepsilon$ order terms represent linear waves, such as Huygens, diffusion, and Riesz waves. The $\varepsilon^2$ order terms arise from nonlinear interactions between these basic waves. They consist of polynomial versions of Huygens and diffusion waves, and primarily concentrate inside the sound wave cone. Compared to the linear waves, the nonlinear
waves not only have a smaller magnitude (order $\varepsilon^2$), but also decay faster by a factor $(1+t)^{-1/2}$ than their linear counterparts.
\end{remark}
\begin{remark} Compared to the estimates for hard sphere model (see e.g. \cite{[LinWangWu2]}), the linear waves here have a polynomial sharp estimate, while in the hard-sphere model they are exponentially sharp. This difference arises because, for hard potentials, the discrete spectrum lacks analyticity, unlike the hard-sphere case.  The estimates for the nonlinear part remain the same across both models. The key distinction appears in the spatial asymptotic behavior: for the hard-sphere model, the decay is exponential, while for hard potentials, the decay rate depends on both $\gamma$ and the velocity weight of the initial data.

\end{remark}

\subsection{Organization of the paper}

The rest of this paper is organized as follows: In Section \ref{sec:pre}, we first present the basic properties of the collision operator, review existing linear estimates, and introduce new derivative estimates for the integral operator $K$. In Section \ref{sec:linear}, we prove the Enhanced Mixture Lemma, decompose the solution operator into particle-like and fluid-like parts, and establish their respective estimates. In Section \ref{sec:nonlinear}, building on the linear estimates, we construct a particle-fluid decomposition for the nonlinear solution as well, develop the space-time structure within the sound wave cone through nonlinear wave interactions, and determine the spatial asymptotic behavior outside the sound wave cone using weighted energy estimates.


\section{Preliminaries}\label{sec:pre}

In this preliminary section, we will review some basic properties of the linearized collision operator $L$ and the collision operator $\Gamma$. Afterwards, we introduce several existing estimates for the linearized equation, and finally, we will prove some new derivative estimates for the operator $K$, which is useful in establishing the Enhanced Mixture Lemma.

\subsection{Basic properties of $L$ and $\Gamma$}

The linearized collision operator $L$, which has been analyzed extensively
by Grad \cite{[Grad]}, consists of a multiplication $\nu \left( \xi \right) $
and an integral operator $K=-K_{1}+K_{2}$:
\begin{equation*}
Lf=-\nu \left( \xi \right) f+Kf\text{,}
\end{equation*}%
where%
\begin{equation}
\nu (\xi )=\int_{\mathbb{R}^{3}}\int_{\mathbf{S}^{2}}\mathcal{M}(\xi _{\ast })B(|\xi
_{\ast }-\xi |,\Omega )d\Omega d\xi _{\ast }\text{,}  \label{Colli-freq}
\end{equation}%
\begin{equation*}
K_{1}f=\int_{\mathbb{R}^{3}}k_{1}\left( \xi ,\xi _{\ast }\right) f\left( \xi
_{\ast }\right) d\xi _{\ast }\text{, }
\end{equation*}%
with%
\begin{equation*}
k_{1}\left( \xi ,\xi _{\ast }\right) =\mathcal{M}^{1/2}(\xi )\int_{\mathbb{R}%
^{3}}\mathcal{M}^{1/2}(\xi _{\ast })B(|\xi _{\ast }-\xi |,\Omega )d\Omega \text{,}
\end{equation*}%
and%
\begin{equation*}
K_{2}f=\mathcal{M}^{-1/2}(\xi )\int_{\mathbb{R}^{3}}\int_{\mathbf{S}^{2}}\left( \mathcal{M}(\xi
^{\prime })\mathcal{M}^{1/2}(\xi _{\ast }^{\prime })f(\xi _{\ast }^{\prime })+\mathcal{M}(\xi
_{\ast }^{\prime })\mathcal{M}^{1/2}(\xi ^{\prime })f(\xi ^{\prime })\right) B(|\xi
_{\ast }-\xi |,\Omega )d\Omega d\xi _{\ast }\text{.}
\end{equation*}

In view of (\ref{Colli-freq}),
\begin{equation*}
\nu (\xi )=\left( 2\pi \right) ^{-3/2}\gamma _{0}\int_{\mathbb{R}%
^{3}}\left\vert \xi -\xi _{\ast }\right\vert ^{\gamma }e^{-\frac{\left\vert
\xi _{\ast }\right\vert ^{2}}{2}}d\xi _{\ast }\text{,}\ \ \gamma _{0}=\int_{%
\mathbf{S}^{2}}\beta (\theta )d\Omega \text{.}
\end{equation*}

In this subsection, we present a number of basic properties and estimates of
the operators $L$, $\nu (\xi )$ and $K$, which can be found in \cite{[Grad],
[LinWangWu]}.

\begin{lemma}
\label{lemma-colli-freq} Let $0\leq \gamma< 1$. For any $g\in L_{\sigma
}^{2} $, we have the coercivity of the linearized collision operator $L$,
that is, there exists a positive constant $\nu _{0}$ such that
\begin{equation*}
\left \langle g,Lg\right \rangle _{\xi }\leq -\nu _{0}\left \vert \mathrm{P}%
_{1}g\right \vert _{L_{\sigma }^{2}}^{2}\text{.}  \label{coercivity}
\end{equation*}%
For the multiplicative operator $\nu (\xi )$, there are positive constants $%
\nu _{0}$ and $\nu _{1}$ such that
\begin{equation*}
\nu _{0}\left \langle \xi \right \rangle ^{\gamma }\leq \nu (\xi )\leq \nu
_{1}\left \langle \xi \right \rangle ^{\gamma }\text{,}  \label{nu-gamma}
\end{equation*}%
and for any multi-index $\alpha\in\mathbb{N}^3_0 $,
\begin{equation}  \label{nu-deriv}
\left\vert \partial _{\xi }^{\alpha }\nu \left( \xi \right) \right\vert \leq
C\left( 1+\left\vert \xi \right\vert \right) ^{\gamma -\left\vert \alpha
\right\vert }\text{.}
\end{equation}
For the integral operator $K$,
\begin{equation*}
Kf=-K_{1}f+K_{2}f=\int_{{\mathbb{R}}^{3}}-k_{1}(\xi ,\xi _{\ast })f(\xi
_{\ast })d\xi _{\ast }+\int_{{\mathbb{R}}^{3}}k_{2}(\xi ,\xi _{\ast })f(\xi
_{\ast })d\xi _{\ast }\text{,}
\end{equation*}%
the kernels $k_{1}(\xi ,\xi _{\ast })$ and $k_{2}(\xi ,\xi _{\ast })$
satisfy
\begin{equation*}
k_{1}(\xi ,\xi _{\ast })\lesssim |\xi -\xi _{\ast }|^{\gamma }\exp \left \{ -%
\frac{1}{4}\left( |\xi |^{2}+|\xi _{\ast }|^{2}\right) \right \} \text{,}
\end{equation*}%
and%
\begin{equation*}
k_{2}(\xi ,\xi _{\ast })=a\left( \xi ,\xi _{\ast },\kappa \right) \exp
\left( -\frac{(1-\kappa )}{8}\left[ \frac{\left( \left \vert \xi \right
\vert ^{2}-\left \vert \xi _{\ast }\right \vert ^{2}\right) ^{2}}{\left
\vert \xi -\xi _{\ast }\right \vert ^{2}}+\left \vert \xi -\xi _{\ast
}\right \vert ^{2}\right] \right) \text{,}
\end{equation*}%
for any $0<\kappa <1$, together with
\begin{equation*}
a(\xi ,\xi _{\ast },\kappa )\leq C_{\kappa }|\xi -\xi _{\ast }|^{-1}(1+|\xi
|+|\xi _{\ast }|)^{\gamma -1}\text{.}
\end{equation*}
\end{lemma}

\begin{proposition}
Let $\tau \in \mathbb{R}$. Then
\begin{equation}
\int_{\mathbb{R}^{3}}\left\vert k\left( \xi ,\eta \right) \right\vert
\left\langle \eta \right\rangle ^{\tau }d\eta \lesssim \left\langle \xi
\right\rangle ^{\tau +\gamma -2}\,\text{,}\quad \int_{\mathbb{R}%
^{3}}\left\vert k\left( \xi ,\eta \right) \right\vert \left\langle \xi
\right\rangle ^{\tau }d\xi \lesssim \left\langle \eta \right\rangle ^{\tau
+\gamma -2}\text{.} \notag
\end{equation}%
Consequently, we have
\begin{equation}
|Kg|_{L_{\xi ,\tau +2-\gamma }^{q}}\lesssim |g|_{L_{\xi ,\tau }^{q}}\text{, }%
1\leq q\leq \infty \text{.} \notag
\end{equation}
\end{proposition}

For the nonlinear operator $\Gamma$, one has

\begin{lemma}[\cite{[Ukai-Yang]}]
\label{basic-Gamma}  Let $0\leq \gamma<1$ and $\tau \geq 0$%
. Then
\begin{equation*}
\left| \nu ^{-1}(\xi)\Gamma (h_{1},h_{2})\right| _{L_{\xi ,\tau }^{\infty
}}\lesssim \left| h_{1}\right|_{L_{\xi ,\tau }^{\infty }}\left|
h_{2}\right|_{L_{\xi ,\tau }^{\infty }}\text{.}
\end{equation*}
\end{lemma}

\subsection{Review of the linearized Boltzmann equation}\label{sec:2.2}

In this subsection, we consider the linearized Boltzmann equation
\begin{equation}\label{eq:linear0}
\partial _{t}h+\xi \cdot \nabla _{x}h=Lh,\quad \left. h\right\vert
_{t=0}=h_{0}\text{.}
\end{equation}%
We will review some basic properties of the solution. To analyze the precise
structure of the solution, the plan is to use spectral information to
decompose the linearized Boltzmann equation into two parts: fluid part and
non-fluid part. According to the semigroup theory, the solution of the
linearized Boltzmann equation $h$ can be represented by
\begin{equation*}
h\left( t,x,\xi \right) =\mathbb{G}^{t}h_{0}=\left( 2\pi \right) ^{-3}\int_{%
\mathbb{R}^{3}}e^{ix\cdot \eta }e^{\left( -i\xi \cdot \eta +L\right) t}%
\hat{h}_{0}\left( \eta \right) d\eta \text{,}
\end{equation*}%
where $\hat{h}_{0}$ is the Fourier transformation of $h_{0}$ with
respect to the space variable $x$. Based on the spectrum analysis of the
operator $-i\xi \cdot \eta +L$, the semigroup $e^{\left( -i\xi \cdot \eta
+L\right) t}$ can be decomposed as
\begin{align*}
e^{\left( -i\xi \cdot \eta +L\right) t} &=\chi _{\delta }\left( \eta
\right) \sum_{j=0}^{4}e^{\lambda _{j}\left( \eta \right) t}\left\vert
e_{j}\left( \eta \right) \big>\big<e_{j}\left( \eta \right) \right\vert
+\chi _{\delta }\left( \eta \right) e^{\left( -i\xi \cdot \eta +L\right)
t}\Pi _{\eta }^{\bot }+\left( 1-\chi _{\delta }\left( \eta \right) \right)
e^{\left( -i\xi \cdot \eta +L\right) t}%
\vspace {3mm}
\\
&=:\hat{\mathbb{G}}_{L;0}\left( \eta ,t\right) +\hat{\mathbb{G}}%
_{L;\bot }\left( \eta ,t\right) +\hat{\mathbb{G}}_{S}\left( \eta
,t\right)
\end{align*}%
where $\chi _{\delta }\left( \eta \right) $ is a smooth cutoff function with
$0\leq \chi _{\delta }\leq 1$, $\chi _{\delta }\left( \eta \right) =1$ for $%
\left\vert \eta \right\vert \leq \frac{\delta }{2}$ and $0$ for $\left\vert
\eta \right\vert \geq \delta $, for $\delta >0$ sufficient small. Note that
the spectrums of $\hat{\mathbb{G}}_{L;\bot }\left( \eta ,t\right) $ and $%
\hat{\mathbb{G}}_{S}\left( \eta ,t\right) $ are strictly away from the
imaginary axis (with negative real part). Moreover for $\left\vert \eta
\right\vert \ll 1$, the spectrum of the operator $-i\xi \cdot \eta +L$
consists of exactly five eigenvalues $\lambda _{j}\left( \eta \right) $ ($%
0\leq j\leq 4$) associated with the corresponding eigenfunctions $%
e_{j}\left( \eta \right) $ (\cite{[EP], [LiuYu1], [YangYu]}):
\begin{equation*}
\begin{array}{l}
\displaystyle\lambda _{j}(\eta )=-i\,a_{j}|\eta |-A_{j}|\eta |^{2}+O(|\eta
|^{3})\text{,} \\
\\
\displaystyle e_{j}(\eta )=E_{j}+O(|\eta |)\text{,}%
\end{array}%
\end{equation*}%
here $A_{j}>0$, $\left\langle e_{j}(-\eta ),e_{l}(\eta )\right\rangle _{\xi
}=\delta _{jl}$, $0\leq j$, $l\leq 4$ and
\begin{equation}
\left\{
\begin{array}{l}
\mathrm{P}_{0}\left( \xi \cdot \om\right) \mathrm{P}_{0}E_{j}=\mathrm{P}%
_{0}\left( \xi \cdot \om\right) E_{j}=a_{j}E_{j}\text{, } \\[2mm]
a_{0}=\sqrt{\frac{5}{3}}\text{,}\quad a_{1}=-\sqrt{\frac{5}{3}}\text{,}%
\,\quad a_{2}=a_{3}=a_{4}=0\text{,} \\
E_{0}=\sqrt{\frac{3}{10}}\chi _{0}+\sqrt{\frac{1}{2}}\om\cdot \overline{\chi
}+\sqrt{\frac{1}{5}}\chi _{4}\text{,} \\[2mm]
E_{1}=\sqrt{\frac{3}{10}}\chi _{0}-\sqrt{\frac{1}{2}}\om\cdot \overline{\chi
}+\sqrt{\frac{1}{5}}\chi _{4}\text{,} \\[2mm]
E_{2}=-\sqrt{\frac{2}{5}}\chi _{0}+\sqrt{\frac{3}{5}}\chi _{4}\,\text{,} \\%
[2mm]
E_{3}=\om_{1}\cdot \overline{\chi }\text{,} \\[2mm]
E_{4}=\om_{2}\cdot \overline{\chi }\text{,}%
\end{array}%
\right.  \label{Eigen function}
\end{equation}%
where $\overline{\chi }=(\chi _{1},\chi _{2},\chi _{3})$, $\eta =\left\vert
\eta \right\vert \om$ ($\om\in S^{2}$) and $\{\om_{1},\om_{2},\om\}$ is an
orthonormal basis of ${\mathbb{R}}^{3}$. And then we define
\begin{equation*}
h_{L;0}=\mathbb{G}_{L;0}^{t}h_{0}=\left( 2\pi \right) ^{-3}\int_{\mathbb{R}%
^{3}}e^{ix\cdot \eta }\hat{\mathbb{G}}_{L;0}\left( \eta ,t\right)
\hat{h}_{0}\left( \eta \right) d\eta \text{,}
\end{equation*}%
\begin{equation*}
h_{L;\bot }=\mathbb{G}_{L;\bot }^{t}h_{0}=\left( 2\pi \right) ^{-3}\int_{%
\mathbb{R}^{3}}e^{ix\cdot \eta }\hat{\mathbb{G}}_{L;\bot }\left( \eta
,t\right) \hat{h}_{0}\left( \eta \right) d\eta \text{,}
\end{equation*}%
\begin{equation*}
h_{S}=\mathbb{G}_{S}^{t}h_{0}=\left( 2\pi \right) ^{-3}\int_{\mathbb{R}%
^{3}}e^{ix\cdot \eta }\hat{\mathbb{G}}_{S}\left( \eta ,t\right) \hat{%
h}_{0}\left( \eta \right) d\eta \text{,}
\end{equation*}%
called the fluid part and nonfluid part of the long wave of $h$, and the
short wave of $h$, respectively. In the meanwhile, we define%
\begin{equation*}
h_{0L}:=\left( 2\pi \right) ^{-3}\int_{\mathbb{R}^{3}}e^{ix\cdot \eta }\chi
_{\delta }\left( \eta \right) \hat{h}_{0}\left( \eta \right) d\eta \text{%
,}
\end{equation*}%
\begin{equation*}
h_{0S}:=\left( 2\pi \right) ^{-3}\int_{\mathbb{R}^{3}}e^{ix\cdot \eta
}\left( 1-\chi _{\delta }\left( \eta \right) \right) \hat{h}_{0}\left(
\eta \right) d\eta \text{.}
\end{equation*}%
One can see that the long wave $h_{L}:=h_{L;0}+h_{L;\bot }$ satisfies the
equation%
\begin{equation}
\left\{
\begin{array}{l}
\partial _{t}h_{L}+\xi \cdot \nabla _{x}h_{L}=Lh_{L}\text{,}%
\vspace {3mm}
\\
h_{L}\left( 0,x,\xi \right) =h_{0L}\left( x,\xi \right) \text{,}%
\end{array}%
\right. \notag 
\end{equation}%
and the short wave satisfies%
\begin{equation}
\left\{
\begin{array}{l}
\partial _{t}h_{S}+\xi \cdot \nabla _{x}h_{S}=Lh_{S}\text{,}
\vspace {3mm}
\\
h_{S}\left( 0,x,\xi \right) =h_{0S}\left( x,\xi \right) \text{.}%
\end{array}%
\right. \notag 
\end{equation}%
The pointwise behavior of the leading fluid part is determined by how the
solution depends on the Fourier transformed variable of the spatial variable
$\eta $, i.e., smoothly or analytically. For the three space dimensional
Boltzmann equation with hard sphere \cite{[LiuYu1],[LiuYu2]}, with the
analytic property of the eigenfunction $\la_{j}(|\eta |)$, one has the fluid
structure of the exponential type. However, for the hard potential case $%
0\leq \gamma <1$, the dependence is no more analytic. We apply the
multiplier method and follow the idea from hard sphere to obtain the almost
exponential pointwise estimate in \cite{[Wang-Wu]}.

\begin{proposition} \label{Prop-linear-fluid}Let $h$ be a solution of the
linearized Boltzmann equation \eqref{eq:linear0} with the initial data $h_{0}$. Then for any multi-indices $\alpha,\beta\in\mathbb{N}_0^3$, any $Q\geq 0$
\begin{equation*}
\left\vert \pa_{x}^{\alpha }\pa_{\xi }^{\beta }\mathbb{G}_{L;0}^{t}h_{0}%
\right\vert _{L_{\xi ,Q}^{\infty }}\lesssim \left[ (1+t)^{-\frac{|\alpha |}{2%
}}\Phi_m (t,x)+e^{-ct}\right] \Vert h_{0}\Vert _{L_{x}^{\infty }L_{\xi }^{2}}%
\text{,}
\end{equation*}%
\begin{equation*}
\left\vert \pa_{x}^{\alpha }\pa_{\xi }^{\beta }\mathbb{G}_{L;0}^{t}\mathrm{P}%
_{1}h_{0}\right\vert _{L_{\xi ,Q}^{\infty }}\lesssim \left[ (1+t)^{-\frac{%
|\alpha |+1}{2}}\Phi_m (t,x)+e^{-ct}\right] \Vert h_{0}\Vert _{L_{x}^{\infty
}L_{\xi }^{2}}\text{.}
\end{equation*}%
Here
\begin{equation}\label{fluid-structure}
\Phi_{m} (t,x)=\left[
\begin{array}{l}
\mathbf{1}_{\{\left\vert x\right\vert \leq \mathbf{c}t\}}\left( 1+t\right)
^{-3/2}B_{3/2}\left( \left\vert x\right\vert ,t\right) \\[2mm]
+\left( 1+t\right) ^{-3/2}B_{m}\left( \left\vert x\right\vert ,t\right)
+\left( 1+t\right) ^{-2}B_{m}\left( \left\vert x\right\vert -\mathbf{c}%
t,t\right)%
\end{array}%
\right]
\end{equation}%
for all $m>0$, the number $\mathbf{c=}\sqrt{\frac{5}{3}}$ is the sound
speed, $\mathbf{1}_{D}$ is the characteristic function of the set $D$ and
\begin{equation*}
B_{m}\left( z,t\right) =\left( 1+\frac{z^{2}}{1+t}\right) ^{-m}\text{.}
\end{equation*}
\end{proposition}

We also have the following time decay estimates from \cite{[LinWangWu1]}:
\begin{proposition}
 \label{Prop-linear}Let $h$ be a solution of the
linearized Boltzmann equation \eqref{eq:linear0} with the initial data $h_{0}$. Then for any  $\alpha,\beta\in\mathbb{N}_0^3$, any $Q>3/2$
\begin{equation*}
\left\Vert \pa_{\xi }^{\beta }\mathbb{G}_{L;0}^{t}h_{0}\right\Vert _{L_{\xi
,Q}^{\infty }L_{x}^{\infty }}\lesssim (1+t)^{-\frac{3}{2}}\Vert h_{0}\Vert
_{L_{\xi }^{\infty }(L_{x}^{\infty }\cap L_{x}^{1})}\text{,}
\end{equation*}%
\begin{equation*}
\left\Vert \mathbb{G}_{S}^{t}h_{0}\right\Vert _{L_{\xi ,Q}^{\infty
}L_{x}^{\infty }}\lesssim e^{-ct}\left( \Vert h_{0}\Vert _{L_{\xi
,Q}^{\infty }L_{x}^{\infty }}+|h_{0}\Vert _{L_{\xi }^{2}L_{x}^{2}}\right)
\text{,}
\end{equation*}%
\begin{equation*}
\left\Vert \pa_{x}^{\alpha }\mathbb{G}_{L;\perp }^{t}h_{0}\right\Vert
_{L_{\xi ,Q}^{\infty }L_{x}^{\infty }}\lesssim e^{-ct}\left( \Vert
h_{0}\Vert _{L_{\xi ,Q}^{\infty }L_{x}^{\infty }}+|h_{0}\Vert _{L_{\xi
}^{2}L_{x}^{2}}\right) \text{.}
\end{equation*}%
\end{proposition}

\subsection{New derivative estimates of the integral operator $K$}

In this subsection, we revisit the integral operator $K$. We
will study its derivative estimates, which is necessary for
the proof of the Enhanced Mixture Lemma.

\subsubsection{Estimate of $K_{1}$ operator\label{K1-Estimate}}
The estimates for $K_{1}$ operator are straightforward. It is an integral operator with kernel $%
k_{1}\left( \xi ,\xi _{\ast }\right) $, i.e.,
\begin{equation*}
K_{1}f=\int_{\mathbb{R}^{3}}k_{1}(\xi ,\xi _{\ast })f(\xi _{\ast })d\xi
_{\ast }
\end{equation*}%
with%
\begin{align*}
k_{1}(\xi ,\xi _{\ast })& =\mathcal{M}^{1/2}(\xi )\int_{\mathbf{S}^{2}}M^{1/2}(\xi
_{\ast })B(|\xi _{\ast }-\xi |,\Omega )d\Omega \\
& =(2\pi )^{-3/2}\gamma _{0}|\xi _{\ast }-\xi |^{\gamma }e^{-\frac{|\xi
|^{2}+|\xi _{\ast }|^{2}}{4}}\text{.}
\end{align*}%
Letting
\begin{equation*}
\xi _{\ast }+\xi =u\text{,}\quad \xi _{\ast }-\xi =v\text{,}
\end{equation*}%
we have
\begin{equation*}
\partial _{\xi _{\ast }}=\partial _{u}+\partial _{v}\text{,}\quad \partial
_{\xi }=\partial _{u}-\partial _{v}\text{.}
\end{equation*}%
In terms of $u$ and $v$, $k_1$ is expressed as (by abuse of notation, and without causing confusion)
\begin{equation*}
k_{1}(u,v)=C|v|^{\gamma }e^{-\frac{|u|^{2}+|v|^{2}}{8}}\text{,}
\end{equation*}%
and
\begin{equation*}
|\partial _{u}^{\beta }k_{1}(u,v)|\leq C|v|^{\gamma }e^{-\frac{%
|u|^{2}+|v|^{2}}{9}}\text{,}\quad |\partial _{v}^{\beta }k_{1}(u,v)|\leq
C|v|^{\gamma -|\beta |}e^{-\frac{|u|^{2}+|v|^{2}}{9}}\text{,}
\end{equation*}%
\begin{equation*}
|\partial _{u}^{\beta _{1}}\partial _{v}^{\beta _{2}}k_{1}(u,v)|\leq
C|v|^{\gamma -|\beta _{2}|}e^{-\frac{|u|^{2}+|v|^{2}}{9}}\text{.}
\end{equation*}

\subsubsection{Estimate of $K_{2}$}

We first use Hilbert-Carleman coordinates to express the operator $K_2$ in terms of a kernel function $k_{2}(\xi ,\eta )$. Although this procedure is standard, we include the details here for the reader's convenience. With this representation, we then perform derivative estimates.

By conservation of energy during the collision
\begin{equation*}
\mathcal{M}(\xi )\mathcal{M}(\xi _{\ast })=\mathcal{M}(\xi ^{\prime })\mathcal{M}(\xi _{\ast }^{\prime })\text{,}
\end{equation*}%
one has
\begin{equation*}
K_{2}f=\int_{\mathbb{R}^{3}}\int_{\mathbf{S}^{2}}\mathcal{M}^{1/2}(\xi _{\ast })\left(
\mathcal{M}^{1/2}(\xi ^{\prime })f(\xi _{\ast }^{\prime })+\mathcal{M}^{1/2}(\xi _{\ast
}^{\prime })f(\xi ^{\prime })\right) B(|\xi _{\ast }-\xi |,\Omega )d\Omega
d\xi _{\ast }\text{.}
\end{equation*}

Let $V=\xi _{\ast }-\xi $. For fixed $\Omega \in \mathbf{S}^{2}$, write
\begin{equation*}
V=V_{\parallel }+V_{\perp }
\end{equation*}%
with
\begin{equation*}
V_{\parallel }=(\Omega \cdot V)\Omega \text{,}\quad V_{\perp
}=V-V_{\parallel }=(\Omega _{\perp }\cdot V)\Omega _{\perp }\text{,}
\end{equation*}%
where $\Omega _{\perp }$ is a unit vector perpendicular to $\Omega $. Hence
we have
\begin{equation*}
\xi ^{\prime }=\xi -\xi _{\ast }+\xi _{\ast }+(\Omega \cdot V)\Omega =\xi
_{\ast }-[V-(\Omega \cdot V)\Omega ]=\xi _{\ast }-(\Omega _{\perp }\cdot
V)\Omega _{\perp }\text{,}
\end{equation*}%
\begin{equation*}
\xi _{\ast }^{\prime }=\xi _{\ast }-\xi +\xi -(\Omega \cdot V)\Omega =\xi
+[V-(\Omega \cdot V)\Omega ]=\xi +(\Omega _{\perp }\cdot V)\Omega _{\perp }%
\text{.}
\end{equation*}%
Setting a coordinate frame with $V$ pointing to the north pole, $\Omega $
can be expressed as
\begin{equation*}
\Omega =\left( \sin \theta \cos \varphi ,\sin \theta \sin \varphi ,\cos
\theta \right) \text{,}
\end{equation*}%
where $\theta $ is the angle between $V$ and $\Omega $. Note that $B$
depends only on the $\lvert \cos \theta \rvert$. Then
\begin{equation*}
K_{2}f=2\int_{\mathbb{R}^{3}}\int_{0}^{\pi /2}\int_{0}^{2\pi }\mathcal{M}^{1/2}(\xi
_{\ast })\left( \mathcal{M}^{1/2}(\xi ^{\prime })f(\xi _{\ast }^{\prime })+\mathcal{M}^{1/2}(\xi
_{\ast }^{\prime })f(\xi ^{\prime })\right) B(|\xi _{\ast }-\xi |,\theta
)\sin \theta d\varphi d\theta d\xi _{\ast }\text{.}
\end{equation*}%
Under the transformation $\theta \rightarrow \pi /2-\theta $, $\varphi
\rightarrow \varphi +\pi $, $\Omega $ will change to $\Omega _{\perp }$, and
$\xi ^{\prime }\rightarrow \xi _{\ast }^{\prime }$, $\xi _{\ast }^{\prime
}\rightarrow \xi ^{\prime }$, so that the first term in $K_{2}f$ becomes
\begin{equation*}
2\int_{\mathbb{R}^{3}}\int_{0}^{\pi /2}\int_{0}^{2\pi }\mathcal{M}^{1/2}(\xi _{\ast
})\mathcal{M}^{1/2}(\xi _{\ast }^{\prime })f(\xi ^{\prime })B(|\xi _{\ast }-\xi |,%
\frac{\pi }{2}-\theta )\cos \theta d\varphi d\theta d\xi _{\ast }\text{.}
\end{equation*}%
Hence,
\begin{equation*}
K_{2}f=2\int_{\mathbb{R}^{3}}\int_{0}^{\pi /2}\int_{0}^{2\pi }\mathcal{M}^{1/2}(\xi
_{\ast })\mathcal{M}^{1/2}(\xi _{\ast }^{\prime })f(\xi ^{\prime })\left[ B(|\xi
_{\ast }-\xi |,\theta )+B(|\xi _{\ast }-\xi |,\frac{\pi }{2}-\theta )\frac{%
\cos \theta }{\sin \theta }\right] \sin \theta d\varphi d\theta d\xi _{\ast }%
\text{.}
\end{equation*}%
Define
\begin{equation*}
\mathcal{B}\left( |\xi _{\ast }-\xi |,\theta \right) =%
\begin{cases}
B(|\xi _{\ast }-\xi |,\theta )+B(|\xi _{\ast }-\xi |,\frac{\pi }{2}-\theta )%
\frac{\cos \theta }{\sin \theta }, & \theta \in (0,\pi /2)\text{,} \\
B(|\xi _{\ast }-\xi |,\pi -\theta )+B(|\xi _{\ast }-\xi |,\theta -\frac{\pi
}{2})|\frac{\cos \theta }{\sin \theta }|\text{,} & \theta \in (\pi /2,\pi )%
\text{.}%
\end{cases}%
\end{equation*}%
Then
\begin{equation*}
\mathcal{B}\left( |\xi _{\ast }-\xi |,\theta \right) =\mathcal{B}\left( |\xi
_{\ast }-\xi |,\pi -\theta \right) \text{,}\quad \theta \in (0,\pi )\text{,}
\end{equation*}%
and so the angle integral can be rewritten as a sphere integral
\begin{equation*}
K_{2}f=\int_{\mathbb{R}^{3}}\int_{\mathbf{S}^{2}}\mathcal{M}^{1/2}(\xi _{\ast
})\mathcal{M}^{1/2}(\xi _{\ast }^{\prime })f(\xi ^{\prime })\mathcal{B}\left( |\xi
_{\ast }-\xi |,\theta \right) d\Omega d\xi _{\ast }\text{.}
\end{equation*}%
Switching the integration order gives
\begin{align*}
K_{2}f& =\int_{\mathbf{S}^{2}}\int_{\mathbb{R}}\int_{V_{\perp }\perp \Omega
}\mathcal{M}^{1/2}(\xi _{\ast })\mathcal{M}^{1/2}(\xi _{\ast }^{\prime })f(\xi ^{\prime })%
\mathcal{B}\left( |\xi _{\ast }-\xi |,\theta \right) d^{2}V_{\perp }dzd\Omega
\\
& =2\int_{\mathbf{S}^{2}}\int_{0}^{+\infty }\int_{V_{\perp }\perp \Omega
}\mathcal{M}^{1/2}(\xi _{\ast })\mathcal{M}^{1/2}(\xi _{\ast }^{\prime })f(\xi ^{\prime })%
\mathcal{B}\left( |\xi _{\ast }-\xi |,\theta \right) d^{2}V_{\perp
}dzd\Omega \text{,}
\end{align*}%
where
\begin{equation*}
\xi _{\ast }-\xi =V=z\Omega +V_{\perp }\text{.}
\end{equation*}%
Combining $z$ and $\Omega $ together, $V_{\parallel }=z\Omega $ and we get a
three dimensional integral with volume element
\begin{equation*}
d^{3}V_{\parallel }=z^{2}dzd\Omega \text{,}
\end{equation*}%
and
\[
\xi _{\ast } =\xi +V=\xi +V_{\parallel }+V_{\perp },\quad \xi ^{\prime } =\xi +(\Omega \cdot V)\Omega =\xi +V_{\parallel },\quad \xi _{\ast }^{\prime } =\xi _{\ast }-(\Omega \cdot V)\Omega =\xi +V_{\perp }.
\]
Therefore,
\begin{align*}
K_{2}f& =2\int_{\mathbb{R}^{3}}\int_{V_{\perp }\perp V_{\parallel
}}\mathcal{M}^{1/2}(\xi _{\ast })\mathcal{M}^{1/2}(\xi _{\ast }^{\prime })f(\xi ^{\prime })%
\mathcal{B}\left( |\xi _{\ast }-\xi |,\theta \right) d^{2}V_{\perp }\frac{1}{%
|V_{\parallel }|^{2}}d^{3}V_{\parallel } \\
& =\int_{\mathbb{R}^{3}}\left[ \frac{2}{|V_{\parallel }|^{2}}\int_{V_{\perp
}\perp V_{\parallel }}\mathcal{M}^{1/2}(\xi +V_{\parallel }+V_{\perp })\mathcal{M}^{1/2}(\xi
+V_{\perp })\mathcal{B}\left( |V|,\theta \right) d^{2}V_{\perp }\right]
f(\xi +V_{\parallel })d^{3}V_{\parallel } \\
& =: \int_{\mathbb{R}^{3}}\tilde{k}_{2}(\xi ,V_{\parallel })f(\xi
+V_{\parallel })d^{3}V_{\parallel }  =\int_{\mathbb{R}^{3}}\tilde{k}_{2}(\xi ,\eta -\xi )f(\eta )d^{3}\eta\\
&
=:\int_{\mathbb{R}^{3}}k_{2}(\xi ,\eta )f(\eta )d^{3}\eta
\end{align*}%
with $\cos \theta =\frac{V_{\parallel }}{\sqrt{|V_{\parallel
}|^{2}+|V_{\perp }|^{2}}}\ $and
\begin{align}
\tilde{k}_{2}(\xi ,V_{\parallel })& =\frac{1}{\sqrt{2\pi ^{3}}|V_{\parallel
}|}\int_{V_{\perp }\perp V_{\parallel }}\exp \left\{ -\frac{1}{4}|\xi
+V_{\parallel }+V_{\perp }|^{2}-\frac{1}{4}|\xi +V_{\perp }|^{2}\right\} \notag \\
& \qquad (|V_{\parallel }|^{2}+|V_{\perp }|^{2})^{\frac{\gamma -1}{2}}\frac{%
\left( \beta (\theta )+\beta \left( \frac{\pi }{2}-\theta \right) \frac{%
\lvert \cos \theta \rvert}{|\sin \theta |}\right) }{\lvert \cos \theta \rvert}d^{2}V_{\perp }\text{%
.} \label{eq:k2tilde}
\end{align}%
In fact, we can further write it into a more compact form
\begin{equation*}
\tilde{k}_{2}(\xi ,V_{\parallel })=\frac{1}{\sqrt{2\pi ^{3}}|V_{\parallel
}|^{2-\gamma }}\int_{V_{\perp }\perp V_{\parallel }}e^{-\frac{1}{4}|\xi
+V_{\parallel }+V_{\perp }|^{2}-\frac{1}{4}|\xi +V_{\perp }|^{2}}\Theta (%
\frac{|V_{\perp }|}{|V_{\parallel }|})d^{2}V_{\perp }\text{,}
\end{equation*}%
with
\begin{equation*}
\Theta (\frac{|V_{\perp }|}{|V_{\parallel }|})=(1+\frac{|V_{\perp }|^{2}}{%
|V_{\parallel }|^{2}})^{\frac{\gamma -1}{2}}\frac{\beta (\arctan \frac{%
|V_{\perp }|}{|V_{\parallel }|})+\beta (\frac{\pi }{2}-\arctan \frac{%
|V_{\perp }|}{|V_{\parallel }|})(\frac{|V_{\perp }|}{|V_{\parallel }|})^{-1}%
}{\frac{1}{\sqrt{1+\frac{|V_{\perp }|^{2}}{|V_{\parallel }|^{2}}}}}\text{.}
\end{equation*}

\begin{remark}
Note that for hard sphere model, $\Theta(\frac{|V_{\perp}|}{|V_{\parallel}|}%
)=1$.
\end{remark}

The primary goal in this subsection is to give the high order derivatives
(both in $\xi ,V_{\parallel }$) estimates of $\tilde{k}_{2}(\xi
,V_{\parallel })$. The $\xi $-derivatives are easy to obtain, while the $%
V_{\parallel }$-derivatives is  more subtle.

Choose an orthonormal basis $\left\{ e_{1},e_{2},\frac{V_{\parallel }}{%
|V_{\parallel }|}\right\} $ for $\mathbb{R}^{3}$ which depends on $%
V_{\parallel }$. Then $V_{\perp }\in \mathrm{span}\left\{
e_{1},e_{2}\right\} $. At first glance, the integral plane depends on the
direction of $V_{\parallel }$. However, a close look at the integrand shows
the variation of the coordinate frame does not have contribution. In fact,
the integrand depends on $V_{\perp }$ through functions of the following
forms
\begin{equation*}
|V_{\perp }|\text{,}\quad \left\langle \xi ,V_{\perp }\right\rangle
=V_{\perp }\cdot \xi _{\perp }\text{,}
\end{equation*}%
where $\xi _{\perp }$ is the projection of $\xi $ to the component
perpendicular to $V_{\parallel }$. The inner product is preserved under the
variation of the direction of normal vector. Therefore,

\begin{equation*}
\partial _{\xi }^{\alpha }\partial _{V_{\parallel }^i}
\tilde{k}_{2}(\xi
,V_{\parallel })=\frac{1}{\sqrt{2\pi ^{3}}}\int_{V_{\perp }\perp
V_{\parallel }}\partial _{\xi }^{\alpha }\partial _{V_{\parallel }^i}\left[
\frac{e^{-\frac{1}{4}|\xi +V_{\parallel }+V_{\perp }|^{2}-\frac{1}{4}|\xi
+V_{\perp }|^{2}}}{|V_{\parallel }|^{2-\gamma }}\Theta (\frac{|V_{\perp }|}{%
|V_{\parallel }|})\right] d^{2}V_{\perp }\text{,}
\end{equation*}%
for multi-index $\alpha =\left( \alpha _{1},\alpha _{2},\alpha _{3}\right) $
and $i=1$, $2$, $3$. The derivative is the sum of the following terms:
\begin{align*}
& \partial _{V_{\parallel }^i}\partial _{\xi }^{\alpha }\left[ e^{-\frac{1%
}{4}|\xi +V_{\parallel }+V_{\perp }|^{2}-\frac{1}{4}|\xi +V_{\perp }|^{2}}%
\right] \Theta (\frac{|V_{\perp }|}{|V_{\parallel }|})|V_{\parallel
}|^{\gamma -2}\text{,} \\
& \partial _{\xi }^{\alpha }\left[ e^{-\frac{1}{4}|\xi +V_{\parallel
}+V_{\perp }|^{2}-\frac{1}{4}|\xi +V_{\perp }|^{2}}\right] \Theta (\frac{%
|V_{\perp }|}{|V_{\parallel }|})\partial _{V_{\parallel }^i}\left[
|V_{\parallel }|^{\gamma -2}\right] \text{,} \\
& \partial _{\xi }^{\alpha }\left[ e^{-\frac{1}{4}|\xi +V_{\parallel
}+V_{\perp }|^{2}-\frac{1}{4}|\xi +V_{\perp }|^{2}}\right] |V_{\parallel
}|^{\gamma -2}\partial _{V_{\parallel }^i}\left[ \Theta (\frac{|V_{\perp }|%
}{|V_{\parallel }|})\right] \text{.}
\end{align*}

\begin{remark}
For the hard sphere model, the third term vanishes.
\end{remark}

The first and second ones are easy to handle. For the first one, direct
inductions show that for any $0<q^{\prime }<1$,
\begin{equation*}
\left\vert \partial _{V_{\parallel }^i}\partial _{\xi }^{\alpha }\left[
e^{-\frac{1}{4}|\xi +V_{\parallel }+V_{\perp }|^{2}-\frac{1}{4}|\xi
+V_{\perp }|^{2}}\right] \Theta (\frac{|V_{\perp }|}{|V_{\parallel }|}%
)|V_{\parallel }|^{\gamma -2}\right\vert \leq Ce^{-\frac{q^{\prime }}{4}|\xi
+V_{\parallel }+V_{\perp }|^{2}-\frac{q^{\prime }}{4}|\xi +V_{\perp
}|^{2}}\Theta (\frac{|V_{\perp }|}{|V_{\parallel }|})|V_{\parallel
}|^{\gamma -2}\text{.}
\end{equation*}%
For the second one,
\begin{equation*}
\left\vert \partial _{\xi }^{\alpha }\left[ e^{-\frac{1}{4}|\xi
+V_{\parallel }+V_{\perp }|^{2}-\frac{1}{4}|\xi +V_{\perp }|^{2}}\right]
\Theta (\frac{|V_{\perp }|}{|V_{\parallel }|})\partial _{V_{\parallel }^i}%
\left[ |V_{\parallel }|^{\gamma -2}\right] \right\vert \leq Ce^{-\frac{%
q^{\prime }}{4}|\xi +V_{\parallel }+V_{\perp }|^{2}-\frac{q^{\prime }}{4}%
|\xi +V_{\perp }|^{2}}\Theta (\frac{|V_{\perp }|}{|V_{\parallel }|}%
)|V_{\parallel }|^{\gamma -3}\text{.}
\end{equation*}%
The third one is more subtle. Observing the following identity
\begin{equation*}
\nabla _{V_{\parallel }}\Theta (\frac{|V_{\perp }|}{|V_{\parallel }|}%
)=-\Theta ^{\prime }(\frac{|V_{\perp }|}{|V_{\parallel }|})\frac{|V_{\perp }|%
}{|V_{\parallel }|^{3}}V_{\parallel },\qquad \nabla _{V_{\perp }}\Theta (%
\frac{|V_{\perp }|}{|V_{\parallel }|})=\Theta ^{\prime }(\frac{|V_{\perp }|}{%
|V_{\parallel }|})\frac{V_{\perp }}{|V_{\parallel }||V_{\perp }|}\text{,}
\end{equation*}%
we thus have
\begin{equation*}
\partial _{V_{\parallel }^i}\Theta (\frac{|V_{\perp }|}{|V_{\parallel }|}%
)=-\left( V_{\perp }\cdot \nabla _{V_{\perp }}\right) \Theta (\frac{%
|V_{\perp }|}{|V_{\parallel }|})\frac{V_{\parallel }^{i}}{|V_{\parallel
}|^{2}}\text{.}
\end{equation*}%
Using integration by parts, we get
\begin{align*}
& \int_{V_{\perp }\perp V_{\parallel }}\partial _{\xi }^{\alpha }\left[ e^{-%
\frac{1}{4}|\xi +V_{\parallel }+V_{\perp }|^{2}-\frac{1}{4}|\xi +V_{\perp
}|^{2}}\right] |V_{\parallel }|^{\gamma -2}\partial _{V_{\parallel }^i}%
\left[ \Theta (\frac{|V_{\perp }|}{|V_{\parallel }|})\right] d^{2}V_{\perp }
\\
&= -\int_{V_{\perp }\perp V_{\parallel }}\partial _{\xi }^{\alpha }\left[
e^{-\frac{1}{4}|\xi +V_{\parallel }+V_{\perp }|^{2}-\frac{1}{4}|\xi
+V_{\perp }|^{2}}\right] |V_{\parallel }|^{\gamma -2}\left( V_{\perp }\cdot
\nabla _{V_{\perp }}\right) \Theta (\frac{|V_{\perp }|}{|V_{\parallel }|})%
\frac{V_{\parallel }^{i}}{|V_{\parallel }|^{2}}d^{2}V_{\perp } \\
&= 2\int_{V_{\perp }\perp V_{\parallel }}\partial _{\xi }^{\alpha }\left[
e^{-\frac{1}{4}|\xi +V_{\parallel }+V_{\perp }|^{2}-\frac{1}{4}|\xi
+V_{\perp }|^{2}}\right] |V_{\parallel }|^{\gamma -2}\Theta (\frac{|V_{\perp
}|}{|V_{\parallel }|})\frac{V_{\parallel }^{i}}{|V_{\parallel }|^{2}}%
d^{2}V_{\perp } \\
& \quad +\int_{V_{\perp }\perp V_{\parallel }}\partial _{\xi }^{\alpha }\left[
\left( V_{\perp }\cdot \nabla _{V_{\perp }}\right) e^{-\frac{1}{4}|\xi
+V_{\parallel }+V_{\perp }|^{2}-\frac{1}{4}|\xi +V_{\perp }|^{2}}\right]
|V_{\parallel }|^{\gamma -2}\Theta (\frac{|V_{\perp }|}{|V_{\parallel }|})%
\frac{V_{\parallel }^{i}}{|V_{\parallel }|^{2}}d^{2}V_{\perp }\text{.}
\end{align*}%
Note that%
\begin{equation*}
\left\vert \partial _{\xi }^{\alpha }\left[ e^{-\frac{1}{4}|\xi
+V_{\parallel }+V_{\perp }|^{2}-\frac{1}{4}|\xi +V_{\perp }|^{2}}\right]
|V_{\parallel }|^{\gamma -2}\Theta (\frac{|V_{\perp }|}{|V_{\parallel }|})%
\frac{V_{\parallel }^{i}}{|V_{\parallel }|^{2}}\right\vert \leq Ce^{-\frac{%
q^{\prime }}{4}|\xi +V_{\parallel }+V_{\perp }|^{2}-\frac{q^{\prime }}{4}%
|\xi +V_{\perp }|^{2}}\Theta (\frac{|V_{\perp }|}{|V_{\parallel }|}%
)|V_{\parallel }|^{\gamma -3}\text{,}
\end{equation*}%
and
\begin{align*}
&\left\vert \partial _{\xi }^{\alpha }\left[ \left( V_{\perp }\cdot \nabla
_{V_{\perp }}\right) e^{-\frac{1}{4}|\xi +V_{\parallel }+V_{\perp }|^{2}-%
\frac{1}{4}|\xi +V_{\perp }|^{2}}\right] \right\vert \\
&=\left\vert -\partial _{\xi }^{\alpha }\left[ \left( \xi +V_{\perp
}\right) \cdot V_{\perp }e^{-\frac{1}{4}|\xi +V_{\parallel }+V_{\perp }|^{2}-%
\frac{1}{4}|\xi +V_{\perp }|^{2}}\right] \right\vert \\
&=\left\vert -\left[ \left( \xi +V_{\perp }\right) \cdot V_{\perp }\right]
\partial _{\xi }^{\alpha }\left[ e^{-\frac{1}{4}|\xi +V_{\parallel
}+V_{\perp }|^{2}-\frac{1}{4}|\xi +V_{\perp }|^{2}}\right] -\sum_{\alpha
_{1}\leq \alpha ,|\alpha _{1}|=1}V_{\perp }^{\alpha _{1}}\partial _{\xi
}^{\alpha -\alpha _{1}}\left[ e^{-\frac{1}{4}|\xi +V_{\parallel }+V_{\perp
}|^{2}-\frac{1}{4}|\xi +V_{\perp }|^{2}}\right] \right\vert \\
&\leq C\left( 1+|\xi _{\perp }|\right) e^{-\frac{q^{\prime }}{4}|\xi
+V_{\parallel }+V_{\perp }|^{2}-\frac{q^{\prime }}{4}|\xi +V_{\perp }|^{2}}%
\text{,}
\end{align*}%
for any $0<q^{\prime }<1$. Combining all the above estimates, we arrive at
\begin{align*}
\left\vert \partial _{\xi }^{\alpha }\partial _{V_{\parallel }^i}\tilde{k}%
_{2}(\xi ,V_{\parallel })\right\vert & \leq C\int_{V_{\perp }\perp
V_{\parallel }}e^{-\frac{q^{\prime }}{4}|\xi +V_{\parallel }+V_{\perp }|^{2}-%
\frac{q^{\prime }}{4}|\xi +V_{\perp }|^{2}}\Theta (\frac{|V_{\perp }|}{%
|V_{\parallel }|})|V_{\parallel }|^{\gamma -2}d^{2}V_{\perp } \\
& \quad +C\int_{V_{\perp }\perp V_{\parallel }}e^{-\frac{q^{\prime }}{4}|\xi
+V_{\parallel }+V_{\perp }|^{2}-\frac{q^{\prime }}{4}|\xi +V_{\perp
}|^{2}}\Theta (\frac{|V_{\perp }|}{|V_{\parallel }|})|V_{\parallel
}|^{\gamma -3}d^{2}V_{\perp } \\
& \quad +C\left( 1+|\xi _{\perp }|\right) \int_{V_{\perp }\perp V_{\parallel
}}e^{-\frac{q^{\prime }}{4}|\xi +V_{\parallel }+V_{\perp }|^{2}-\frac{%
q^{\prime }}{4}|\xi +V_{\perp }|^{2}}\Theta (\frac{|V_{\perp }|}{%
|V_{\parallel }|})|V_{\parallel }|^{\gamma -3}d^{2}V_{\perp }\text{.}
\end{align*}%
Returning back to the  expression similar to \eqref{eq:k2tilde}, one has
\begin{equation*}
\left\vert \partial _{\xi }^{\alpha }\partial _{V_{\parallel }^i}\tilde{k}%
_{2}(\xi ,V_{\parallel })\right\vert \leq C\frac{1+|V_{\parallel }|+|\xi
_{\perp }|}{|V_{\parallel }|^{2}}\int_{V_{\perp }\perp V_{\parallel }}\frac{%
e^{-\frac{q^{\prime }}{4}|\xi +V_{\parallel }+V_{\perp }|^{2}-\frac{%
q^{\prime }}{4}|\xi +V_{\perp }|^{2}}}{(|V_{\parallel }|^{2}+|V_{\perp
}|^{2})^{\frac{1-\gamma }{2}}}\frac{\beta ^{\ast }(\theta )}{\lvert \cos \theta \rvert}%
d^{2}V_{\perp },
\end{equation*}
where $\beta^{\ast}(\theta)= \beta (\theta )+\beta \left( \frac{\pi }{2}-\theta \right) \frac{%
	\lvert \cos \theta \rvert}{|\sin \theta |}$.

In light of the following basic identity
\begin{equation*}
|\xi +V_{\parallel }+V_{\perp }|^{2}+|\xi +V_{\perp }|^{2}=2|V_{\perp }+\xi +%
\frac{V_{\parallel }}{2}|^{2}+\frac{1}{2}|V_{\parallel }|^{2}=2|V_{\perp
}+\xi _{\perp }|^{2}+2|\xi _{\parallel }+\frac{V_{\parallel }}{2}|^{2}+\frac{%
1}{2}|V_{\parallel }|^{2}\text{,}
\end{equation*}%
where
\begin{equation*}
\xi =\xi _{\parallel }+\xi _{\perp },\quad \xi _{\parallel }=(\xi \cdot
\frac{V_{\parallel }}{|V_{\parallel }|})\frac{V_{\parallel }}{|V_{\parallel
}|},\quad \xi _{\perp }=\xi -\xi _{\parallel }=\frac{V_{\parallel }\wedge
(\xi \wedge V_{\parallel })}{|V_{\parallel }|^{2}}\text{,}
\end{equation*}%
and the fact that $\frac{\beta ^{\ast }(\theta )}{\lvert \cos \theta \rvert}\leq C$,
one obtains
\begin{equation*}
\left\vert \partial _{\xi }^{\alpha }\partial _{V_{\parallel }^i}\tilde{k}%
_{2}(\xi ,V_{\parallel })\right\vert \leq C\frac{1+|V_{\parallel }|+|\xi
_{\perp }|}{|V_{\parallel }|^{2}}e^{-\frac{q^{\prime }}{2}|\xi _{\parallel }+%
\frac{V_{\parallel }}{2}|^{2}-\frac{q^{\prime }}{8}|V_{\parallel
}|^{2}}\int_{V_{\perp }\perp V_{\parallel }}\frac{e^{-\frac{q^{\prime }}{2}%
|V_{\perp }+\xi _{\perp }|^{2}}}{(|V_{\parallel }|^{2}+|V_{\perp }|^{2})^{%
\frac{1-\gamma }{2}}}d^{2}V_{\perp }\text{.}
\end{equation*}%
When there  is no $V_{\parallel }$-derivative, one easily gets
\begin{equation*}
\left\vert \partial _{\xi }^{\alpha }\tilde{k}_{2}(\xi ,V_{\parallel
})\right\vert \leq \frac{C}{|V_{\parallel }|}e^{-\frac{q^{\prime }}{2}|\xi
_{\parallel }+\frac{V_{\parallel }}{2}|^{2}-\frac{q^{\prime }}{8}%
|V_{\parallel }|^{2}}\int_{V_{\perp }\perp V_{\parallel }}\frac{e^{-\frac{%
q^{\prime }}{2}|V_{\perp }+\xi _{\perp }|^{2}}}{(|V_{\parallel
}|^{2}+|V_{\perp }|^{2})^{\frac{1-\gamma }{2}}}d^{2}V_{\perp }\text{.}
\end{equation*}%
In view of Proposition 5.2 in \cite{[Caflisch]}, we have%
\begin{align*}
\int_{V_{\perp }\perp V_{\parallel }}\frac{e^{-\frac{q^{\prime }}{2}%
|V_{\perp }+\xi _{\perp }|^{2}}}{(|V_{\parallel }|^{2}+|V_{\perp }|^{2})^{%
\frac{1-\gamma }{2}}}d^{2}V_{\perp } &=\int_{w\perp V_{\parallel }}\frac{%
e^{-\frac{q^{\prime }}{2}|w|^{2}}}{(|V_{\parallel }|^{2}+|w-\xi _{\perp
}|^{2})^{\frac{1-\gamma }{2}}}d^{2}w \\
&\leq C\frac{1}{\left( 1+|V_{\parallel }|+\left\vert \xi _{\perp
}\right\vert \right) ^{1-\gamma }}\text{.}
\end{align*}%
Therefore, the following estimates for $\tilde{k}_{2}(\xi ,V_{\parallel })$
hold:
\begin{lemma}
For any $\left\vert \alpha \right\vert \geq 0$, $0<q^{\prime }<1$,
\begin{align*}
\left\vert \partial _{\xi }^{\alpha }\tilde{k}_{2}(\xi ,V_{\parallel
})\right\vert & \leq \frac{C}{|V_{\parallel }|}\frac{e^{-\frac{q^{\prime }}{2%
}|\xi _{\parallel }+\frac{V_{\parallel }}{2}|^{2}-\frac{q^{\prime }}{8}%
|V_{\parallel }|^{2}}}{(1+|V_{\parallel }|+\left\vert \xi _{\perp
}\right\vert )^{1-\gamma }}\text{,}%
\vspace {3mm}
\\
\left\vert \partial _{\xi }^{\alpha }\partial _{V_{\parallel }^i}\tilde{k}%
_{2}(\xi ,V_{\parallel })\right\vert & \leq C\frac{1+|\xi _{\perp }|}{%
|V_{\parallel }|^{2}}\frac{e^{-\frac{q^{\prime }}{2}|\xi _{\parallel }+\frac{%
V_{\parallel }}{2}|^{2}-\frac{q^{\prime }}{8}|V_{\parallel }|^{2}}}{%
(1+|V_{\parallel }|+\left\vert \xi _{\perp }\right\vert )^{1-\gamma }}\text{.%
}
\end{align*}%
where%
\begin{equation*}
\xi _{\parallel }=\frac{(\xi \cdot V_{\parallel })V_{\parallel }}{%
|V_{\parallel }|^{2}},\quad \xi _{\perp }=\xi -\xi _{\parallel }\text{.}
\end{equation*}
\end{lemma}
Note that $V_{\parallel }=\eta -\xi $, $\xi _{\perp }=\frac{1}{2}\left( \eta
+\xi \right) _{\bot }$, $\xi _{\parallel }+\frac{V_{\parallel }}{2}=\frac{1}{%
2}\left( \eta +\xi \right) _{\Vert }$, and thus
\begin{align*}
1+|V_{\parallel }|+\left\vert \xi _{\perp }\right\vert +\frac{1}{2}%
\left\vert \left( \eta +\xi \right) _{\Vert }\right\vert &=1+\left\vert
\eta -\xi \right\vert +\frac{1}{2}\left\vert \left( \eta +\xi \right) _{\bot
}\right\vert +\frac{1}{2}\left\vert \left( \eta +\xi \right) _{\Vert
}\right\vert \\
&\geq \left( 1+\left\vert \eta -\xi \right\vert ^{2}+\frac{1}{4}\left\vert
\left( \eta +\xi \right) _{\bot }\right\vert ^{2}+\frac{1}{4}\left\vert
\left( \eta +\xi \right) _{\Vert }\right\vert ^{2}\right) ^{1/2} \\
&\geq \frac{1}{2}\left( 1+\left\vert \eta \right\vert ^{2}+\left\vert \xi
\right\vert ^{2}\right) ^{1/2}\geq c\left( 1+\left\vert \eta \right\vert
+\left\vert \xi \right\vert \right) \text{.}
\end{align*}%
Since $k_{2}(\xi ,\eta )=\tilde{k}_{2}(\xi ,\eta -\xi )$, one can rephrase
the above estimates for $k_{2}$ as follows.

\begin{lemma}
For any $|\alpha |\geq 0$, $0<q^{\prime }<1$,%
\begin{equation*}
\left\vert \left( \partial _{\xi }+\partial _{\eta }\right) ^{\alpha
}k_{2}(\xi ,\eta )\right\vert \leq \frac{C}{|\eta -\xi |}\frac{e^{-\frac{%
q^{\prime }}{8}\frac{\left\vert \left\vert \eta \right\vert ^{2}-\left\vert
\xi \right\vert ^{2}\right\vert ^{2}}{|\eta -\xi |^{2}}-\frac{q^{\prime }}{8}%
|\eta -\xi |^{2}}}{(1+|\xi |+|\eta |)^{1-\gamma }}\text{,}
\end{equation*}%
\begin{equation*}
|\partial _{\xi }\left( \partial _{\xi }+\partial _{\eta }\right) ^{\alpha
}k_{2}(\xi ,\eta )|,|\partial _{\eta }\left( \partial _{\xi }+\partial
_{\eta }\right) ^{\alpha }k_{2}(\xi ,\eta )|\leq \frac{1+|\xi _{\perp }|}{%
|\eta -\xi |^{2}}\frac{e^{-\frac{q^{\prime }}{8}\frac{\left\vert \left\vert
\eta \right\vert ^{2}-\left\vert \xi \right\vert ^{2}\right\vert ^{2}}{|\eta
-\xi |^{2}}-\frac{q^{\prime }}{8}|\eta -\xi |^{2}}}{(1+|\xi |+|\eta
|)^{1-\gamma }}\text{.}
\end{equation*}%
where $|\xi _{\perp }|=\left\vert \xi \wedge \left( \eta -\xi \right)
\right\vert /\left\vert \eta -\xi \right\vert $.
\end{lemma}

Combining the estimates for the operator $K_{1}$ in Section \ref{K1-Estimate}
and following the similar argument as that in \cite{[Caflisch]}, we have the
following estimates for $K=-K_{1}+K_{2}$.

\begin{proposition}
\label{prop: Derivatives of kernel k}The integral operator $K$ can be
expressed as%
\begin{equation*}
Kf=\int_{\mathbb{R}^{3}}k\left( \xi ,\eta \right) f\left( \eta \right) d\eta
\text{,}
\end{equation*}%
where the kernel $k\left( \xi ,\eta \right) $ satisfies
\begin{equation*}
\left\vert k\left( \xi ,\eta \right) \right\vert \leq \frac{C}{|\eta -\xi |}%
\frac{e^{-\frac{q^{\prime }}{8}\frac{\left\vert \left\vert \eta \right\vert
^{2}-\left\vert \xi \right\vert ^{2}\right\vert ^{2}}{|\eta -\xi |^{2}}-%
\frac{q^{\prime }}{8}|\eta -\xi |^{2}}}{(1+|\xi |+|\eta |)^{1-\gamma }}\text{%
,}
\end{equation*}%
for any $0<q^{\prime }<1$. Moreover, for any $\left\vert \alpha \right\vert
\geq 0$, $0<q^{\prime }<1$,%
\begin{equation*}
\left\vert \left( \partial _{\xi }+\partial _{\eta }\right) ^{\alpha }k(\xi
,\eta )\right\vert \leq \frac{C}{|\eta -\xi |}\frac{e^{-\frac{q^{\prime }}{8}%
\frac{\left\vert \left\vert \eta \right\vert ^{2}-\left\vert \xi \right\vert
^{2}\right\vert ^{2}}{|\eta -\xi |^{2}}-\frac{q^{\prime }}{8}|\eta -\xi
|^{2}}}{(1+|\xi |+|\eta |)^{1-\gamma }}=:H^{\left( 0\right) }\left( \xi
,\eta \right)
\end{equation*}%
\begin{align*}
|\partial _{\xi }\left( \partial _{\xi }+\partial _{\eta }\right) ^{\alpha
}k(\xi ,\eta )|\text{, }|\partial _{\eta }\left( \partial _{\xi }+\partial
_{\eta }\right) ^{\alpha }k(\xi ,\eta )| &\leq \frac{1+|\xi _{\perp }|}{%
|\eta -\xi |^{2}}\frac{e^{-\frac{q^{\prime }}{8}\frac{\left\vert \left\vert
\eta \right\vert ^{2}-\left\vert \xi \right\vert ^{2}\right\vert ^{2}}{|\eta
-\xi |^{2}}-\frac{q^{\prime }}{8}|\eta -\xi |^{2}}}{(1+|\xi |+|\eta
|)^{1-\gamma }} \\
&=:H^{(1)}\left( \xi ,\eta \right) \text{.}
\end{align*}%
where $|\xi _{\perp }|=\left\vert \xi \wedge \left( \eta -\xi \right)
\right\vert /\left\vert \eta -\xi \right\vert $.
\end{proposition}

The following lemma is helpful for the proof of the Enhanced Mixture
Lemma:
\begin{lemma}[Proposition 5.3, \cite{[Caflisch]}]
\label{lem-Caflisch}For any $\varsigma >-3$ and any $a>0$, $b>0$, there is a
constant $c>0$ (depending on $\varsigma $, $a$, $b$) so that%
\begin{equation*}
\int_{\mathbb{R}^{3}}\left\vert \xi -\eta \right\vert ^{\varsigma }\exp
\left\{-a\left\vert \xi -\eta \right\vert ^{2}-b\frac{\left\vert \left\vert
\eta \right\vert ^{2}-\left\vert \xi \right\vert ^{2}\right\vert ^{2}}{|\eta
-\xi |^{2}}\right\}d\eta \leq c\left( 1+\left\vert \xi \right\vert \right)
^{-1}
\end{equation*}%
for any $\xi \in \mathbb{R}^{3}$.
\end{lemma}

\section{Linearized Problem} \label{sec:linear}

We write the linearized Boltzmann equation as

\begin{equation*}
	\partial _{t}h+\xi \cdot \nabla _{x}h+\nu (\xi )h=Kh,\quad \left.
	h\right\vert _{t=0}=h_{0}\text{.}
\end{equation*}%
A Picard-type iteration for extracting the initial singularity is designed by
\begin{equation*}
	\partial _{t}h^{(0)}+\xi \cdot \nabla _{x}h^{(0)}+\nu (\xi )h^{(0)}=0,\qquad
	h^{(0)}\big|_{t=0}=h_{0},
\end{equation*}%
\begin{equation*}
	\partial _{t}h^{(m)}+\xi \cdot \nabla _{x}h^{(m)}+\nu (\xi
	)h^{(m)}=Kh^{(m-1)},\qquad h^{(m)}\big|_{t=0}=0\quad \mbox{ for }m\geq 1.
\end{equation*}%
The solution can be written down explicitly as
\begin{equation*}
	h^{(0)}(t,x,\xi )=\mathbb{S}^{t}h_{0}\left( x,\xi \right) ,
\end{equation*}%
\begin{equation*}
	h^{(1)}\left( t,x,\xi \right) =\int_{0}^{t}\mathbb{S}^{t-t_{1}}Kh^{(0)}%
	\left( t_{1}\right) dt_{1}=\int_{0}^{t}\mathbb{S}^{t-t_{1}}K\mathbb{S}%
	^{t_{1}}h_{0}dt_{1},
\end{equation*}

\begin{equation*}
	h^{(m)}(t,x,\xi )=\int_{0}^{t}\int_{0}^{t_{1}}\cdots \int_{0}^{t_{m-1}}%
	\mathbb{S}^{t-t_{1}}K\mathbb{S}^{t_{1}-t_{2}}K\cdots \mathbb{S}%
	^{t_{m-1}-t_{m}}K\mathbb{S}^{t_{m}}h_{0}dt_{1}dt_{2}\cdots dt_{m}
\end{equation*}

We introduce the following mixture operator

\begin{align}
	\mathbf{M}_{N}^{t}h_{0}(x,\xi )& :=\left( \int_{0}^{t}\int_{0}^{t_{1}}\cdots
	\int_{0}^{t_{N-1}}\right) \mathbb{S}^{t-t_{1}}K\mathbb{S}^{t_{1}-t_{2}}K%
	\cdots \mathbb{S}^{t_{N-1}-t_{N}}K\mathbb{S}^{t_{N}}h_{0}dt_{1}dt_{2}\cdots
	dt_{N}  \notag \\
	& =\int_{T_{N}}\int_{\Xi _{N}}e^{-\sum_{j=0}^{N}\nu (\xi
		_{j})(t_{j}-t_{j+1})}\left[ \prod_{j=0}^{N-1}k(\xi _{j},\xi _{j+1})\right]
	h_{0}(x-\sum_{i=0}^{N}\xi _{i}(t_{i}-t_{i+1}),\xi _{N})\text{,}\notag
\end{align}%
for $N\geq 1$, and $\mathbf{M}_{0}^{t}h_{0}(x,\xi ):=\mathbb{S}%
^{t}h_{0}\left( x,\xi \right) $, where $\Xi _{N}=\mathbb{R}^{3N}$, $t_{N+1}=0$. Then in
terms of the mixture operator,
\begin{equation*}
	h^{(m)}\left( t,x,\xi \right) =\mathbf{M}_{m}^{t}h_{0}\text{, }\quad m\geq 0%
	\text{.}
\end{equation*}

The solution to the linearized Boltzmann equation is formally written
\begin{equation*}
	h\left( t,x,\xi \right) =\sum_{k=0}^{\infty }h^{(k)}\left( t,x,\xi \right)
	=\left( \sum_{k=0}^{\infty }\mathbf{M}_{k}^{t}\right) h_{0}\text{.}
\end{equation*}%
However, the convergence is not justified. In what follows, we will truncate
the summation at some $N$, and define the singular part to be first several
terms, which contain the most singular terms. The remainder is called a
regular part, which satisfies the linearized Boltzmann equation with source
term (given by $Kh^{(N)}$). Specifically, we define
\begin{equation*}  
	\mathcal{W}^{\left( N\right) }=\sum_{k=0}^{N}h^{\left( k\right) }\text{,}%
	\qquad \mathcal{R}^{\left( N\right) }=h-\mathcal{W}^{\left( N\right) }\text{.%
	}
\end{equation*}%
We call $\mathcal{W}^{\left( N\right) }$ the singular part and $\mathcal{R}%
^{\left( N\right) }$ the remainder, with $\mathcal{R}^{\left( N\right) }$ solving
the linearized Boltzmann equation with source $Kh^{(N)}$,
\begin{equation*}
	\left\{
	\begin{array}{l}
		\partial _{t}\mathcal{R}^{\left( N\right) }+\xi \cdot \nabla _{x}\mathcal{R}%
		^{\left( N\right) }-L\mathcal{R}^{\left( N\right) }=Kh^{(N)}\text{,}%
		\vspace {3mm}
		\\
		\left. \mathcal{R}^{\left( N\right) }\right\vert _{t=0}=0\text{.}%
	\end{array}%
	\right.
\end{equation*}%
The number $N$ will be determined later.

In the next subsection, we will prove after sufficient mixture, namely, as $%
N $ becomes sufficiently large, $h^{(N)}$ will become as good as we please.
Therefore, we are able to apply many tools such as spectral analysis, higher
order energy estimate, weighted energy estimates to analyze the remainder.

\subsection{Singularity extraction and Enhanced Mixture Lemma}

In this subsection, we show $h^{(N)}$ will become as good as we please. The
main result is stated in Lemma \ref{Enhanced Mixture Lemma}. In the
following we shall discuss the weighted $L_{\xi }^{2}L_{x}^{2}$-estimate and
the pointwise structure for $h^{\left( 0\right) }$ and $h^{\left( N\right) }$
$\left( N\geq 1\right) $ in order.

\begin{lemma}
	Let $Q\geq 0$. Then
	\begin{equation*}
		\left\Vert \left( 1+\left\vert \xi \right\vert \right) ^{Q}\mathbf{M}%
		_{0}^{t}h_{0}(x,\xi )\right\Vert _{L_{\xi }^{2}L_{x}^{2}}\leq e^{-\nu
			_{0}t}\left\Vert \left( 1+\left\vert \xi \right\vert \right)
		^{Q}h_{0}\right\Vert _{L_{\xi }^{2}L_{x}^{2}}\text{, }\forall Q\geq 0\text{.}
	\end{equation*}%
	Furthermore, if $h_{0}\left( x,\xi \right) =0$ for $\left\vert x\right\vert
	>1$, $\xi \in \mathbb{R}^{3}$, then%
	\begin{equation*}
		\left\vert \mathbf{M}_{0}^{t}h_{0}(x,\xi )\right\vert _{L_{\xi }^{\infty
		}}\leq Ce^{-\frac{\nu \left( \xi \right) t}{2}}\left( 1+\left\vert
		x\right\vert \right) ^{-M}\left( 1+\left\vert \xi \right\vert \right)
		^{-P}\left\Vert h_{0}\right\Vert _{L_{\xi ,Q}^{\infty }L_{x}^{\infty }}
	\end{equation*}%
	where $M,P$ are non-negative with $M\left( 1-\gamma \right) +P=Q$.
\end{lemma}

\begin{proof}
By definition , $\mathbf{M}_{0}^{t}h_{0}(x,\xi ):=\mathbb{S}^{t}h_{0}\left(
x,\xi \right) $. The first part is readily obtained.

Now assume that $h_{0}\left( x,\xi \right) =0$ for $\left\vert x\right\vert
>1$, $\xi \in \mathbb{R}^{3}$. Consider the case $\left\vert x\right\vert
\geq 2$. Letting $y=x-\xi t$ gives%
\begin{align*}
	&\left\vert \mathbf{M}_{0}^{t}h_{0}(x,\xi )\right\vert \\
	&\leq e^{-\frac{\nu \left( \xi \right) t}{2}}e^{-\frac{\nu _{0}}{2}\left(
		1+\left\vert \xi \right\vert \right) ^{\gamma }t}\left( 1+\left\vert \xi
	\right\vert \right) ^{-(Q-P)}\left( 1+\left\vert \xi \right\vert \right)
	^{-P}\left\vert \left( 1+\left\vert \xi \right\vert \right) ^{Q}h_{0}\left(
	x-\xi t,\xi \right) \right\vert \\
	&=e^{-\frac{\nu \left( \xi \right) t}{2}}e^{-\frac{\nu _{0}}{2}\left(
		t+\left\vert x-y\right\vert \right) ^{\gamma }t^{1-\gamma }}\left(
	t+\left\vert x-y\right\vert \right) ^{-(Q-P)}t^{Q-P}\left( 1+\left\vert \xi
	\right\vert \right) ^{-P}\left\vert \left( 1+\left\vert \xi \right\vert
	\right) ^{Q}h_{0}\left( y,\xi \right) \right\vert \\
	&\leq e^{-\frac{\nu \left( \xi \right) t}{2}}\sup_{\left\vert y\right\vert
		\leq 1}e^{-\frac{\nu _{0}}{2}\left\vert x-y\right\vert ^{\gamma }t^{1-\gamma
	}}\left( \left\vert x-y\right\vert ^{\gamma }t^{1-\gamma }\right) ^{\frac{Q-P%
		}{1-\gamma }}\left\vert x-y\right\vert ^{-\frac{Q-P}{1-\gamma }}\left(
	1+\left\vert \xi \right\vert \right) ^{-P}\left\vert \left( 1+\left\vert \xi
	\right\vert \right) ^{Q}h_{0}\left( y,\xi \right) \right\vert \\
	&\leq Ce^{-\frac{\nu \left( \xi \right) t}{2}}(1+|x|)^{-%
		\frac{Q-P}{1-\gamma }}\left( 1+\left\vert \xi \right\vert \right)
	^{-P}\left\Vert h_{0}\right\Vert _{L_{\xi ,Q}^{\infty }L_{x}^{\infty }} \\
	&=Ce^{-\frac{\nu \left( \xi \right) t}{2}}(1+|x|)^{-M}\left( 1+\left\vert \xi \right\vert \right) ^{-P}\left\Vert
	h_{0}\right\Vert _{L_{\xi ,Q}^{\infty }L_{x}^{\infty }}
\end{align*}%
where $M=\frac{Q-P}{1-\gamma }$, $M$, $P\geq 0$ with $M\left( 1-\gamma
\right) +P=Q$. In the case $\left\vert x\right\vert \leq 2$, it is easy to
see that%
\begin{equation*}
	\left\vert \mathbb{S}^{t}h_{0}\left( x,\xi \right) \right\vert \leq
	3^{M}e^{-\nu \left( \xi \right) t}\left( 1+\left\vert x\right\vert \right)
	^{-M}\left( 1+\left\vert \xi \right\vert \right) ^{-P}\left\Vert
	h_{0}\right\Vert _{L_{\xi ,Q}^{\infty }L_{x}^{\infty }}
\end{equation*}%
for any $M\geq 0$, and $0\leq P\leq Q$. This completes the proof.
\end{proof}

\begin{lemma}[Enhanced Mixture Lemma]
	\label{Enhanced Mixture Lemma}$%
	\begin{array}{c}
	\end{array}%
	$
	
	\begin{enumerate}
		\item For any $\alpha ,\beta \in \mathbb{N}_{0}^{3}$, any $P\geq 0$, when
		\begin{equation*}
			N\geq \max \left\{ P+\frac{|\alpha |+|\beta |+|\alpha |\gamma }{2},3\left(
			|\alpha |+|\beta |\right) \right\} \text{,}
		\end{equation*}%
		we have
		\begin{equation*}
			\left\Vert \left( 1+|\xi |\right) ^{P}\partial _{x}^{\alpha }\partial _{\xi
			}^{\beta }\mathbf{M}_{N}^{t}h_{0}\right\Vert _{L_{\xi }^{2}L_{x}^{2}}\leq
			Ce^{-ct}\left\Vert h_{0}\right\Vert _{L_{\xi }^{2}L_{x}^{2}}\text{.}
		\end{equation*}
		
		\item Suppose $h_{0}\left( x,\xi \right) =0$ for $|x|>1$, we have for any $%
		\alpha ,\beta \in \mathbb{N}_{0}^{3}$, any $M$, $P\geq 0$, any $Q\in \mathbb{R}$,
		when
		\begin{equation*}
			N\geq \max \left\{ \frac{P+\left( |\alpha |+|\beta |+|\alpha |\gamma \right)
				+M\left( 1-\gamma \right) -Q}{2}\text{, }3\left( |\alpha |+|\beta |\right)
			\right\} \text{,}
		\end{equation*}%
		we have
		\begin{equation*}
			\left\vert \partial _{x}^{\alpha }\partial _{\xi }^{\beta }\mathbf{M}%
			_{N}^{t}h_{0}\left( t,x,\xi \right) \right\vert \leq Ce^{-c_{0}\nu (\xi
				)t}\left( 1+|x|\right) ^{-M}\left( 1+|\xi |\right) ^{-P}\left\Vert \left(
			1+|\xi |\right) ^{Q}h_{0}\right\Vert _{L_{x,\xi }^{\infty }}\text{.}
		\end{equation*}
	\end{enumerate}
\end{lemma}

\begin{proof}
By definition of $\mathbf{M}_{N}^{t}$,
\begin{align*}
	\mathbf{M}_{N}^{t}h_{0}(x,\xi )& :=\left( \int_{0}^{t}\int_{0}^{t_{1}}\cdots
	\int_{0}^{t_{N-1}}\right) \mathbb{S}^{t-t_{1}}K\mathbb{S}^{t_{1}-t_{2}}K%
	\cdots \mathbb{S}^{t_{N-1}-t_{N}}K\mathbb{S}^{t_{N}}h_{0}dt_{1}\cdots dt_{N}
	\\
	& =\int_{T_{N}}\int_{\Xi _{N}}e^{-\sum_{j=0}^{N}\nu (\xi
		_{j})(t_{j}-t_{j+1})}\left[ \prod_{j=0}^{N-1}k(\xi _{j},\xi _{j+1})\right]
	h_{0}(x-\sum_{i=0}^{N}\xi _{i}(t_{i}-t_{i+1}),\xi _{N})\text{.}
\end{align*}%
Let $\ell =|\beta |$, $m=|\alpha |+|\beta |$. We have

\begin{align*}
	& \partial _{x}^{\alpha }\partial _{\xi }^{\beta }\mathbf{M}%
	_{N}^{t}h_{0}(x,\xi ) \\
	=& \sum_{\beta _{1}+\beta _{2}+\beta _{3}=\beta }C\int_{T_{N}}\int_{\Xi
		_{N}}e^{-\sum_{j=1}^{N}\nu (\xi _{j})(t_{j}-t_{j+1})}\prod_{j=1}^{N-1}k(\xi
	_{j},\xi _{j+1}) \\
	& \partial _{\xi }^{\beta _{1}}\left[ e^{-\nu (\xi )(t-t_{1})}\right]
	\partial _{\xi }^{\beta _{2}}\left[ k\left( \xi ,\xi _{1}\right) \right]
	\partial _{\xi }^{\beta _{3}}\partial _{x}^{\alpha
	}h_{0}(x-\sum_{i=0}^{N}\xi _{i}(t_{i}-t_{i+1}),\xi _{N}).
\end{align*}%
From (\ref{nu-deriv}), $\partial _{\xi }^{\beta _{1}}\left[ e^{-\nu (\xi
	)(t-t_{1})}\right] $ satisfies similar estimates as the one without
differentiation. We only consider two most interesting terms,

\begin{align*}
	\left( I\right) & =\int_{T_{N}}\int_{\Xi _{N}}e^{-\sum_{j=0}^{N}\nu (\xi
		_{j})(t_{j}-t_{j+1})}\partial _{\xi }^{\beta }\left[ k\left( \xi ,\xi
	_{1}\right) \right]  \prod_{j=1}^{N-1}k(\xi _{j},\xi _{j+1})\partial _{x}^{\alpha
	}h_{0}(x-\sum_{i=0}^{N}\xi _{i}(t_{i}-t_{i+1}),\xi _{N})\text{,}
\end{align*}%
and
\begin{equation*}
	\left( II\right) =\int_{T_{N}}\int_{\Xi _{N}}e^{-\sum_{j=0}^{N}\nu (\xi
		_{j})(t_{j}-t_{j+1})}\left[ \prod_{j=0}^{N-1}k(\xi _{j},\xi _{j+1})\right]
	\partial _{\xi }^{\beta }\partial _{x}^{\alpha }h_{0}(x-\sum_{i=0}^{N}\xi
	_{i}(t_{i}-t_{i+1}),\xi _{N})\text{.}
\end{equation*}

Look at $\left( I\right) $ first. For $\partial _{\xi }^{\beta }\left[
k\left( \xi ,\xi _{1}\right) \right] $, we write
\begin{align*}
	\partial _{\xi }^{\beta }\left[ k\left( \xi ,\xi _{1}\right) \right] &
	=\left( \partial _{\xi }+\partial _{\xi _{1}}-\partial _{\xi _{1}}\right)
	^{\beta }k\left( \xi ,\xi _{1}\right) \\
	& =\sum_{\beta ^{\prime }+\beta ^{\prime \prime }=\beta }C\left( \partial
	_{\xi }+\partial _{\xi _{1}}\right) ^{\beta ^{\prime \prime }}\partial _{\xi
		_{1}}^{\beta ^{\prime }}k\left( \xi ,\xi _{1}\right) .
\end{align*}%
In view of Proposition \ref{prop: Derivatives of kernel k}, it suffices to
consider the term $\partial _{\xi _{1}}^{\beta }k\left( \xi ,\xi _{1}\right)
$ among all the terms in the summation. One keeps only one $\partial _{\xi
	_{1}}$ on $k\left( \xi ,\xi _{1}\right) $, and uses integration by parts to
transfer the other $\xi _{1}$-derivatives to the rest terms, then the
integrand looks like
\begin{align*}
	& e^{-\sum_{j=0}^{N}\nu (\xi _{j})(t_{j}-t_{j+1})}\partial _{\xi
		_{1}}^{\varsigma _{1}}k\left( \xi ,\xi _{1}\right) \partial _{\xi
		_{1}}^{\beta _{1}}k\left( \xi _{1},\xi _{2}\right) \\
	& \cdot \left[ \prod_{j=2}^{N-1}k(\xi _{j},\xi _{j+1})\right] \partial _{\xi
		_{1}}^{\beta _{2}}\partial _{x}^{\alpha }h_{0}(x-\sum_{i=0}^{N}\xi
	_{i}(t_{i}-t_{i+1}),\xi _{N})
\end{align*}%
where $\beta _{1}+\beta _{2}+\varsigma _{1}=\beta $ and $|\varsigma _{1}|=1$%
. If $|\beta _{1}|>1$, one can repeat this argument, namely, replaces $%
\partial _{\xi _{1}}^{\beta _{1}}k\left( \xi _{1},\xi _{2}\right) $ by $%
\left( \partial _{\xi _{1}}+\partial _{\xi _{2}}-\partial _{\xi _{2}}\right)
^{\beta _{1}}k\left( \xi _{1},\xi _{2}\right) $, expands the binomial, keeps
only one $\xi _{2}$-derivative on $k\left( \xi _{1},\xi _{2}\right) $ and
uses integration by parts again to transfer other $\xi _{2}$-derivatives to
the rest terms. Continuing this process for $|\beta |$ times, we distribute
the $|\beta |$-order $\xi $-derivatives evenly on $\xi _{j}$,$\ j=1$,$\ldots
$, $|\beta |$, and eventually obtain the integrands having the following
forms (neglecting other easier terms)
\begin{align*}
	& e^{-\sum_{j=0}^{N}\nu (\xi _{j})(t_{j}-t_{j+1})}\left[ \prod_{j=1}^{\ell
	}\partial _{\xi _{j}}^{\varsigma _{j}^{\prime }}k\left( \xi _{j-1},\xi
	_{j}\right) \right] \left[ \prod_{j=\ell +1}^{N}k\left( \xi _{j-1},\xi
	_{j}\right) \right] \\
	\cdot & \left[ \prod_{j=1}^{\ell }\partial _{\xi _{j}}^{\varsigma
		_{j}^{\prime \prime }}\right] \left\{ \partial _{x}^{\alpha
	}h_{0}(x-\sum_{i=0}^{N}\xi _{i}(t_{i}-t_{i+1}),\xi _{N})\right\} \text{,}
\end{align*}%
with $\varsigma _{j}^{\prime }+\varsigma _{j}^{\prime \prime }=\varsigma
_{j} $ and $|\varsigma _{j}|=1$. Without loss of generality, we only
illustrate how to handle the following two terms,
\begin{align*}
	\left( I_{1}\right) & =\int_{T_{N}}\int_{\Xi _{N}}e^{-\sum_{j=0}^{N}\nu (\xi
		_{j})(t_{j}-t_{j+1})}\left[ \prod_{j=1}^{\ell }\partial _{\xi
		_{j}}^{\varsigma _{j}}k\left( \xi _{j-1},\xi _{j}\right) \right] \left[
	\prod_{j=\ell +1}^{N}k\left( \xi _{j-1},\xi _{j}\right) \right] \\
	& \quad \times \partial _{x}^{\alpha }h_{0}(x-\sum_{i=0}^{N}\xi
	_{i}(t_{i}-t_{i+1}),\xi _{N})\text{,}
\end{align*}%
and
\begin{align*}
	\left( I_{2}\right) & =\int_{T_{N}}\int_{\Xi _{N}}e^{-\sum_{j=0}^{N}\nu (\xi
		_{j})(t_{j}-t_{j+1})}\left[ \prod_{j=0}^{N-1}k\left( \xi _{j},\xi
	_{j+1}\right) \right] \\
	& \quad \times \left[ \prod_{j=1}^{\ell }\partial _{\xi _{j}}^{\varsigma
		_{j}}\right] \partial _{x}^{\alpha }h_{0}(x-\sum_{i=0}^{N}\xi
	_{i}(t_{i}-t_{i+1}),\xi _{N})\text{.}
\end{align*}%
Note the latter one is similar to $\left( II\right) $.

Let us focus on $\left( I_{2}\right) $ and the estimates of other terms can
be easily adapted from this. Our strategy is to express the $x$-derivatives
of $h_{0}\left( x-\sum_{i=0}^{N}\xi _{i}(t_{i}-t_{i+1}),\xi _{N}\right) $ in
terms of $\xi _{j}$-derivatives, use integration by parts to assign them to
kernel functions $k\left( \xi _{j},\xi _{j+1}\right) $, and absorb them by
the good properties of $k\left( \xi _{j},\xi _{j+1}\right) $.

Note
\begin{align*}
	& \left[ \prod_{j=1}^{\ell }\partial _{\xi _{j}}^{\varsigma _{j}}\right]
	\left\{ \partial _{x}^{\alpha }h_{0}(x-\sum_{i=0}^{N}\xi
	_{i}(t_{i}-t_{i+1}),\xi _{N})\right\} \\
	& =\left( -1\right) ^{\ell }\left[ \prod_{j=1}^{\ell }(t_{j}-t_{j+1})\right]
	\partial _{x}^{\alpha +\beta }h_{0}\bigl( x-\sum_{i=0}^{N}\xi
	_{i}(t_{i}-t_{i+1}),\xi _{N}\bigr) \text{.}
\end{align*}%
Further, noticing that
\begin{equation*}
	\partial _{\xi _{j}}h_{0}\bigl( x-\sum_{i=0}^{N}\xi _{i}(t_{i}-t_{i+1}),\xi
	_{N}\bigr) =-\left( t_{j}-t_{j+1}\right) \partial _{x}h_{0}\bigl(
	x-\sum_{i=0}^{N}\xi _{i}(t_{i}-t_{i+1}),\xi _{N}\bigr) \text{,}\quad j=1%
	\text{,}\dots \text{, }N-1\text{,}
\end{equation*}%
it gives
\begin{equation*}
	\partial _{x}h_{0}\bigl( x-\sum_{i=0}^{N}\xi _{i}(t_{i}-t_{i+1}),\xi
	_{N}\bigr) =-\frac{\partial _{\xi _{j}}+\partial _{\xi _{j-1}}}{%
		t_{j-1}-t_{j+1}}h_{0}\bigl( x-\sum_{i=0}^{N}\xi _{i}(t_{i}-t_{i+1}),\xi
	_{N}\bigr) \text{,}\ j=2\text{,}\dots \text{, }N-1\text{.}
\end{equation*}%
Denoting $\alpha +\beta =\sigma $, $m=|\sigma |$, we decompose $\sigma
=\sum_{j=0}^{m-1}\sigma _{j}$ with each $\sigma _{j}\in \mathbb{N}_{0}^{3}$
and $|\sigma _{j}|=1$, and write
\begin{align*}
	& \partial _{x}^{\alpha +\beta }h_{0}\bigl( x-\sum_{i=0}^{N}\xi
	_{i}(t_{i}-t_{i+1}),\xi _{N}\bigr) \\
	& =(-1)^{m}\left( \frac{\partial _{\xi _{N-1}}+\partial _{\xi _{N-2}}}{%
		t_{N-2}-t_{N}}\right) ^{\sigma _{0}}\left( \frac{\partial _{\xi
			_{N-4}}+\partial _{\xi _{N-5}}}{t_{N-5}-t_{N-3}}\right) ^{\sigma _{1}}\cdots
	\left( \frac{\partial _{N-3(m-1)-1}+\partial _{N-3(m-1)-2}}{%
		t_{N-3(m-1)-2}-t_{N-3(m-1)}}\right) ^{\sigma _{m-1}}h_{0}\text{.}
\end{align*}%
It therefore follows from integration by parts that
\begin{align*}
	\left( I_{2}\right) & =\int_{T_{N}}\int_{\Xi _{N}}\Biggl\{
	e^{-\sum_{j=0}^{N}\nu (\xi _{j})(t_{j}-t_{j+1})}\left[ \prod_{j=0}^{N-1}k%
	\left( \xi _{j},\xi _{j+1}\right) \right] \frac{\left( -1\right) ^{|\alpha
			|}\prod_{j=1}^{\ell }(t_{j}-t_{j+1})}{\prod_{j=0}^{m-1}\left(
		t_{N-3j-2}-t_{N-3j}\right) } \\
	& \quad  \cdot \prod_{j=0}^{m-1}\left( \partial _{\xi
		_{N-3j-1}}+\partial _{\xi _{N-3j-2}}\right) ^{\sigma _{j}}h_{0}\bigl(
	x-\sum_{i=0}^{N}\xi _{i}(t_{i}-t_{i+1}),\xi _{N}\bigr) \Biggr\} \\
	& =\int_{T_{N}}\int_{\Xi _{N}}\Biggl\{ \frac{\left( -1\right) ^{|\alpha
			|}\prod_{j=1}^{\ell }(t_{j}-t_{j+1})}{\prod_{j=0}^{m-1}\left(
		t_{N-3j-2}-t_{N-3j}\right) }h_{0}\bigl( x-\sum_{i=0}^{N}\xi
	_{i}(t_{i}-t_{i+1}),\xi _{N}\bigr)  \\
	& \quad  \cdot \prod_{j=0}^{m-1}\left( \partial _{\xi _{N-3j-1}}+\partial
	_{\xi _{N-3j-2}}\right) ^{\sigma _{j}}\left( e^{-\sum_{j=0}^{N}\nu (\xi
		_{j})(t_{j}-t_{j+1})}\left[ \prod_{j=0}^{N-1}k\left( \xi _{j},\xi
	_{j+1}\right) \right] \right) \Biggr\} \text{.}
\end{align*}%
Thanks to our arrangement of velocity derivatives, in the above expression,
each $k\left( \xi _{j},\xi _{j+1}\right) $ bears at most one derivative,
which is still integrable. It follows from Proposition \ref{prop:
	Derivatives of kernel k} that
\begin{align*}
	& \left\vert \prod_{j=0}^{m-1}\left( \partial _{\xi _{N-3j-1}}+\partial
	_{\xi _{N-3j-2}}\right) ^{\sigma _{j}}\left( e^{-\sum_{j=0}^{N}\nu (\xi
		_{j})(t_{j}-t_{j+1})}\left[ \prod_{j=0}^{N-1}k\left( \xi _{j},\xi
	_{j+1}\right) \right] \right) \right\vert \\
	& \leq Ce^{-q^{\prime }\sum_{j=0}^{N}\nu (\xi _{j})(t_{j}-t_{j+1})}\left[
	\prod_{j=0}^{N-1}H^{(\varsigma _{j})}(\xi _{j},\xi _{j+1})\right] \text{,}
\end{align*}%
and
\begin{equation*}
	\prod_{j=0}^{N-1}H^{(\varsigma _{j})}(\xi _{j},\xi _{j+1})\sim
	\prod_{j=0}^{N-1}\frac{e^{-\frac{q^{\prime }}{8}\frac{\left( |\xi
				_{j}|^{2}-|\xi _{j+1}|^{2}\right) ^{2}}{|\xi _{j}-\xi _{j+1}|^{2}}-\frac{%
				q^{\prime }}{8}\left\vert \xi _{j}-\xi _{j+1}\right\vert ^{2}}}{|\xi
		_{j}-\xi _{j+1}|\left( 1+|\xi _{j}|+|\xi _{j+1}|\right) ^{1-\gamma }}%
	\prod_{j=0}^{N-1}\left( \frac{1+|\xi _{j}|}{|\xi _{j}-\xi _{j+1}|}\right)
	^{\varsigma _{j}}\text{,}
\end{equation*}%
where $\varsigma _{j}\in \left\{ 0,1\right\} $ with $\sum_{j=0}^{N-1}%
\varsigma _{j}=m$, and $H^{(0)}$, $H^{(1)}$ are defined in Proposition \ref%
{prop: Derivatives of kernel k}. Therefore,
\begin{align*}
	\left\vert \left( I_{2}\right) \right\vert & \leq C\int_{T_{N}}\int_{\Xi
		_{N}}\Biggl\{ e^{-q^{\prime }\sum_{j=0}^{N}\nu (\xi _{j})(t_{j}-t_{j+1})}%
	\Bigl[\prod_{j=0}^{N-1}H^{(\varsigma _{j})}(\xi _{j},\xi _{j+1})\Bigr]
	\\
	& \quad \cdot \frac{\prod_{j=1}^{\ell }(t_{j}-t_{j+1})}{\prod_{j=0}^{m-1}%
		\left( t_{N-3j-2}-t_{N-3j}\right) }\Bigl\vert h_{0}\bigl(
	x-\sum_{i=0}^{N}\xi _{i}(t_{i}-t_{i+1}),\xi _{N}\bigr) \Bigr\vert \Biggr\}
	\text{.}
\end{align*}

Now we will derive two types of estimates for $\left( I_{2}\right) $: the
first one is concerning the weighted $L_{\xi }^{2}L_{x}^{2}$ norm, while the
second one is concerning its pointwise estimate assuming $h_{0}$ is
compactly supported in $x$-variable. As for the first one, applying
Cauchy-Schwartz inequality, we have
\begin{align*}
	\left\vert \left( I_{2}\right) \right\vert & \leq C\left[ \int_{T_{N}}\int_{%
		\Xi _{N}}e^{-q^{\prime }\sum_{j=0}^{N}\nu (\xi _{j})(t_{j}-t_{j+1})}\Bigl[%
	\prod_{j=0}^{N-1}H^{(\varsigma _{j})}(\xi _{j},\xi _{j+1})\Bigr]\frac{%
		\prod_{j=1}^{\ell }(t_{j}-t_{j+1})}{\prod_{j=0}^{m-1}\left(
		t_{N-3j-2}-t_{N-3j}\right) }\right] ^{1/2} \\
	& \qquad \cdot \Biggl[ \int_{T_{N}}\int_{\Xi _{N}}e^{-q^{\prime
		}\sum_{j=0}^{N}\nu (\xi _{j})(t_{j}-t_{j+1})}\Bigl[\prod_{j=0}^{N-1}H^{(%
		\varsigma _{j})}(\xi _{j},\xi _{j+1})\Bigr] \\
	& \qquad \frac{\prod_{j=1}^{\ell }(t_{j}-t_{j+1})}{\prod_{j=0}^{m-1}%
		\left( t_{N-3j-2}-t_{N-3j}\right) }\Bigl\vert h_{0}\bigl(
	x-\sum_{i=0}^{N}\xi _{i}(t_{i}-t_{i+1}),\xi _{N}\bigr) \Bigr\vert ^{2}%
	\Biggr] ^{1/2}\text{.}
\end{align*}

To estimate the first integral, we divide the integral domain $\int_{\Xi
	_{N}}=\int_{\mathbb{R}^{3N}}d\xi _{1}d\xi _{2}\cdots d\xi _{N}$ into two
parts:
\begin{equation*}
	\mathbf{D}_{1}=\left\{ \left( \xi _{1},\xi _{2},\cdots ,\xi _{N}\right) \in
	\mathbb{R}^{3N}\Big|\,\xi _{\max }\geq 2\xi _{\min }\right\} \text{,}\quad
	\mathbf{D}_{2}=\left\{ \left( \xi _{1},\xi _{2},\cdots ,\xi _{N}\right) \in
	\mathbb{R}^{3N}\Big|\,\xi _{\max }\leq 2\xi _{\min }\right\} \text{,}
\end{equation*}%
where $\xi _{\max }:=\max \left\{ |\xi _{0}|,|\xi _{1}|,\cdots ,|\xi
_{N}|\right\} $, $\xi _{\min }:=\min \left\{ |\xi _{0}|,|\xi _{1}|,\cdots
,|\xi _{N}|\right\} $. On $\mathbf{D}_{1}$, from the inequality
\begin{equation}
	\sum_{j=0}^{N-1}|\xi _{j}-\xi _{j+1}|^{2}\geq c|\xi _{\max }-\xi _{\min
	}|^{2}\geq c|\xi _{\max }|^{2}\geq c|\xi |^{2}\text{,}  \label{D1-xi}
\end{equation}%
and the estimates of $H^{(\varsigma _{j})}$, we find
\begin{align*}
	& \int_{T_{N}}\int_{\Xi _{N}}e^{-q^{\prime }\sum_{j=0}^{N}\nu (\xi
		_{j})(t_{j}-t_{j+1})}\Bigl[\prod_{j=0}^{N-1}H^{(\varsigma _{j})}(\xi
	_{j},\xi _{j+1})\Bigr]\frac{\prod_{j=1}^{\ell }(t_{j}-t_{j+1})}{%
		\prod_{j=0}^{m-1}\left( t_{N-3j-2}-t_{N-3j}\right) } \\
	& \leq Ce^{-c\left( t+|\xi |^{2}\right) }\int_{T_{N}}\frac{\prod_{j=1}^{\ell
		}(t_{j}-t_{j+1})}{\prod_{j=0}^{m-1}\left( t_{N-3j-2}-t_{N-3j}\right) } \\
	& \leq Ce^{-c_{0}\left( t+|\xi |^{2}\right) }\text{.}
\end{align*}%
On $\mathbf{D}_{2}$, using the relation $|\xi _{j}|\sim |\xi |$ and Lemma %
\ref{lem-Caflisch} yields
\begin{align*}
	& \int_{T_{N}}\int_{\mathbf{D}_{2}}e^{-q^{\prime }\sum_{j=0}^{N}\nu (\xi
		_{j})(t_{j}-t_{j+1})}\Bigl[\prod_{j=0}^{N-1}H^{(\varsigma _{j})}(\xi
	_{j},\xi _{j+1})\Bigr]\frac{\prod_{j=1}^{\ell }(t_{j}-t_{j+1})}{%
		\prod_{j=0}^{m-1}\left( t_{N-3j-2}-t_{N-3j}\right) } \\
	& \leq  Ce^{-c\nu (\xi )t}\int_{T_{N}}\int_{\mathbf{D}_{2}}\left[ \frac{%
		\prod_{j=1}^{\ell }(t_{j}-t_{j+1})}{\prod_{j=0}^{m-1}\left(
		t_{N-3j-2}-t_{N-3j}\right) }\frac{\left( 1+|\xi |\right) ^{m}}{\left( 1+|\xi
		|\right) ^{N\left( 1-\gamma \right) }}\prod_{j=0}^{N-1}\frac{e^{-\frac{%
				q^{\prime }}{8}\frac{\left( |\xi _{j}|^{2}-|\xi _{j+1}|^{2}\right) ^{2}}{%
				|\xi _{j}-\xi _{j+1}|^{2}}-\frac{q^{\prime }}{8}\left\vert \xi _{j}-\xi
			_{j+1}\right\vert ^{2}}}{|\xi _{j}-\xi _{j+1}|^{2}}\right] \\
	& \leq  Ce^{-c\nu (\xi )t}\left( 1+|\xi |\right) ^{m-N\left( 1-\gamma \right)
		-N}\int_{T_{N}}\frac{\prod_{j=1}^{\ell }(t_{j}-t_{j+1})}{\prod_{j=0}^{m-1}%
		\left( t_{N-3j-2}-t_{N-3j}\right) } \\
	& \leq  Ce^{-c\nu (\xi )t}\left( 1+|\xi |\right) ^{m-N\left( 1-\gamma \right)
		-N}\frac{t^{N-m}}{\left( N-m\right) !}t^{\ell } \\
	& \leq  Ce^{-c\nu (\xi )t}\left( 1+|\xi |\right) ^{m-N\left(
		1-\gamma \right) -N}t^{N-|\alpha |} \\
	& = Ce^{-c\nu (\xi )t}\left[ \left( 1+|\xi |\right) ^{\gamma }t%
	\right] ^{N-|\alpha |}\left( 1+|\xi |\right) ^{m-N\left( 1-\gamma \right)
		-N-\left( N-|\alpha |\right) \gamma } \\
	& \leq  C e^{-c_{0}\nu (\xi )t}\left( 1+|\xi |\right)
	^{-2N+m+|\alpha |\gamma }\text{.}
\end{align*}%
Therefore,

\begin{align*}
	\left\vert \left( I_{2}\right) \right\vert ^{2}& \leq Ce^{-c_{0}\nu (\xi
		)t}\left( 1+|\xi |\right) ^{-2N+m+|\alpha |\gamma } \\
	& \quad  \int_{T_{N}}\int_{\Xi _{N}}e^{-q^{\prime }\sum_{j=0}^{N}\nu
		(\xi _{j})(t_{j}-t_{j+1})}\Bigl[\prod_{j=0}^{N-1}H^{(\varsigma _{j})}(\xi
	_{j},\xi _{j+1})\Bigr] \\
	& \quad \cdot \frac{\prod_{j=1}^{\ell }(t_{j}-t_{j+1})}{\prod_{j=0}^{m-1}\left(
		t_{N-3j-2}-t_{N-3j}\right) }\Bigl\vert h_{0}\bigl( x-\sum_{j=0}^{N}\xi
	_{j}(t_{j}-t_{j+1}),\xi _{N}\bigr) \Bigr\vert ^{2}.
\end{align*}%
Integrating over $\Xi _{N}$ and $T_{N}$, using $L^{2}$-boundedness of $K$
operator and integrability of time integral, we have
\begin{equation*}
	\left\Vert \left( 1+|\xi |\right) ^{N-\frac{m+|\alpha |\gamma }{2}}\left(
	I_{2}\right) \right\Vert _{L_{\xi }^{2}L_{x}^{2}}\leq Ce^{-ct}\left\Vert
	h_{0}\right\Vert _{L_{\xi }^{2}L_{x}^{2}}\text{.}
\end{equation*}%
The estimates for other terms, like $\left( II\right) $ are similar, so that
we can conclude
\begin{equation*}
	\left\Vert \left( 1+|\xi |\right) ^{N-\frac{|\alpha |+|\beta |+|\alpha
			|\gamma }{2}}\partial _{x}^{\alpha }\partial _{\xi }^{\beta }\mathbf{M}%
	_{N}^{t}h_{0}\right\Vert _{L_{\xi }^{2}L_{x}^{2}}\leq Ce^{-ct}\left\Vert
	h_{0}\right\Vert _{L_{\xi }^{2}L_{x}^{2}}\quad \mbox{for }N\geq 3\left(
	|\alpha |+|\beta |\right) \text{.}
\end{equation*}

In the rest part, we assume $h_{0}\left( x,\xi \right) =0$ for $|x|>1$, $\xi
\in \mathbb{R}^{3}$, and show the pointwise result. We still focus on $%
\left( I_{2}\right) $:

\begin{align*}
	\left\vert \left( I_{2}\right) \right\vert & \leq C\int_{T_{N}}\int_{\Xi
		_{N}}\Biggl\{ e^{-q^{\prime }\sum_{j=0}^{N}\nu (\xi _{j})(t_{j}-t_{j+1})}%
	\Bigl[\prod_{j=0}^{N-1}H^{(\varsigma _{j})}(\xi _{j},\xi _{j+1})\Bigr]
	\\
	& \quad \cdot \frac{\prod_{j=1}^{\ell }(t_{j}-t_{j+1})}{\prod_{j=0}^{m-1}%
		\left( t_{N-3j-2}-t_{N-3j}\right) }\Bigl\vert h_{0}\bigl(
	x-\sum_{i=0}^{N}\xi _{i}(t_{i}-t_{i+1}),\xi _{N}\bigr) \Bigr\vert \Biggr\}
	\text{.}
\end{align*}%
Because of the support of $h_{0}$, we only need to consider the integration
over
\begin{equation*}
	\mathscr{D}=\left( T_{N}\times \mathbb{R}^{3N}\right) \cap \left\{ \left(
	t_{1},t_{2},\cdots ,t_{N};\xi _{1},\xi _{2},\cdots ,\xi _{N}\right) \,\big|%
	\,|x-\sum_{i=0}^{N}\xi _{i}(t_{i}-t_{i+1})|<1\right\} .
\end{equation*}%
Again we decompose $\int_{\Xi _{N}}=\int_{\mathbb{R}^{3N}}d\xi _{1}d\xi
_{2}\cdots d\xi _{N}$ into $\mathbf{D}_{1}$ and $\mathbf{D}_{2}$ as above
and consider the integral on
\begin{equation*}
	\mathscr{D}\cap \left( T_{N}\times \mathbf{D}_{1}\right) \qquad \text{and}%
	\qquad \mathscr{D}\cap \left( T_{N}\times \mathbf{D}_{2}\right) \text{,}
\end{equation*}%
respectively.

On $\mathscr{D}\cap \left( T_{N}\times \mathbf{D}_{1}\right) $, in view of (%
\ref{D1-xi}),
\begin{align*}
	&e^{-q^{\prime }\sum_{j=0}^{N}\nu (\xi
		_{j})(t_{j}-t_{j+1})}\prod_{j=0}^{N-1}H^{(\varsigma _{j})}(\xi _{j},\xi
	_{j+1}) \\
	&\leq e^{-c(t+\xi _{\max }^{2})}\prod_{j=0}^{N-1}\left[ \frac{1}{|\xi
		_{j}-\xi _{j+1}|}\frac{e^{-\frac{q^{\prime }}{9}\frac{\left\vert \left\vert
				\xi _{j+1}\right\vert ^{2}-\left\vert \xi _{j}\right\vert ^{2}\right\vert
				^{2}}{|\xi _{j+1}-\xi _{j}|^{2}}-\frac{q^{\prime }}{9}|\xi _{j+1}-\xi
			_{j}|^{2}}}{(1+|\xi _{j}|+|\xi _{j+1}|)^{1-\gamma }}\left( \frac{1+|\xi _{j}|%
	}{|\xi _{j}-\xi _{j+1}|}\right) ^{\varsigma _{j}}\right] \text{.}
\end{align*}%
Letting $y=x-\sum_{i=0}^{N}\xi _{i}(t_{i}-t_{i+1})$ with $\left\vert
y\right\vert <1$, gives%
\begin{multline}
	t+\xi _{\max }^{2} \geq  \frac{t}{2}+\frac{1}{2}\left\vert \xi \right\vert ^{2}+\sqrt{\xi
		_{\max }t}-\frac{1}{2}\\
	\geq \frac{t}{2}+\frac{1}{2}\left\vert \xi \right\vert ^{2}+\sqrt{%
		\sum_{i=0}^{N}\left\vert \xi _{i}\right\vert \left( t_{i}-t_{i+1}\right) }-%
	\frac{1}{2} 	\geq \frac{t}{2}+\frac{1}{2}\left\vert \xi \right\vert ^{2}+\left\vert
	x\right\vert ^{1/2}-\frac{3}{2},\notag
\end{multline}
and thus
\begin{align*}
	& \iint_{\mathscr{D}\cap \left( T_{N}\times \mathbf{D}_{1}\right) }\Biggl\{
	e^{-q^{\prime }\sum_{j=0}^{N}\nu (\xi _{j})(t_{j}-t_{j+1})}\Bigl[%
	\prod_{j=0}^{N-1}H^{(\varsigma _{j})}(\xi _{j},\xi _{j+1})\Bigr] \\
	&  \quad \cdot \frac{\prod_{j=1}^{\ell }(t_{j}-t_{j+1})}{\prod_{j=0}^{m-1}%
		\left( t_{N-3j-2}-t_{N-3j}\right) }\Bigl\vert h_{0}\bigl(
	x-\sum_{i=0}^{N}\xi _{i}(t_{i}-t_{i+1}),\xi _{N}\bigr) \Bigr\vert \Biggr\}
	\\
	& \leq Ce^{-c_{0}\left( t+|\xi |^{2}+|x|^{1/2}\right) }\left\Vert e^{-\frac{%
			c\left\vert \xi \right\vert ^{2}}{4}}h_{0}\right\Vert _{L_{x,\xi }^{\infty }}%
	\text{.}
\end{align*}%

On $\mathscr{D}\cap \left( T_{N}\times \mathbf{D}_{2}\right) $, $\frac{1}{2}%
\left\vert \xi \right\vert \leq |\xi _{j}|\leq 2|\xi |$ for $j=1$,$\dots $, $%
N$, and
\begin{align*}
	&e^{-q^{\prime }\sum_{j=0}^{N}\nu (\xi _{j})(t_{j}-t_{j+1})}\Bigl[%
	\prod_{j=0}^{N-1}H^{(\varsigma _{j})}(\xi _{j},\xi _{j+1})\Bigr]\left(
	1+\left\vert \xi \right\vert \right) ^{-Q} \\
	&\leq e^{-c\nu \left( \xi \right) t}\left( 1+\left\vert \xi \right\vert
	\right) ^{-N\left( 1-\gamma \right) +m-Q}\prod_{j=0}^{N-1}\left[ \frac{1}{%
		|\xi _{j}-\xi _{j+1}|}e^{-\frac{q^{\prime }}{9}\frac{\left\vert \left\vert
			\xi _{j+1}\right\vert ^{2}-\left\vert \xi _{j}\right\vert ^{2}\right\vert
			^{2}}{|\xi _{j+1}-\xi _{j}|^{2}}-\frac{q^{\prime }}{9}|\xi _{j+1}-\xi
		_{j}|^{2}}\left( \frac{1}{|\xi _{j}-\xi _{j+1}|}\right) ^{\varsigma _{j}}%
	\right] \text{.}
\end{align*}%
Assume $\left\vert x\right\vert \geq 2$. Letting $y=x-\sum_{i=0}^{N}\xi
_{i}(t_{i}-t_{i+1})$, $\left\vert y\right\vert <1$, gives%
\begin{equation*}
	1+\left\vert \xi \right\vert \geq 1+\frac{1}{2t}\sum_{i=0}^{N}\left\vert \xi
	_{i}\right\vert \left( t_{i}-t_{i+1}\right) \geq \left( 1+\frac{1}{2t}%
	\left\vert x-y\right\vert \right) \geq \frac{1}{8}\left( 1+\left\vert
	x\right\vert \right) t^{-1}\text{,}
\end{equation*}%
on $\mathscr{D}\cap \left( T_{N}\times \mathbf{D}_{2}\right) $, so that for $%
M\geq 0$,%
\begin{align*}
	e^{-\frac{c}{2}\nu \left( \xi \right) t}\left( 1+\left\vert x\right\vert
	\right) ^{M} &\leq Ce^{-\frac{c^{\prime }}{2}\left( 1+\left\vert \xi
		\right\vert \right) ^{\gamma }t}\left( 1+\left\vert \xi \right\vert \right)
	^{M}t^{M} \\
	&=C^{\prime }e^{-\frac{c^{\prime }}{2}\left( 1+\left\vert \xi \right\vert
		\right) ^{\gamma }t}\left( \left( 1+\left\vert \xi \right\vert \right)
	^{\gamma }t\right) ^{M}\left( 1+\left\vert \xi \right\vert \right)
	^{-M\gamma }\left( 1+\left\vert \xi \right\vert \right) ^{M} \\
	&\leq C^{\prime \prime }\left( 1+\left\vert \xi \right\vert \right)
	^{M\left( 1-\gamma \right) }\text{.}
\end{align*}%
Together with Lemma \ref{lem-Caflisch}, it follows that for $\left\vert
x\right\vert \geq 2$,
\begin{align*}
	& \iint_{\mathscr{D}\cap \left( T_{N}\times \mathbf{D}_{2}\right) }\Biggl\{
	e^{-q^{\prime }\sum_{j=0}^{N}\nu (\xi _{j})(t_{j}-t_{j+1})}\Bigl[%
	\prod_{j=0}^{N-1}H^{(\varsigma _{j})}(\xi _{j},\xi _{j+1})\Bigr] \\
	& \qquad \cdot \frac{\prod_{j=1}^{\ell }(t_{j}-t_{j+1})}{\prod_{j=0}^{m-1}%
		\left( t_{N-3j-2}-t_{N-3j}\right) }\Bigl\vert h_{0}\bigl(
	x-\sum_{i=0}^{N}\xi _{i}(t_{i}-t_{i+1}),\xi _{N}\bigr) \Bigr\vert \Biggr\}
	\\
	& \leq Ce^{-\frac{c}{2}\nu \left( \xi \right) t}\left( 1+\left\vert \xi
	\right\vert \right) ^{-N\left( 1-\gamma \right) +m+M\left( 1-\gamma \right)
		-Q}\left( 1+\left\vert x\right\vert \right) ^{-M}\left\Vert \left( 1+|\xi
	|\right) ^{Q}h_{0}\right\Vert _{L_{x,\xi }^{\infty }}\left( 1+\left\vert \xi
	\right\vert \right) ^{-N}\frac{t^{N-m}t^{\ell }}{\left( N-m\right) !} \\
	& \leq C^{\prime }e^{-\frac{c}{2}\nu \left( \xi \right) t}\left(
	1+\left\vert \xi \right\vert \right) ^{-N\left( 1-\gamma \right) +m+M\left(
		1-\gamma \right) -N-Q}\left( 1+\left\vert x\right\vert \right)
	^{-M}t^{N-\left\vert \alpha \right\vert }\left\Vert \left( 1+|\xi |\right)
	^{Q}h_{0}\right\Vert _{L_{x,\xi }^{\infty }} \\
	& =C^{\prime }e^{-\frac{c}{2}\nu \left( \xi \right) t}\left( \left(
	1+\left\vert \xi \right\vert \right) ^{\gamma }t\right) ^{N-\left\vert
		\alpha \right\vert }\left( 1+\left\vert \xi \right\vert \right) ^{-N\left(
		1-\gamma \right) +m+M\left( 1-\gamma \right) -N-\left( N-\left\vert \alpha
		\right\vert \right) \gamma -Q}\left( 1+\left\vert x\right\vert \right)
	^{-M}\left\Vert \left( 1+|\xi |\right) ^{Q}h_{0}\right\Vert _{L_{x,\xi
		}^{\infty }} \\
	& \leq C^{\prime \prime }e^{-c_{0}\nu \left( \xi \right) t}\left(
	1+\left\vert \xi \right\vert \right) ^{-2N-Q+\left\vert \alpha \right\vert
		+\left\vert \beta \right\vert +\left\vert \alpha \right\vert \gamma +M\left(
		1-\gamma \right) }\left( 1+\left\vert x\right\vert \right) ^{-M}\left\Vert
	\left( 1+|\xi |\right) ^{Q}h_{0}\right\Vert _{L_{x,\xi }^{\infty }}\text{.}
\end{align*}%
It is easy to see that for $\left\vert x\right\vert \leq 2$,
\begin{align*}
	& \iint_{\mathscr{D}\cap \left( T_{N}\times \mathbf{D}_{2}\right) }
	\Biggl\{
	e^{-q^{\prime }\sum_{j=0}^{N}\nu (\xi _{j})(t_{j}-t_{j+1})}\Bigl[%
	\prod_{j=0}^{N-1}H^{(\varsigma _{j})}(\xi _{j},\xi _{j+1})\Bigr]\\
	& \qquad \cdot \frac{\prod_{j=1}^{\ell }(t_{j}-t_{j+1})}{\prod_{j=0}^{m-1}%
		\left( t_{N-3j-2}-t_{N-3j}\right) }\Bigl\vert h_{0}\bigl(
	x-\sum_{i=0}^{N}\xi _{i}(t_{i}-t_{i+1}),\xi _{N}\bigr) \Bigr\vert \Biggr\}
	\\
	& \leq Ce^{-c_{0}\nu \left( \xi \right) t}\left( 1+\left\vert \xi
	\right\vert \right) ^{-2N-Q+\left\vert \alpha \right\vert +\left\vert \beta
		\right\vert +\left\vert \alpha \right\vert \gamma }\left( 1+\left\vert
	x\right\vert \right) ^{-M}\left\Vert \left( 1+|\xi |\right)
	^{Q}h_{0}\right\Vert _{L_{x,\xi }^{\infty }}\text{.}
\end{align*}%
The estimates for other terms, like $(II)$ are similar, so that we can
conclude%
\begin{align*}
	\left\vert \partial _{x}^{\alpha }\partial _{\xi }^{\beta }\mathbf{M}%
	_{N}^{t}h_{0}\left( x,\xi \right) \right\vert & \leq Ce^{-c_{0}\nu (\xi
		)t}\left( 1+|x|\right) ^{-M}\left( 1+|\xi |\right) ^{-2N-Q+\left[ |\alpha
		|+|\beta |+|\alpha |\gamma \right] +M(1-\gamma )} \\
	& \quad \cdot \left\Vert \left( 1+|\xi |\right) ^{Q}h_{0}\right\Vert
	_{L_{x,\xi }^{\infty }}\quad \mbox{for }Q\in \mathbb{R}\text{,}\;N\geq
	3\left( |\alpha |+|\beta |\right) \text{,}
\end{align*}%
when $h_{0}\left( x,\xi \right) =0$ for $\left\vert x\right\vert >1$, $\xi
\in \mathbb{R}^{3}$. This completes the proof.
\end{proof}

With this Enhanced Mixture Lemma, we have the following estimate for the
singular wave $h^{(j)}$:

\begin{proposition}
	\label{singular} Assume that $h_{0}\in L_{\xi ,Q}^{\infty }L_{x}^{\infty }$
	with $Q\geq 0$ is compactly supported in the $x$ variable for all $\xi $,
	namely,
	\begin{equation*}
		h_{0}\left( x,\xi \right) \equiv 0\text{ for }\left\vert x\right\vert \geq 1%
		\text{, }\xi \in \mathbb{R}^{3}\text{.}
	\end{equation*}%
	Then for any $0\leq \theta \leq Q$,
	\begin{equation}\label{sin}
		\left\vert h^{(j)}\right\vert _{L_{\xi ,Q-\theta }^{\infty }}\lesssim
		e^{-ct}\left\langle x\right\rangle ^{-\frac{\theta }{1-\gamma }}\left\vert
		h_{0}\right\vert _{L_{x}^{\infty }L_{\xi ,Q}^{\infty }}\text{,}\quad j\geq 0%
		\text{.}
	\end{equation}
\end{proposition}

\subsection{Estimate of remainder I (structure inside finite Mach number region): Long wave short wave decomposition} \label{sec:SR-spec}

Throughout this subsection, we will devote ourselves to the study of
remainder term $\mathcal{R}^{(\hat{N})}$:

\begin{equation*}
\partial _{t}\mathcal{R}^{(\hat{N})}+\xi \cdot \nabla _{x}\mathcal{R}^{(\hat{%
N})}-L\mathcal{R}^{(\hat{N})}=Kh^{(\hat{N})}\text{,}\qquad \mathcal{R}^{(%
\hat{N})}\big|_{t=0}=0\text{.}
\end{equation*}%
From the Enhanced Mixture Lemma, the source term $Kh^{(\hat{N})}$ satisfies
that for any given $\hat{q}_{1}$, $\hat{q}_{2}$, $\hat{M}$, $\hat{P}\geq 0$
and $Q\in \mathbb{R}$, there exists positive integer $\hat{N}$ such that

\begin{equation*}
\sum_{0\leq |\alpha |\leq \hat{q}_{1},0\leq |\beta |\leq \hat{q}%
_{2}}\left\vert \partial _{x}^{\alpha }\partial _{\xi }^{\beta }Kh^{(\hat{N}%
)}\left( t,x,\xi \right) \right\vert \leq Ce^{-c_{0}t}\left( 1+|x|\right) ^{-%
\hat{M}}\left( 1+|\xi |\right) ^{-\hat{P}}\left\Vert \left\langle \xi
\right\rangle ^{Q}h_{0}\right\Vert _{L_{x,\xi }^{\infty }}.
\end{equation*}%
Now fix $\hat{q}_{1}$, $\hat{q}_{2}$, $\hat{M}$, $\hat{P}$ with $\hat{q}%
_{1}\geq 2$, $\hat{M}>3$, $\hat{P}>\frac{3}{2}$. (The actual choice will be
decided later.)

It is easy to see that
\begin{equation*}
\mathcal{R}^{(\hat{N})}=\int_{0}^{t}\mathbb{G}^{t-s}Kh^{(\hat{N})}(s)ds,
\end{equation*}%
and therefore using Proposition \ref{Prop-linear} gives
\begin{equation}\label{remainder}
\Vert \mathcal{R}^{(\hat{N})}\Vert _{L_{\xi ,\hat{P}}^{\infty }L_{x}^{\infty
}}\leq (1+t)^{-3/2}\left\Vert \left\langle \xi \right\rangle
^{Q}h_{0}\right\Vert _{L_{x,\xi }^{\infty }}\qquad \text{for }Q>3/2\text{. }
\end{equation}

Note that
\begin{equation*}
g\equiv \mathbb{G}_{L;\perp }^{t}h_{0}+\mathbb{G}_{S}^{t}h_{0}-\mathcal{W}^{(%
\hat{N})}=\mathcal{R}^{(\hat{N})}-\mathbb{G}_{L;0}^{t}h_{0}\text{,}
\end{equation*}%
by Propositions \ref{Prop-linear} and \ref{singular}, we have that for $Q>3/2
$,
\begin{equation*}
\left\vert g\right\vert _{L_{\xi ,Q}^{\infty }}=\left\vert \mathbb{G}%
_{L;\perp }^{t}h_{0}+\mathbb{G}_{S}^{t}h_{0}-\mathcal{W}^{(\hat{N}%
)}\right\vert _{L_{\xi ,Q}^{\infty }}\leq e^{-ct}\left( \Vert h_{0}\Vert
_{L_{\xi ,Q}^{\infty }L_{x}^{\infty }}+\Vert h_{0}\Vert _{L_{\xi
}^{2}L_{x}^{2}}\right) \text{.}
\end{equation*}

On the other hand, for arbitrary $\hat{Q}>0$, one can choose $\hat{P}>\max
\{2\hat{Q}-Q,\frac{3}{2}\}$ and then use Proposition \ref{Prop-linear} and (\ref{remainder}) to obtain
\begin{equation*}
\left\vert g\right\vert _{L_{\xi ,\hat{P}}^{\infty }}=\left\vert \mathcal{R}%
^{(\hat{N})}-\mathbb{G}_{L;0}^{t}h_{0}\right\vert _{L_{\xi ,\hat{P}}^{\infty
}}\leq (1+t)^{-3/2}\left( \Vert h_{0}\Vert _{L_{\xi ,Q}^{\infty
}L_{x}^{\infty }}+\Vert h_{0}\Vert _{L_{\xi }^{\infty }(L_{x}^{\infty }\cap
L_{x}^{1})}\right) \text{.}
\end{equation*}%
By interpolation, we get
\begin{equation*}
\left\vert \mathbb{G}_{L;\perp }^{t}h_{0}+\mathbb{G}_{S}^{t}h_{0}-\mathcal{W}%
^{(\hat{N})}\right\vert _{L_{\xi ,\hat{Q}}^{\infty }}\leq \left\vert
g\right\vert _{L_{\xi ,Q}^{\infty }}^{1/2}\left\vert g\right\vert _{L_{\xi ,%
\hat{P}}^{\infty }}^{1/2}\leq (1+t)^{-3/4}e^{-\frac{c}{2}t}\Vert h_{0}\Vert
_{L_{\xi ,Q}^{\infty }(L_{x}^{\infty }\cap L_{x}^{1})}
\end{equation*}%
and therefore (using Proposition \ref{Prop-linear-fluid})
\begin{equation}\label{remainder-inside}
\left\vert \mathcal{R}^{(\hat{N})}\right\vert _{L_{\xi ,\hat{Q}}^{\infty
}}\leq \left\vert g\right\vert _{L_{\xi ,\hat{Q}}^{\infty }}+\left\vert
\mathbb{G}_{L;0}^{t}h_{0}\right\vert _{L_{\xi ,\hat{Q}}^{\infty }}\leq \left[
\Phi_{m} (t,x)+e^{-ct}\right] \Vert h_{0}\Vert _{L_{\xi ,Q}^{\infty
}L_{x}^{\infty }}
\end{equation}%
for $h_{0}$ compactly supported in $x$. Furthermore, if $\mathrm{P}%
_{0}h_{0}=0$, then we will have extra $1/2$ time decay, that is,
\begin{equation*}
\left\vert \mathcal{R}^{(\hat{N})}\right\vert _{L_{\xi ,\hat{Q}}^{\infty
}}\leq \left[ (1+t)^{-1/2}\Phi_{m} (t,x)+e^{-ct}\right] \Vert h_{0}\Vert
_{L_{\xi ,Q}^{\infty }L_{x}^{\infty }}\text{,}\quad \hbox{if}\quad \mathrm{P}%
_{0}h_{0}=0\text{.}
\end{equation*}

\subsection{Estimate of remainder II (structure outside finite Mach number
region): Weighted energy estimate}

In the previous subsection, we obtain the estimate of $\mathcal{R}^{(\hat{N}%
)}$. However, one term in the estimate is only concerning the time decay. If
restricting to the finite Mach number region, i.e., $|x|\leq 2\mathfrak{M}t$, for any $\mathfrak{M}>0$, one can
use time decay to gain the space decay. In order to get the precise
structure outside the finite Mach number region, we will use the weighted
energy method. We introduce the weight function:
\begin{equation*}
w\left( t,x,\xi \right) :=c_{1}\left[ \delta \left( \left\langle
x\right\rangle -\mathfrak{M}t\right) \right] ^{\frac{\sigma }{1-\gamma }}\left( 1-\chi
\left( \frac{\delta \left( \left\langle x\right\rangle -\mathfrak{M}t\right) }{%
\left\langle \xi \right\rangle _{D}^{1-\gamma }}\right) \right)
+c_{2}\left\langle \xi \right\rangle _{D}^{\sigma }\chi \left( \frac{\delta
\left( \left\langle x\right\rangle -\mathfrak{M}t\right) }{\left\langle \xi
\right\rangle _{D}^{1-\gamma }}\right) \text{,}
\end{equation*}%
where $\left\langle \xi \right\rangle _{D}=(D^{2}+|\xi |^{2})^{1/2}$. The cut-off function $\chi :{\mathbb{R}}\rightarrow {\mathbb{R}}$,
which is a smooth non-increasing function, $\chi (s)=1$ for $s\leq 1$, $\chi
(s)=0$ for $s\geq 2$ and $0\leq \chi \leq 1$.  The
constants $c_{1}$, $c_{2}>0$ are chosen to meet the requirement $\partial
_{t}w\leq 0$.
Note that for
arbitrary $\hat{Q}>0$, $\sigma \geq 2$, choose $\hat{M}\geq \frac{\sigma }{%
1-\gamma }$ and $\hat{P}\geq 2\hat{Q}-Q+\sigma$, by the Enhanced Mixture Lemma, the source term $Kh^{(\hat{N})}$ is good enough and allows us to do the weighted energy estimate. Applying the same argument in \cite{[LinWangWu]}, one can prove that there exists $\mathfrak{M}>0$
such that
\begin{equation}\label{remainder-outside}
\left\vert \mathcal{R}^{(\hat{N})}\right\vert _{L_{\xi ,\hat{Q}}^{\infty
}}\leq C\left( 1+\left\vert x\right\vert \right) ^{-\frac{\sigma }{1-\gamma }%
}\Vert h_{0}\Vert _{L_{\xi ,Q}^{\infty }L_{x}^{\infty }}\quad \mbox{ for }%
|x|>2\mathfrak{M}t\text{.}
\end{equation}

\subsection{Conclusion of the linearized problem}

Define
\begin{equation*}
\mathbb{G}_{\ast }^{t}h_{0}=\mathcal{W}^{(\hat{N}) }\text{,}%
\quad \mathbb{G}_{\sharp }^{t}h_{0}=\mathcal{R}^{(\hat{N}) }%
\text{.}
\end{equation*}%
In virtue of (\ref{sin}), (\ref{remainder-inside}) and (\ref{remainder-outside}), we have the following theorem of the linearized
problem:

\begin{theorem}\label{PF-estimate}
Assume that $h_{0}\in L_{\xi ,Q}^{\infty }(L_{x}^{\infty }\cap L_{x}^{1})$, $%
Q>3/2$. Then the solution $\mathbb{G}^{t}h_{0}$ can be decomposed as
\begin{equation*}
\mathbb{G}^{t}h_{0}=\mathbb{G}_{\ast }^{t}h_{0}+\mathbb{G}_{\sharp }^{t}h_{0}%
\text{,}
\end{equation*}%
with the following estimates
\begin{equation*}
\left\vert \mathbb{G}_{\ast }^{t}h_{0}\right\vert _{L_{\xi ,Q}^{\infty
}}\lesssim e^{-ct}\left\Vert h_{0}\right\Vert _{L_{\xi ,Q}^{\infty
}L_{x}^{\infty }}\text{,}
\end{equation*}%
and for arbitrary $\hat{Q}>0$,
\begin{equation*}
\Vert \mathbb{G}_{\sharp }^{t}h_{0}\Vert _{L_{\xi ,\hat{Q}}^{\infty
}L_{x}^{\infty }}\lesssim (1+t)^{-3/2}\Vert h_{0}\Vert _{L_{\xi ,Q}^{\infty
}(L_{x}^{\infty }\cap L_{x}^{1})}\text{.}
\end{equation*}%
Moreover, if we further assume $h_{0}$ is compactly supported in the
 $x$ variable for all $\xi $, namely,
\begin{equation*}
h_{0}\left( x,\xi \right) \equiv 0\text{ for }\left\vert x\right\vert \geq 1%
\text{, }\xi \in \mathbb{R}^{3}\text{,}
\end{equation*}%
then for any $\hat{Q}>0$, $\sigma \geq 2$ and $0<\theta \leq Q$,
\begin{equation*}
\left\vert \mathbb{G}_{\ast }^{t}h_{0}\right\vert _{L_{\xi ,Q-\theta
}^{\infty }}\lesssim e^{-ct}\left\langle x\right\rangle ^{-\frac{\theta }{%
1-\gamma }}\left\Vert h_{0}\right\Vert _{L_{x}^{\infty }L_{\xi ,Q}^{\infty }}%
\text{,}
\end{equation*}%
\begin{equation*}
\left\vert \mathbb{G}_{\sharp }^{t}h_{0}\right\vert _{L_{\xi ,\hat{Q}%
}^{\infty }}\lesssim \left[ \Phi_{m} (t,x)+(1+t+|x|)^{-\frac{\sigma }{1-\gamma }}%
\right] \Vert h_{0}\Vert _{L_{\xi ,Q}^{\infty }L_{x}^{\infty }}\text{,}
\end{equation*}%
and
\begin{equation*}
\left\vert \mathbb{G}_{\sharp }^{t}\mathrm{P}_{1}h_{0}\right\vert _{L_{\xi ,%
\hat{Q}}^{\infty }}\lesssim \left[ (1+t)^{-\frac{1}{2}}\Phi_{m}
(t,x)+(1+t+|x|)^{-\frac{\sigma }{1-\gamma }}\right] \Vert h_{0}\Vert
_{L_{\xi ,Q}^{\infty }L_{x}^{\infty }}\text{,}
\end{equation*}
here $m$ can be any arbitrarily large number.
\end{theorem}

\section{Nonlinear Problem}\label{sec:nonlinear}

In this section, we will decompose the nonlinear solution to the fluid and
the non-fluid parts to complete the nonlinear problem.

The fully nonlinear Boltzmann equation is given by
\begin{align}  \label{nonlinear}
\left\{ \begin{aligned} &
\partial_{t}f+\xi\cdot\nabla_{x}f=Lf+\Gamma\left(f,f\right),\\ &
f\big|_{t=0}=\eps f_{0}. \end{aligned} \right.
\end{align}

\subsection{Nonlinear fluid structure}

By the Duhamel principle, the solution satisfies the integal equation
\begin{equation*}
f\left( t\right) =\eps\mathbb{G}^{t}f_{0}+\int_{0}^{t}\mathbb{G}^{t-\tau
}\Gamma \left( f,f\right) \left( \tau \right) d\tau \text{.}
\end{equation*}%
The solution $f$ can be decomposed into two parts:
\begin{equation*}
f=f_{\ast }+f_{\sharp }
\end{equation*}%
through the following coupled integral system
\begin{equation*}
f_{\ast }=\eps\mathbb{G}_{\ast }^{t}f_{0}+\int_{0}^{t}\mathbb{G}_{\ast
}^{t-\tau }\Gamma \left( f_{\ast }+f_{\sharp },f_{\ast }+f_{\sharp }\right)
\left( \tau \right) d\tau \text{,}
\end{equation*}
\begin{equation*}
f_{\sharp }=\eps\mathbb{G}_{\sharp }^{t}f_{0}+\int_{0}^{t}\mathbb{G}_{\sharp
}^{t-\tau }\Gamma \left( f_{\ast }+f_{\sharp },f_{\ast }+f_{\sharp }\right)
\left( \tau \right) d\tau \text{.}
\end{equation*}

For any fixed $\beta >\frac{3}{2}+\gamma $ and $\theta _{0}$, $\theta _{1}>0$%
, we have
\begin{align*}
\left\vert f_{\ast }\right\vert _{L_{\xi ,\beta }^{\infty }}& \leq \eps%
\left\vert \mathbb{G}_{\ast }^{t}f_{0}\right\vert _{L_{\xi ,\beta }^{\infty
}}+\int_{0}^{t}\left\vert \mathbb{G}_{\ast }^{t-\tau }\Gamma \left(
f,f\right) \left( \tau \right) \right\vert _{L_{\xi ,\beta }^{\infty }}d\tau
\\
& \leq \eps\left\vert \mathbb{G}_{\ast }^{t}\left\langle \xi \right\rangle
^{-\theta _{0}}\left\langle \xi \right\rangle ^{\theta _{0}}f_{0}\right\vert
_{L_{\xi ,\beta }^{\infty }}+\int_{0}^{t}\left\vert \mathbb{G}_{\ast
}^{t-\tau }\left\langle \xi \right\rangle ^{-\theta _{1}}\left\langle \xi
\right\rangle ^{\theta _{1}}\Gamma \left( f,f\right) \left( \tau \right)
\right\vert _{L_{\xi ,\beta }^{\infty }}d\tau  \\
& \lesssim \eps e^{-ct}\left\langle x\right\rangle ^{-\frac{\theta _{0}}{%
1-\gamma }}\Vert f_{0}\Vert _{L_{\xi ,\beta +\theta _{0}}^{\infty
}L_{x}^{\infty }} \\
& \quad +\int_{0}^{t}\int_{{\mathbb{R}}^{3}}e^{-c(t-\tau )}(1+|x-y|)^{-\frac{%
\theta _{1}}{1-\gamma }}\left\vert \Gamma (f,f)\right\vert _{L_{\xi ,\beta
+\theta _{1}}^{\infty }}(y,\tau )dyd\tau .
\end{align*}%
Note that by interpolation and Lemma \ref{basic-Gamma}
\begin{align*}
\left\vert \Gamma (f,f)\right\vert _{L_{\xi ,\beta +\theta _{1}}^{\infty }}&
\leq \left\vert \Gamma (f,f)\right\vert _{L_{\xi ,\beta -\gamma }^{\infty
}}^{1/2}\left\vert \Gamma (f,f)\right\vert _{L_{\xi ,2(\beta +\theta
_{1})-(\beta -\gamma )}^{\infty }}^{1/2} \\
& \leq \left\vert \Gamma (f,f)\right\vert _{L_{\xi ,\beta -\gamma }^{\infty
}}^{1/2}\left( \left\vert \Gamma (f,f)\right\vert _{L_{\xi ,\beta -\gamma
}^{\infty }}^{1/2}\left\vert \Gamma (f,f)\right\vert _{L_{\xi ,2^{2}(\beta
+\theta _{1})-(2^{2}-1)(\beta -\gamma )}^{\infty }}^{1/2}\right) ^{1/2} \\
& \leq \left\vert \Gamma (f,f)\right\vert _{L_{\xi ,\beta -\gamma }^{\infty
}}^{1-\frac{1}{2^{k}}}\left\vert \Gamma (f,f)\right\vert _{L_{\xi
,2^{k}(\beta +\theta _{1})-(2^{k}-1)(\beta -\gamma )}^{\infty }}^{\frac{1}{%
2^{k}}} \\
& \lesssim \left\vert f\right\vert _{L_{\xi ,\beta }^{\infty }}^{2-\frac{1}{%
2^{k-1}}}\left\vert f\right\vert _{L_{\xi ,\beta +2^{k}(\theta _{1}+\gamma
)}^{\infty }}^{\frac{1}{2^{k-1}}}.
\end{align*}%
It is well known that the large time behavior of $f$ is given by (see ref.
\cite{[LinWangLyuWu]}):
\begin{equation*}
\left\Vert f\right\Vert _{L_{\xi ,\beta +2^{k}(\theta _{1}+\gamma )}^{\infty
}L_{x}^{\infty }}\lesssim \eps(1+t)^{-3/2}\left\Vert f_{0}\right\Vert
_{L_{\xi ,\beta +2^{k}(\theta _{1}+\gamma )}^{\infty }(L_{x}^{\infty }\cap
L_{x}^{1})}\,\text{,}
\end{equation*}%
which implies that (here we denote $\hat{\beta}=\max \{\beta +\theta
_{0},\beta +2^{k}(\theta _{1}+\gamma )\}$)
\begin{align}
& \quad \left\vert f_{\ast }\right\vert _{L_{\xi ,\beta }^{\infty }}  \notag
\label{nonlinear-singular} \\
& \lesssim \eps e^{-ct}\left\langle x\right\rangle ^{-\frac{\theta _{0}}{%
1-\gamma }}\Vert f_{0}\Vert _{L_{\xi ,\hat{\beta}}^{\infty }L_{x}^{\infty }}
\\
& \quad +\left[ e^{-ct}(1+|x|)^{^{-\frac{\theta _{1}}{1-\gamma }}}\right]
\ast _{t,x}\left[ (1+t)^{-\frac{3}{2^{k}}}\left( |f_{\ast }|_{L_{\xi ,\beta
}^{\infty }}+|f_{\sharp }|_{L_{\xi ,\beta }^{\infty }}\right) ^{2-\frac{1}{%
2^{k-1}}}\right]   \notag \\
& \quad \quad \cdot \eps^{\frac{1}{2^{k-1}}}\left\Vert f_{0}\right\Vert
_{L_{\xi ,\hat{\beta}}^{\infty }(L_{x}^{\infty }\cap L_{x}^{1})}^{\frac{1}{%
2^{k-1}}}\,.  \notag
\end{align}%
As for $f_{\sharp }$, we have
\begin{align}
\left\vert f_{\sharp }\right\vert _{L_{\xi ,\beta }^{\infty }}& \leq \eps%
\left\vert \mathbb{G}_{\sharp }^{t}f_{0}\right\vert _{L_{\xi ,\beta
}^{\infty }}+\int_{0}^{t}\left\vert \mathbb{G}_{\sharp }^{t-\tau }\Gamma
\left( f,f\right) \left( \tau \right) \right\vert _{L_{\xi ,\beta }^{\infty
}}d\tau   \notag  \label{nonlinear-regular} \\
& \lesssim \eps\left[ \Phi_{m} (t,x)+(1+t+|x|)^{-\frac{\sigma }{1-\gamma }}%
\right] \left\Vert f_{0}\right\Vert _{L_{\xi ,\hat{\beta}}^{\infty
}L_{x}^{\infty }}  \notag \\
& \quad +\left[ (1+t)^{-1/2}\Phi_{m} (t,x)+(1+t+|x|)^{-\frac{\sigma }{1-\gamma }}%
\right] \ast _{t,x}\left\vert \Gamma (f,f)\right\vert _{L_{\xi ,\beta
-\gamma }^{\infty }} \\
& \lesssim \eps\left[ \Phi_{m} (t,x)+(1+t+|x|)^{-\frac{\sigma }{1-\gamma }}%
\right] \left\Vert f_{0}\right\Vert _{L_{\xi ,\hat{\beta}}^{\infty
}L_{x}^{\infty }}  \notag \\
& \quad +\left[ (1+t)^{-1/2}\Phi_{m} (t,x)+(1+t+|x|)^{-\frac{\sigma }{1-\gamma }}%
\right] \ast _{t,x}\left[ \left\vert f_{\ast }\right\vert _{L_{\xi ,\beta
}^{\infty }}+\left\vert f_{\sharp }\right\vert _{L_{\xi ,\beta }^{\infty }}%
\right] ^{2}\text{,}  \notag
\end{align}%
here $\sigma $ can be arbitrarily large.

Now, we design the following ansatz:
\begin{align}
\left\vert f_{\sharp }\right\vert _{L_{\xi ,\beta }^{\infty }}& \lesssim \eps%
\left[ \Phi_{m} (x,t)+(1+t+|x|)^{-\frac{\sigma }{1-\gamma }}\right]
\notag \\
& \quad +\eps^{2}\left[ \left( 1+t\right) ^{-2}B_{3/2}(|x|,t)+\left(
1+t\right) ^{-\frac{5}{2}}B_{1}(|x|-\mathbf{c}t,t)\right] \text{,}  \notag
\end{align}%
and
\begin{align}
\left\vert f_{\ast }\right\vert _{L_{\xi ,\beta }^{\infty }}& \lesssim \eps %
e^{-ct}\left\langle x\right\rangle ^{-\frac{\theta _{0}}{1-\gamma }}  \notag
\label{ansatz2} \\
& \quad +\eps^{2}(1+t)^{-3\cdot 2^{-k}}\left[ \left( 1+t\right)
^{-3/2}B_{3/2}(|x|,t)+\left( 1+t\right) ^{-2}B_{1}(|x|-\mathbf{c}t,t)\right]
^{2-\frac{1}{2^{k-1}}}  \notag \\
& \lesssim \eps e^{-ct}\left\langle x\right\rangle ^{-\frac{\theta _{0}}{%
1-\gamma }} \notag \\
& \quad +\eps^{2}\left[ \left( 1+t\right) ^{-3}B_{3\left( 1-2^{-k}\right)
}(|x|,t)+\left( 1+t\right) ^{-(4-2^{-k})}B_{2(1-2^{-k})}(|x|-\mathbf{c}t,t)%
\right] \text{.}  \notag
\end{align}%
Based on the ansatz assumptions, we have
\begin{equation*}
\left\vert f_{\sharp }\right\vert _{L_{\xi ,\beta }^{\infty }}\lesssim \eps%
\left[ \left( 1+t\right) ^{-3/2}B_{3/2}(|x|,t)+\left( 1+t\right)
^{-2}B_{1}(|x|-\mathbf{c}t,t)\right] \text{,}
\end{equation*}%
and
\begin{equation*}
\left\vert f_{\ast }\right\vert _{L_{\xi ,\beta }^{\infty }}\lesssim \eps%
\left[ \left( 1+t\right) ^{-3}B_{3\left( 1-2^{-k}\right) }(|x|,t)+\left(
1+t\right) ^{-(4-2^{-k})}B_{2(1-2^{-k})}(|x|-\mathbf{c}t,t)\right] \text{.}
\end{equation*}%
Therefore,
\begin{equation*}
\left\vert f_{\ast }\right\vert _{L_{\xi ,\beta }^{\infty }}+\left\vert
f_{\sharp }\right\vert _{L_{\xi ,\beta }^{\infty }}\lesssim \eps\left[
\left( 1+t\right) ^{-3/2}B_{3/2}(|x|,t)+\left( 1+t\right) ^{-2}B_{1}(|x|-%
\mathbf{c}t,t)\right] \text{\thinspace .}
\end{equation*}%
To verify the ansatz, in accordance with (\ref{nonlinear-singular}) and (\ref%
{nonlinear-regular}), we need to confirm that
\begin{align*}
& \quad \left[ (1+t)^{-1/2}\Phi_{m} (t,x)+(1+t+|x|)^{-\frac{\sigma }{1-\gamma }}%
\right] \ast _{t,x}\left[ \left( 1+t\right) ^{-3/2}B_{3/2}(|x|,t)+\left(
1+t\right) ^{-2}B_{1}(|x|-\mathbf{c}t,t)\right] ^{2} \\
&\lesssim \left[ (1+t)^{-1/2}\Phi_{m} (t,x)\right] \ast _{t,x}\left[ \left(
1+t\right) ^{-3}B_{3}(|x|,t)+\left( 1+t\right) ^{-4}B_{2}(|x|-\mathbf{c}t,t)%
\right] \\
& \lesssim \left( 1+t\right) ^{-2}B_{3/2}(|x|,t)+\left( 1+t\right) ^{-\frac{5}{2}%
}B_{1}(|x|-\mathbf{c}t,t)
\end{align*}%
and
\begin{align*}
& \quad \left[ e^{-ct}(1+|x|)^{^{-\frac{\theta _{1}}{1-\gamma }}}\right]
\ast _{t,x}(1+t)^{-3\cdot 2^{-k}}\left[ \left( 1+t\right)
^{-3/2}B_{3/2}(|x|,t)+\left( 1+t\right) ^{-2}B_{1}(|x|-\mathbf{c}t,t)\right]
^{2(1-2^{-k})} \\
& \lesssim \left[ e^{-ct}(1+|x|)^{^{-\frac{\theta _{1}}{1-\gamma }}}\right] \ast
_{t,x}\left[ \left( 1+t\right) ^{-3}B_{3\left( 1-2^{-k}\right)
}(|x|,t)+\left( 1+t\right) ^{-(4-2^{-k})}B_{2(1-2^{-k})}(|x|-\mathbf{c}t,t)%
\right] \\
& \lesssim\left( 1+t\right) ^{-3}B_{3\left( 1-2^{-k}\right) }(|x|,t)+\left(
1+t\right) ^{-(4-2^{-k})}B_{2(1-2^{-k})}(|x|-\mathbf{c}t,t)\,.
\end{align*}

For $f_{\sharp }$, it corresponds to the fluid part. Note that the fluid
structure of the hard potential case $\left( 0\leq \gamma <1\right) $ is
almost the same as hard sphere case, the only difference is that the hard
potential case is of polynomial type of arbitrary order, and the hard sphere
case is of exponential type. Therefore, the following nonlinear wave interaction
estimates are parallel to those for the hard sphere case (see Section 6.2 of \cite%
{[LinWangWu2]}) and hence we just state the result.

\begin{lemma}[Nonlinear wave interactions for $f_{\sharp }$]
Let $m$ be any large positive number. Then the following convolution estimates hold:
\begin{enumerate}
\item \begin{equation*}
		\left( 1+t\right) ^{-2}B_{m}(|x|,t)\ast _{x,t}\left( 1+t\right)
		^{-3}B_{3}(|x|,t)\lesssim \left( 1+t\right) ^{-2}B_{3}(|x|,t)\hbox{.}
	\end{equation*}
\item \begin{align*}
	&\mathbf{1}_{\{ \left \vert x\right \vert \leq \mathbf{c}t\}}\left(
	1+t\right) ^{-2}B_{3/2}(|x|,t)\ast _{x,t}\left( 1+t\right) ^{-4}B_{2}(|x|-%
	\mathbf{c}t,t)  \notag \\
	&\lesssim \left( 1+t\right) ^{-2}B_{3/2}(|x|,t)+\left( 1+t\right)
	^{-5/2}B_{1}(|x|-\mathbf{c}t,t) \hbox{.}
\end{align*}
\item \begin{align*}
	&\left( 1+t\right) ^{-5/2}B_{m}(|x|-\mathbf{c}t,t)*_{x,t}\left(1+t\right)
	^{-3}B_{3}(|x|,t)  \\
	&\lesssim \left( 1+t\right) ^{-2}B_{3/2}(|x|,t)+\left( 1+t\right)
	^{-5/2}B_{1}(|x|-\mathbf{c}t,t) \hbox{.}  \notag
\end{align*}
\item \begin{align*}
	&\left( 1+t\right) ^{-2}B_{m}(|x|,t)\ast _{x,t}\left( 1+t\right)
	^{-4}B_{2}(|x|-\mathbf{c}t,t) \\
	&\lesssim \left( 1+t\right) ^{-2}B_{3/2}(|x|,t)+\left(
	1+t\right)^{-5/2}B_{1}(|x|-\mathbf{c}t,t) \hbox{.}  \notag
\end{align*}
\item \begin{align*}
	&\left( 1+t\right) ^{-5/2}B_{m}(|x|-\mathbf{c}%
	t,t)*_{x,t}(1+t)^{-4}B_{2}(|x|-\mathbf{c}t,t)  \label{C10} \\
	&\lesssim \left( 1+t\right) ^{-2}B_{3/2}(|x|,t)+\left( 1+t\right)
	^{-5/2}B_{1}(|x|-\mathbf{c}t,t) \hbox{.}  \notag
\end{align*}
\end{enumerate}

\end{lemma}

For $f_{\ast }$, according to the convolution estimates in Section \ref{sec:convolution}, we have

\begin{lemma}[Nonlinear wave interaction for $f_{*}$] \label{lem:wave-int-*}
Let $k\geq 3$ and $\theta _{1}\geq 6(1-\gamma )(1-2^{-k})$. Then
\begin{enumerate}
	\item \begin{align*}
		&\left[ e^{-ct}(1+|x|)^{^{-\frac{\theta _{1}}{1-\gamma }}}\right] \ast
		_{t,x}\left( 1+t\right) ^{-3}B_{3\left( 1-2^{-k}\right) }(|x|,t) \\
		&\lesssim \left( 1+t\right) ^{-3}B_{3\left( 1-2^{-k}\right) }(|x|,t)\,.
	\end{align*}
	\item \begin{align*}
		&\left[e^{-ct}(1+|x|)^{^{-\frac{\theta_{1}}{1-\gamma}}}\right]%
		\ast_{t,x}\left( 1+t\right)^{-(4-2^{-k})}B_{2(1-2^{-k})}(|x|-\mathbf{c}t,t)
		\\
		&\lesssim \left( 1+t\right)^{-(4-2^{-k})}B_{2(1-2^{-k})}(|x|-\mathbf{c}%
		t,t)\,.
	\end{align*}
\end{enumerate}
\end{lemma}

Therefore, we confirm the structure of the nonlinear part.

\begin{theorem}\label{thm:main}
Let $0\leq \gamma <1$, $k\geq 3$, $\beta >3/2+\gamma $ and let $\theta
_{0},\theta _{1}\geq 6(1-\gamma )(1-2^{-k})$. Assume that $f_{0}\left(
x,\cdot \right) \in L_{\xi ,\hat{\beta}}^{\infty }$ and $f_{0}$ is compactly
supported in the variable $x$ for all $\xi $, namely,
\begin{equation*}
f_{0}\left( x,\xi \right) \equiv 0\text{ for }\left\vert x\right\vert \geq 1%
\text{, }\xi \in \mathbb{R}^{3}\text{,}
\end{equation*}%
where $\hat{\beta}=\max \{\beta +\theta _{0},\beta +2^{k}(\theta _{1}+\gamma
)\}$. If $\eps>0$ is sufficiently small, then the solution $f$ of the
nonlinear problem $(\ref{nonlinear})$ can be decomposed as $f=f_{\ast
}+f_{\sharp }$, where
\begin{align*}
\left\vert f_{\ast }\right\vert _{L_{\xi ,\beta }^{\infty }}& \lesssim \eps %
e^{-ct}\left\langle x\right\rangle ^{-\frac{\theta _{0}}{1-\gamma }} \\
& \quad +\eps^{2}\left[ \left( 1+t\right) ^{-3}B_{3\left( 1-2^{-k}\right)
}(|x|,t)+\left( 1+t\right) ^{-(4-2^{-k})}B_{2(1-2^{-k})}(|x|-\mathbf{c}t,t)%
\right]
\end{align*}%
and
\begin{align*}
\left\vert f_{\sharp }\right\vert _{L_{\xi ,\beta }^{\infty }}& \lesssim \eps%
\left[\Phi_{m}(x,t)+(1+t+|x|)^{-\frac{\sigma }{1-\gamma }}\right]
\\
& \quad +\eps^{2}\left[ \left( 1+t\right) ^{-2}B_{3/2}(|x|,t)+\left(
1+t\right) ^{-\frac{5}{2}}B_{1}(|x|-\mathbf{c}t,t)\right] \text{,}
\end{align*}%
here $m$ and $\sigma $ can be any large positive numbers.
\end{theorem}

\begin{remark}
	One can choose $\theta_{0}=2^{k}(\theta _{1}+\gamma)$ and therefore
	$$
	\theta_{1}=2^{-k}\theta_{0}-\gamma\geq 6(1-\gamma)(1-2^{-k}).
	$$
	If we take $k=3$, then $\theta_{0}\geq 42-34\gamma$. This reduces the assumption of the initial condition.
\end{remark}

\subsection{Structure outside the finite Mach number region}

To improve the asymptotic behavior outside the finite Mach number region, we
take the weight function%
\begin{equation*}
w=c_{1}\left[ \delta \left( \left\langle x\right\rangle -Mt\right) \right] ^{%
\frac{p}{1-\gamma }}\left( 1-\chi \left( \frac{\delta \left( \left\langle
x\right\rangle -Mt\right) }{\left\langle \xi \right\rangle _{D}^{1-\gamma }}%
\right) \right) +c_{2}\left\langle \xi \right\rangle _{D}^{p}\chi \left(
\frac{\delta \left( \left\langle x\right\rangle -Mt\right) }{\left\langle
\xi \right\rangle _{D}^{1-\gamma }}\right)
\end{equation*}%
into account again, where $\delta >0$, $p\geq 2$, $D\geq 1$, and $M>0$. The
constants $c_{1}$, $c_{2}>0$ are chosen to meet the requirement $\partial
_{t}w\leq 0$. In \cite{[LinWangLyuWu]}, we have proved that
\begin{equation*}
\left\vert f\left( t,x,\xi \right) \right\vert _{L_{\xi ,\beta
}^{2}}\lesssim \varepsilon \left( 1+t\right) ^{1/2}\left( \left\langle
x\right\rangle +Mt\right) ^{-\frac{p}{1-\gamma }}\left\Vert f_{0}\right\Vert
_{L_{\xi ,p+\beta +\gamma /2}^{\infty }L_{x}^{\infty }}\text{ for }%
\left\vert x\right\vert >2Mt\text{,}
\end{equation*}%
where $\delta >0$ sufficiently small, $M$, $D>1$ sufficiently large, and $%
p\geq 2$. The largeness of $M$ arises from the regularization estimate of
the remainder. In fact, we find out that the constant $M$ is not needed to
be sufficiently large, and we can choose $\mathbf{c}<M<\mathbf{c+}1$ to be
as close to $\mathbf{c}$ as we desire, by more exquisite calculation for the
regularization estimate of the remainder. In the following, we shall sketch
the essential point of the proof and the rest is just followed line by line
as those in \cite{[LinWangLyuWu]}. Our improvement is stated as below:

\begin{theorem}\label{thm:space-asy}
Let $0<\leq\gamma <1$ and let $f$ be a solution to the Boltzmann equation $%
\left( 1.3\right) $ with initial data $\varepsilon f_{0}$. If $\left\vert
f_{0}\right\vert _{L_{x}^{\infty }}\in L_{\xi ,\hat{\beta }}^{\infty }$, where $\hat{\beta }=p+\beta +\gamma /2$
for $p\geq 2$ and $\beta >3/2$, then for any $\delta >0$, there exists a
constant $M=\left( \mathbf{c}+\delta /2\right) $ such that
\begin{equation*}
\left\vert f\right\vert _{L_{\xi ,\beta }^{\infty }}\leq C_{\delta
}\varepsilon \left( 1+t\right) ^{1/2}\left( \left\langle x\right\rangle
+t\right) ^{-\frac{p}{1-\gamma }}\left\Vert f_{0}\right\Vert _{L_{\xi
,\hat{\beta }}^{\infty }L_{x}^{\infty }}
\end{equation*}%
when $\left\vert x\right\vert >\left( M+\frac{\delta }{2}\right) t=\left(
\mathbf{c}+\delta \right) t$. The constant $C_{\delta }>0$ is independent of time and $C_{\delta
}\rightarrow \infty $ as $\delta \rightarrow 0$.
\end{theorem}

Now let $u=wf$ and we do a wave-remainder decomposition for the weighted
equation%
\begin{equation*}
\partial _{t}u+\xi \cdot \nabla _{x}u-\left( \partial _{t}w+\xi \cdot \nabla
_{x}w\right) w^{-1}u=L_{w}u+\Gamma _{w}\left( u,f\right)
\end{equation*}%
as follows:
\begin{equation*}
\left\{
\begin{array}{l}
\partial _{t}u^{\left( 0\right) }+\xi \cdot \nabla _{x}u^{\left( 0\right) }+%
\widetilde{\nu }\left( \xi \right) u^{\left( 0\right) }=\Gamma _{w}\left(
u,f\right) \text{,}%
\vspace {3mm}
\\
u^{\left( 0\right) }\left( 0,x,\xi \right) =\varepsilon f_{0w}\text{,}%
\end{array}%
\right.
\end{equation*}%
and%
\begin{equation*}
\left\{
\begin{array}{l}
\partial _{t}u^{\left( j\right) }+\xi \cdot \nabla _{x}u^{\left( j\right) }+%
\widetilde{\nu }\left( \xi \right) u^{\left( j\right) }=K_{w}u^{\left(
j-1\right) }%
\vspace {3mm}
\\
u^{\left( j\right) }\left( 0,x,\xi \right) =0%
\end{array}%
\right.
\end{equation*}%
for $j\geq 1$, where
\begin{equation*}
\widetilde{\nu }\left( \xi \right) =-\left( \partial _{t}w+\xi \cdot \nabla
_{x}w\right) w^{-1}+\nu \left( \xi \right) \text{,}
\end{equation*}%
$L_{w}h:=wLw^{-1}h=\nu \left( \xi \right) h+K_{w}h$, $K_{w}\left( x,\xi
\right) :=w\left( t,x,\xi \right) Kw^{-1}\left( t,x,\xi _{\ast }\right) $, $%
\Gamma _{w}\left( h,f\right) :=w\Gamma \left( w^{-1}h,f\right) $. After
choosing $\delta >0$ sufficiently small, we have%
\begin{equation*}
\widetilde{\nu }\left( \xi \right) =-\left( \partial _{t}w+\xi \cdot \nabla
_{x}w\right) w^{-1}+\nu \left( \xi \right) \geq \frac{3}{4}\nu \left( \xi
\right) \text{.}
\end{equation*}%
Let $\hat{N}\geq 7$ be fixed. The wave part and the remainder part are
defined as
\begin{equation*}
\mathcal{W}_{w}^{(\hat{N}) }=\sum_{j=0}^{\hat{N}%
}u^{\left( j\right) }\text{, \ \ \ }\mathcal{R}_{w}^{( \hat{N}%
) }=u-\mathcal{W}^{(\hat{N}) }\text{,}
\end{equation*}%
with $\mathcal{R}_{w}^{(\hat{N}) }$ solving the equation%
\begin{equation*}
\left\{
\begin{array}{l}
\partial _{t}\mathcal{R}_{w}^{(\hat{N}) }+\xi \cdot \nabla
_{x}\mathcal{R}_{w}^{(\hat{N}) }-\left( \partial _{t}w+\xi
\cdot \nabla _{x}w\right) w^{-1}\mathcal{R}_{w}^{(\hat{N})
}=L_{w}\mathcal{R}_{w}^{(\hat{N}) }+K_{w}u^{\left( \hat{N%
}\right) }%
\vspace {3mm}
\\
\mathcal{R}_{w}^{(\hat{N}) }\left( 0,x,\xi \right) =0\text{.}%
\end{array}%
\right.
\end{equation*}%
For convenience of notations, let $g=\mathcal{R}_{w}^{( \hat{N}%
) }$ and $U=K_{w}u^{(\hat{N}) }$. Let $T>0$ be any
finite number and $\beta >3/2$. Define%
\begin{equation*}
\sup_{0\leq t\leq T}\left( 1+t\right) ^{-A}\left\Vert u\right\Vert _{L_{\xi
,\beta }^{\infty }L_{x}^{\infty }}=C_{u,T}^{\infty }\text{, }\sup_{0\leq
t\leq T}\left\Vert u\right\Vert _{L_{\xi ,\beta }^{\infty
}L_{x}^{2}}=C_{u,T}^{2}\text{,}
\end{equation*}%
\begin{equation*}
\sup_{t\geq 0}\left( 1+t\right) ^{-3/2}\left\Vert \left\langle \xi
\right\rangle ^{p}f\right\Vert _{L_{\xi ,\beta }^{\infty }L_{x}^{\infty
}}=C_{f}^{\infty }\text{,}
\end{equation*}%
where the constant $A>0$ will be eventually determined via the a prior
estimate. (In fact, we have obtained $A=1/2$ independent of time $T$ in \cite%
{[LinWangLyuWu]}.)

To do the regularization estimate of the remainder, we need some estimates
for the wave part and the weighted operator $K_{w}$ as below.

\begin{lemma}[$L_{\protect\xi ,\protect\beta }^{\infty }L_{x}^{\infty }$%
-estimate of $u^{(j)}$]
\label{Wave part}Let $p\geq 2$ and $f_{0w}\in L_{\xi ,\beta }^{\infty
}L_{x}^{\infty }$ with $\beta >3/2$. Then for $0\leq t\leq T$,%
\begin{equation*}
\left\Vert u^{\left( j\right) }\right\Vert _{L_{\xi ,\beta }^{\infty
}L_{x}^{\infty }}\lesssim \varepsilon \frac{t^{j}}{j!}\exp \left( -\frac{\nu
_{0}t}{2}\right) \left\Vert f_{0w}\right\Vert _{L_{\xi ,\beta }^{\infty
}L_{x}^{\infty }}+\left( 1+t\right) ^{-1}C_{u,T}^{\infty }C_{f}^{\infty }%
\text{.}
\end{equation*}
\end{lemma}

\begin{lemma}[$L^{2}$-estimate of $u^{(j)}$, $j\geq 1$]
Let $p\geq 2$ and $f_{0w}\in L^{2}$. Then for $0\leq t\leq T$,%
\begin{equation*}
\left\Vert u^{\left( j\right) }\right\Vert _{L^{2}}\lesssim \left(
1+t\right) ^{-\frac{3}{2}}\left[ \varepsilon \left\Vert f_{0w}\right\Vert
_{L_{\xi ,\beta }^{\infty }L_{x}^{2}}+C_{u,T}^{2}C_{f}^{\infty }\right]
\text{.}
\end{equation*}
\end{lemma}

\begin{lemma}[Regularization estimate of $u^{(\hat{N}) }$]
\begin{equation*}
\left \Vert u^{(\hat{N}) }\right \Vert _{L_{\xi
}^{2}H_{x}^{2}}\lesssim \left( 1+t\right) ^{-\frac{3}{2}}\left[ \varepsilon
\left \Vert f_{0w}\right \Vert _{L^{2}}+C_{u}^{2}C_{f}^{\infty }\right]
\text{.}
\end{equation*}
\end{lemma}

\begin{lemma}[{\protect\cite{[LinWangWu3]}}]
\label{Lemma for Kw-K}Let $p\geq 2$ and $D\geq 1$. Define
\begin{align*}
H_{+} &=\{\left( x,\xi \right) :\delta \left( \left\langle x\right\rangle
-Mt\right) >2\left\langle \xi \right\rangle _{D}^{1-\gamma }\}\text{,}%
\vspace {3mm}
\\
H_{0} &=\{\left( x,\xi \right) :\left\langle \xi \right\rangle
_{D}^{1-\gamma }\leq \delta \left( \left\langle x\right\rangle -Mt\right)
\leq 2\left\langle \xi \right\rangle _{D}^{1-\gamma }\}\text{,}%
\vspace {3mm}
\\
H_{-} &=\{\left( x,\xi \right) :\delta \left( \left\langle x\right\rangle
-Mt\right) <2\left\langle \xi \right\rangle _{D}^{1-\gamma }\}\text{.}
\end{align*}%
There exists a constant $C=C\left( \delta ,M,\gamma \right) >0$ such that
\begin{align*}
&\int_{\mathbb{R}^{3}}\left\langle g,\left( K_{w}-K\right) g\right\rangle
_{\xi }dx \\
&\leq \frac{C}{D^{2}}\left[ \int_{\mathbb{R}^{3}}\left\langle \xi
\right\rangle ^{\gamma }\left\vert P_{1}g\right\vert ^{2}d\xi dx+\int_{H_{+}}%
\left[ \delta \left( \left\langle x\right\rangle -Mt\right) \right]
^{-1}\left\vert P_{0}g\right\vert ^{2}d\xi dx+\int_{H_{0}\cup
H_{-}}\left\vert P_{0}g\right\vert ^{2}d\xi dx\right] \text{.}
\end{align*}
\end{lemma}

We are ready to do an energy estimate for the functional
\begin{equation*}
E\left( t\right) =\frac{1}{2}\left( \left\Vert g\right\Vert _{L^{2}}+\alpha
\left\Vert \nabla _{x}g\right\Vert _{L^{2}}+\alpha ^{2}\left\Vert \nabla
_{x}^{2}g\right\Vert _{L^{2}}\right) \text{, }\alpha =O\left( \delta
^{2}\right) \text{,}
\end{equation*}%
with the equation

\begin{equation*}
\left\{
\begin{array}{l}
\partial _{t}g+\xi \cdot \nabla _{x}g-\left( \partial _{t}w+\xi \cdot \nabla
_{x}w\right) w^{-1}g=Lg+\left( K_{w}-K\right) g+U%
\vspace {3mm}
\\
g\left( 0,x,\xi \right) =0\text{.}%
\end{array}%
\right.
\end{equation*}%
Direct computations show
\begin{align*}
&\partial _{t}\left( \partial _{x_{j}}g\right) +\xi \cdot \nabla _{x}\left(
\partial _{x_{j}}g\right) -\left( \partial _{t}w+\xi \cdot \nabla
_{x}w\right) w^{-1}\partial _{x_{j}}g \\
&=L\left( \partial _{x_{j}}g\right) +\left( K_{w}-K\right) \partial
_{x_{j}}g+\partial _{x_{j}}U+\mathbb{A}_{1}\left( t,x,\xi \right) \text{,}
\end{align*}%
\begin{align*}
&\partial _{t}\left( \partial _{x_{i}x_{j}}^{2}g\right) +\xi \cdot \nabla
_{x}\left( \partial _{x_{i}x_{j}}^{2}g\right) -\left( \partial _{t}w+\xi
\cdot \nabla _{x}w\right) w^{-1}\partial _{x_{i}x_{j}}^{2}g \\
&=L\left( \partial _{x_{i}x_{j}}^{2}g\right) +\left( K_{w}-K\right)
\partial _{x_{i}x_{j}}^{2}g+\partial _{x_{i}x_{j}}^{2}U+\mathbb{A}_{2}\left(
t,x,\xi \right) \text{,}
\end{align*}%
$1\leq i$, $j\leq 3$, where
\begin{equation*}
\mathbb{A}_{1}\left( t,x,\xi \right) =\frac{\partial _{x_{j}}w}{w}%
K_{w}g+K_{w}\left( -\frac{\partial _{x_{j}}w}{w}g\right) +\partial _{x_{j}}%
\left[ \left( \partial _{t}w+\xi \cdot \nabla _{x}w\right) w^{-1}\right] g%
\text{,}
\end{equation*}%
\begin{align*}
\mathbb{A}_{2}\left( t,x,\xi \right) & =\partial _{x_{i}}\left[ \left(
\partial _{t}w+\xi \cdot \nabla _{x}w\right) w^{-1}\right] \cdot \partial
_{x_{j}}g+\partial _{x_{i}x_{j}}^{2}\left[ \left( \partial _{t}w+\xi \cdot
\nabla _{x}w\right) w^{-1}\right] \cdot g \\
&+\partial _{x_{j}}\left[ \left( \partial _{t}w+\xi \cdot \nabla
_{x}w\right) w^{-1}\right] \cdot \partial _{x_{i}}g+\frac{\partial _{x_{i}}w%
}{w}K_{w}\partial _{x_{j}}g+K_{w}\left( -\frac{\partial _{x_{i}}w}{w}%
\partial _{x_{j}}g\right) \\
&+\partial _{x_{i}}\left( \frac{\partial _{x_{j}}w}{w}\right) K_{w}g+\frac{%
\partial _{x_{j}}w}{w}\frac{\partial _{x_{i}}w}{w}K_{w}g+\frac{\partial
_{x_{j}}w}{w}K_{w}\left( -\frac{\partial _{x_{i}}w}{w}g\right) +\frac{%
\partial _{x_{j}}w}{w}K_{w}\left( \partial _{x_{i}}g\right) \\
&+\frac{\partial _{x_{i}}w}{w}K_{w}\left( -\frac{\partial _{x_{j}}w}{w}%
g\right) +K_{w}\left( \frac{\partial _{x_{i}}w}{w}\frac{\partial _{x_{j}}w}{w%
}g\right) +K_{w}\left( -\partial _{x_{i}}\left( \frac{\partial _{x_{j}}w}{w}%
\right) g\right) +K_{w}\left( -\frac{\partial _{x_{j}}w}{w}\partial
_{x_{i}}g\right) \text{.}
\end{align*}%
We compute $\frac{d}{dt}\frac{1}{2}\left\Vert g\right\Vert _{L^{2}}^{2}$
first,
\begin{align*}
&\frac{d}{dt}\frac{1}{2}\left\Vert g\right\Vert _{L^{2}}^{2}-\int_{\mathbb{R%
}^{3}}\left\langle Lg,g\right\rangle _{\xi }dx-\int_{\mathbb{R}%
^{3}}\left\langle g,\left( K_{w}-K\right) g\right\rangle _{\xi }dx \\
&=\int_{\mathbb{R}^{3}}\left\langle \left( \partial _{t}w+\xi \cdot \nabla
_{x}w\right) w^{-1}g,g\right\rangle d\xi dx+\int_{\mathbb{R}^{3}\times
\mathbb{R}^{3}}Ugd\xi dx\text{.}
\end{align*}%
It readily follows from \cite[Lemma 4]{[LinWangWu3]} that
\begin{equation*}
\left\vert \int_{\mathbb{R}^{3}\times \mathbb{R}^{3}}Ugd\xi dx\right\vert
\leq C\left\Vert g\right\Vert _{L^{2}}\left\Vert u^{\left( \hat{N}%
\right) }\right\Vert _{L^{2}}\text{.}
\end{equation*}%
Since%
\begin{align*}
\left( \partial _{t}w+\xi \cdot \nabla _{x}w\right) w^{-1} &=\frac{-\delta p%
}{1-\gamma }\left( M-\frac{\xi \cdot x}{\left\langle x\right\rangle }\right) %
\left[ \delta \left( \left\langle x\right\rangle -Mt\right) \right]
^{-1}\left( 1-\chi \left( \frac{\delta \left( \left\langle x\right\rangle
-Mt\right) }{\left\langle \xi \right\rangle _{D}^{1-\gamma }}\right) \right)
\\
&+w^{-1}\left( M-\frac{\xi \cdot x}{\left\langle x\right\rangle }\right)
\frac{\delta }{\left\langle \xi \right\rangle _{D}^{1-\gamma }}c_{1}\left[
\delta \left( \left\langle x\right\rangle -Mt\right) \right] ^{\frac{p}{%
1-\gamma }}\chi ^{\prime }\left( \frac{\delta \left( \left\langle
x\right\rangle -Mt\right) }{\left\langle \xi \right\rangle _{D}^{1-\gamma }}%
\right) \\
&-w^{-1}\left( M-\frac{\xi \cdot x}{\left\langle x\right\rangle }\right)
\frac{\delta }{\left\langle \xi \right\rangle _{D}^{1-\gamma }}%
c_{2}\left\langle \xi \right\rangle _{D}^{p}\chi ^{\prime }\left( \frac{%
\delta \left( \left\langle x\right\rangle -Mt\right) }{\left\langle \xi
\right\rangle _{D}^{1-\gamma }}\right) \text{,}
\end{align*}%
we have%
\begin{align*}
&\int_{\mathbb{R}^{3}}\left\langle \left( \partial _{t}w+\xi \cdot \nabla
_{x}w\right) w^{-1}g,g\right\rangle _{\xi }dx \\
&\leq \frac{-\delta p}{1-\gamma }\int_{H_{+}}\left( M-\frac{\xi \cdot x}{%
\left\langle x\right\rangle }\right) \left[ \delta \left( \left\langle
x\right\rangle -Mt\right) \right] ^{-1}\left\vert g\right\vert ^{2}dxd\xi
+C_{1}\delta \left( \frac{M}{D}+1\right) \frac{p}{1-\gamma }%
\int_{H_{0}}\left\langle \xi \right\rangle _{D}^{\gamma }\left\vert
g\right\vert ^{2}dxd\xi \\
&\leq \frac{-\delta p}{1-\gamma }\int_{H_{+}}\left( M-\frac{\xi \cdot x}{%
\left\langle x\right\rangle }\right) \left[ \delta \left( \left\langle
x\right\rangle -Mt\right) \right] ^{-1}\left\vert \mathrm{P}_{0}g\right\vert
^{2}dxd\xi \\
&\quad +\frac{\delta p}{1-\gamma }\int_{H_{+}}\left\vert M-\frac{\xi \cdot x}{%
\left\langle x\right\rangle }\right\vert \left[ \delta \left( \left\langle
x\right\rangle -Mt\right) \right] ^{-1}\left( \hat{\epsilon }\left\vert
\mathrm{P}_{0}g\right\vert ^{2}+\frac{1}{4\hat{\epsilon }}\left\vert
\mathrm{P}_{1}g\right\vert ^{2}+\left\vert \mathrm{P}_{1}g\right\vert
^{2}\right) dxd\xi \\
&\quad +2C_{1}\delta \left( \frac{M}{D}+1\right) \frac{p}{1-\gamma }%
\int_{H_{0}}\left\langle \xi \right\rangle _{D}^{\gamma }\left( \left\vert
\mathrm{P}_{0}g\right\vert ^{2}+\left\vert \mathrm{P}_{1}g\right\vert
^{2}\right) dxd\xi \\
&\leq \frac{-\delta p}{1-\gamma }\int_{H_{+}}\left( M-\frac{\xi \cdot x}{%
\left\langle x\right\rangle }\right) \left[ \delta \left( \left\langle
x\right\rangle -Mt\right) \right] ^{-1}\left\vert \mathrm{P}_{0}g\right\vert
^{2}dxd\xi \\
&\quad +\frac{\delta pM\hat{\epsilon }}{1-\gamma }\int_{H_{+}}\left(
1+\left\vert \xi \right\vert \right) \left[ \delta \left( \left\langle
x\right\rangle -Mt\right) \right] ^{-1}\left\vert \mathrm{P}_{0}g\right\vert
^{2}dxd\xi \\
&\quad +2C_{1}\delta \left( \frac{M}{D}+1\right) D^{\gamma }\frac{p}{1-\gamma }%
\int_{H_{0}}\left\langle \xi \right\rangle ^{\gamma }\left\vert \mathrm{P}%
_{0}g\right\vert ^{2}dxd\xi \\
&\quad +\frac{2\delta p}{1-\gamma }\left( \frac{M}{D}+1\right) D^{\gamma }\left(
C_{1}+1+\frac{1}{4\hat{\epsilon }}\right) \int_{\mathbb{R}^{3}\times
\mathbb{R}^{3}}\left\vert \left\langle \xi \right\rangle ^{\frac{\gamma }{2}}%
\mathrm{P}_{1}g\right\vert ^{2}dxd\xi \text{.}
\end{align*}%
It is easy to see
\begin{align*}
&\int_{H_{+}}\left( 1+\left\vert \xi \right\vert \right) \left[ \delta
\left( \left\langle x\right\rangle -Mt\right) \right] ^{-1}\left\vert
\mathrm{P}_{0}g\right\vert ^{2}dxd\xi \\
&\leq \int_{\delta \left( \left\langle x\right\rangle -Mt\right) \geq
2D^{1-\gamma }}\left[ \delta \left( \left\langle x\right\rangle -Mt\right) %
\right] ^{-1}\int_{\delta \left( \left\langle x\right\rangle -Mt\right) \geq
2\left\langle \xi \right\rangle _{D}^{1-\gamma }}\left( 1+\left\vert \xi
\right\vert \right) \left( \sum_{j=0}^{4}\left\langle g,\chi
_{j}\right\rangle ^{2}\right) \left( \sum_{j=0}^{4}\chi _{j}^{2}\right)
dxd\xi \\
&\leq \int_{\delta \left( \left\langle x\right\rangle -Mt\right) \geq
2D^{1-\gamma }}\left[ \delta \left( \left\langle x\right\rangle -Mt\right) %
\right] ^{-1}\left( \sum_{j=0}^{4}\left\langle g,\chi _{j}\right\rangle
^{2}\right) dx \\
&\leq \int_{H_{+}}\left[ \delta \left( \left\langle x\right\rangle
-Mt\right) \right] ^{-1}\left\vert \mathrm{P}_{0}g\right\vert ^{2}dxd\xi +%
\frac{1}{2D^{1-\gamma }}\int_{H_{0}\cup H_{-}}\left\vert \mathrm{P}%
_{0}g\right\vert ^{2}dxd\xi \text{.}
\end{align*}%
In view of (\ref{Eigen function}), we have
\begin{equation*}
\mathrm{P}_{0}\left( \frac{\xi \cdot x}{\left\vert x\right\vert }\right)
\mathrm{P}_{0}g=\sum_{j=0}^{4}a_{j}\left\langle g,E_{j}\right\rangle E_{j}%
\text{,}
\end{equation*}%
and thus for $M>\mathbf{c}$ (\textit{this is the key point to the improvement}),
\begin{align*}
&\frac{-\delta p}{1-\gamma }\int_{H_{+}}\left( M-\frac{\xi \cdot x}{%
\left\langle x\right\rangle }\right) \left[ \delta \left( \left\langle
x\right\rangle -Mt\right) \right] ^{-1}\left\vert \mathrm{P}_{0}g\right\vert
^{2}dxd\xi \\
&=\frac{-\delta p}{1-\gamma }\int_{\delta \left( \left\langle
x\right\rangle -Mt\right) \geq 2D^{1-\gamma }}\int_{\delta \left(
\left\langle x\right\rangle -Mt\right) \geq 2\left\langle \xi \right\rangle
_{D}^{1-\gamma }}\left( M-\frac{\xi \cdot x}{\left\langle x\right\rangle }%
\right) \left[ \delta \left( \left\langle x\right\rangle -Mt\right) \right]
^{-1}\left\vert \mathrm{P}_{0}g\right\vert ^{2}d\xi dx \\
&=\frac{-\delta p}{1-\gamma }\int_{\delta \left( \left\langle
x\right\rangle -Mt\right) \geq 2D^{1-\gamma }}\int_{\mathbb{R}^{3}}\left( M-%
\frac{\xi \cdot x}{\left\langle x\right\rangle }\right) \left[ \delta \left(
\left\langle x\right\rangle -Mt\right) \right] ^{-1}\left\vert \mathrm{P}%
_{0}g\right\vert ^{2}d\xi dx \\
&\quad +\frac{\delta p}{1-\gamma }\int_{\delta \left( \left\langle x\right\rangle
-Mt\right) \geq 2D^{1-\gamma }}\int_{\delta \left( \left\langle
x\right\rangle -Mt\right) \leq 2\left\langle \xi \right\rangle
_{D}^{1-\gamma }}\left( M-\frac{\xi \cdot x}{\left\langle x\right\rangle }%
\right) \left[ \delta \left( \left\langle x\right\rangle -Mt\right) \right]
^{-1}\left\vert \mathrm{P}_{0}g\right\vert ^{2}d\xi dx \\
&\leq \frac{-\delta p}{1-\gamma }\int_{\delta \left( \left\langle
x\right\rangle -Mt\right) \geq 2D^{1-\gamma }}\int_{\mathbb{R}^{3}}\left( M-%
\frac{\xi \cdot x}{\left\langle x\right\rangle }\right) \left[ \delta \left(
\left\langle x\right\rangle -Mt\right) \right] ^{-1}\left\vert \mathrm{P}%
_{0}g\right\vert ^{2}d\xi dx \\
&\quad +\frac{\delta p}{1-\gamma }\int_{\delta \left( \left\langle x\right\rangle
-Mt\right) \geq 2D^{1-\gamma }}\int_{\delta \left( \left\langle
x\right\rangle -Mt\right) \leq 2\left\langle \xi \right\rangle
_{D}^{1-\gamma }}\frac{\left( M+\left\vert \xi \right\vert \right) }{%
2D^{1-\gamma }}\left\vert \mathrm{P}_{0}g\right\vert ^{2}d\xi dx \\
&\leq \frac{-\delta p}{1-\gamma }\int_{\delta \left( \left\langle
x\right\rangle -Mt\right) \geq 2D^{1-\gamma }}\left[ \delta \left(
\left\langle x\right\rangle -Mt\right) \right] ^{-1}\int_{\mathbb{R}%
^{3}}\left( M-\frac{\xi \cdot x}{\left\langle x\right\rangle }\right)
\left\vert \mathrm{P}_{0}g\right\vert ^{2}d\xi dx \\
&\quad +\frac{\delta pM}{2D^{1-\gamma }\left( 1-\gamma \right) }\int_{H_{0}\cup
H_{-}}\left( 1+\left\vert \xi \right\vert \right) \left\vert
P_{0}g\right\vert ^{2}d\xi dx \\
&=\frac{-\delta p}{1-\gamma }\int_{\delta \left( \left\langle
x\right\rangle -Mt\right) \geq 2D^{1-\gamma }}\left[ \delta \left(
\left\langle x\right\rangle -Mt\right) \right] ^{-1}\left( M\left\vert
\mathrm{P}_{0}g\right\vert _{L_{\xi }^{2}}^{2}-\frac{\left\vert x\right\vert
}{\left\langle x\right\rangle }\sum_{j=0}^{4}a_{j}\left\langle
g,E_{j}\right\rangle ^{2}\right) dx \\
&\quad +\frac{\delta pM}{2D^{1-\gamma }\left( 1-\gamma \right) }\int_{H_{0}\cup
H_{-}}\left( 1+\left\vert \xi \right\vert \right) \left\vert \mathrm{P}%
_{0}g\right\vert ^{2}d\xi dx \\
&\leq \frac{-\delta p}{1-\gamma }\int_{\delta \left( \left\langle
x\right\rangle -Mt\right) \geq 2D^{1-\gamma }}\left[ \delta \left(
\left\langle x\right\rangle -Mt\right) \right] ^{-1}\left( M-\mathbf{c}%
\right) \left\vert \mathrm{P}_{0}g\right\vert _{L_{\xi }^{2}}^{2}dx \\
&\quad +\frac{\delta pM}{2D^{1-\gamma }\left( 1-\gamma \right) }\int_{H_{0}\cup
H_{-}}\left( 1+\left\vert \xi \right\vert \right) \left\vert \mathrm{P}%
_{0}g\right\vert ^{2}d\xi dx \\
&\leq \frac{-\delta p}{1-\gamma }\left( M-\mathbf{c}\right) \int_{H_{+}}%
\left[ \delta \left( \left\langle x\right\rangle -Mt\right) \right]
^{-1}\left\vert \mathrm{P}_{0}g\right\vert ^{2}d\xi dx+\frac{\delta pM}{%
2D^{1-\gamma }\left( 1-\gamma \right) }\int_{H_{0}\cup H_{-}}\left(
1+\left\vert \xi \right\vert \right) \left\vert \mathrm{P}_{0}g\right\vert
^{2}d\xi dx\text{.}
\end{align*}%
Together with Lemma \ref{Lemma
for Kw-K}, we find for $M\geq \mathbf{c}+\delta $,
\begin{align*}
&\frac{d}{dt}\frac{1}{2}\left\Vert g\right\Vert _{L^{2}}^{2} \\
&\leq \int_{\mathbb{R}^{3}}\left\langle \left( \partial _{t}w+\xi \cdot
\nabla _{x}w\right) w^{-1}g,g\right\rangle _{\xi }dx+\int_{\mathbb{R}%
^{3}}\left\langle Lg,g\right\rangle _{\xi }dx+\int_{\mathbb{R}%
^{3}}\left\langle g,\left( K_{w}-K\right) g\right\rangle _{\xi }dx+\int_{%
\mathbb{R}^{3}\times \mathbb{R}^{3}}Ugdxd\xi \\
&\leq \frac{-\delta p}{1-\gamma }\left[ M-\mathbf{c-}M\hat{\epsilon }-%
\frac{C\left( 1-\gamma \right) }{D^{2}\delta p}\right] \int_{H_{+}}\left[
\delta \left( \left\langle x\right\rangle -Mt\right) \right] ^{-1}\left\vert
\mathrm{P}_{0}g\right\vert ^{2}d\xi dx \\
&\quad -\left[ \mu -\frac{2\delta p}{1-\gamma }\left( \frac{M}{D}+1\right)
D^{\gamma }\left( C_{1}+1+\frac{1}{4\hat{\epsilon }}\right) -\frac{C}{%
D^{2}}\right] \int_{\mathbb{R}^{3}\times \mathbb{R}^{3}}\left\vert
\left\langle \xi \right\rangle ^{\frac{\gamma }{2}}\mathrm{P}%
_{1}g\right\vert ^{2}dxd\xi \text{.} \\
&\quad +\left[ \frac{\delta p}{1-\gamma }\left( \frac{M}{2D^{1-\gamma }}+\frac{M%
\hat{\epsilon }}{2D^{1-\gamma }}+2C_{1}\left( \frac{M}{D}+1\right)
D^{\gamma }\right) +\frac{C}{D^{2}}\right] \int_{H_{0}\cup H_{-}}\left(
1+\left\vert \xi \right\vert \right) \left\vert \mathrm{P}_{0}g\right\vert
^{2}d\xi dx \\
&\quad +C\left\Vert g\right\Vert _{L^{2}}\left\Vert u^{(\hat{N})
}\right\Vert _{L^{2}}\text{.}
\end{align*}%
It is similar for $\left\Vert \nabla _{x}g\right\Vert _{L^{2}}^{2}$ and $%
\left\Vert \nabla _{x}^{2}g\right\Vert _{L^{2}}^{2}$. The rest of proofs for
$\lVert \mathcal{R}_{w}^{( \hat{N}) }\rVert _{L_{\xi
}^{2}H_{x}^{2}}$ and $\lVert \mathcal{R}_{w}^{( \hat{N})
}\rVert _{L_{\xi ,\beta }^{\infty }L_{x}^{\infty }}$ is all the same as
those in \cite{[LinWangLyuWu]} and we omit the details here. Therefore,
after choosing $\hat{\epsilon }>0$ sufficiently small, $D\geq 1$
sufficiently large, $\delta >0$ sufficiently small, together with the wave
part (Lemma \ref{Wave part}), we can conclude that for any $\delta >0$, $%
p\geq 2$, $\beta >3/2$, there exist positive constants $D\geq 1$
sufficiently large and $M=\mathbf{c}+\delta $ for the choice of weight
function $w$ such that
\begin{equation*}
\sup_{0\leq t}\left( 1+t\right) ^{-1/2}\left\Vert u\right\Vert _{L_{\xi
,\beta }^{\infty }L_{x}^{\infty }}\leq C_{\delta }\varepsilon \left\Vert
\left\langle \xi \right\rangle ^{p+\beta +\gamma /2}f_{0}\right\Vert
_{L_{\xi }^{\infty }L_{x}^{\infty }}\text{,}
\end{equation*}%
or equivalently,
\begin{equation*}
\left\Vert wf\right\Vert _{L_{\xi ,\beta }^{\infty }L_{x}^{\infty
}}=\left\Vert u\right\Vert _{L_{\xi ,\beta }^{\infty }L_{x}^{\infty }}\leq
C_{\delta }\varepsilon \left( 1+t\right) ^{1/2}\left\Vert \left\langle \xi
\right\rangle ^{p+\beta +\gamma /2}f_{0}\right\Vert _{L_{\xi }^{\infty
}L_{x}^{\infty }}\text{, }t\geq 0\text{,}
\end{equation*}%
where $C_{\delta }>0$ is a constant independent of time with $C_{\delta
}\rightarrow \infty $ as $\delta \rightarrow 0$.

Note that for $\left\vert x\right\vert >\left( M+\delta \right) t=\left(
\mathbf{c}+2\delta \right) t$, we have%
\begin{equation*}
\left\langle x\right\rangle -Mt=\frac{\frac{\delta }{2M}\left\langle
x\right\rangle }{2+\frac{\delta }{2M}}+\frac{2\left\langle x\right\rangle }{%
2+\frac{\delta }{2M}}-Mt>\frac{\frac{\delta }{2M}\left\langle x\right\rangle
}{2+\frac{\delta }{2M}}+\frac{\frac{\delta }{2}t}{2+\frac{\delta }{2M}}\geq
\frac{\delta \left( \left\langle x\right\rangle +t\right) }{4M+2\delta }%
\text{,}
\end{equation*}%
so that
\begin{equation*}
\left\vert f\right\vert _{L_{\xi }^{\infty }}\leq C_{\delta }\varepsilon
\left( 1+t\right) ^{1/2}\left( \left\langle x\right\rangle +t\right) ^{-%
\frac{p}{1-\gamma }}\left\Vert \left\langle \xi \right\rangle ^{p+\beta
+\gamma /2}f_{0}\right\Vert _{L_{\xi }^{\infty }L_{x}^{\infty }}\text{, }%
t\geq 0\text{,}
\end{equation*}%
where $C_{\delta }>0$ is a constant independent of time with $C_{\delta
}\rightarrow \infty $ as $\delta \rightarrow 0$, as desired.

\section{Some convolution estimates} \label{sec:convolution}

In this section, we provide the proof of nonlinear wave interactions for $f_{\ast}$ of Lemma \ref{lem:wave-int-*} in Section \ref{sec:nonlinear}.

\begin{lemma}
\begin{align*}
A&=:\int_{0}^{t}\int_{\mathbb{R}^{3}}\left(1+t-\tau\right)^{-m} \left(1+%
\frac{|x-y|^{2}}{1+t-\tau}\right)^{-N}e^{-c\tau}\left(1+|y|\right)^{-\hat{P}%
}dyd\tau \\
&\lesssim\left(1+t\right)^{-m}\left(1+\frac{|x|^{2}}{1+t}\right)^{-N}\,.
\end{align*}
\end{lemma}

\begin{proof}
We prove it for two cases: \newline
\newline
\textbf{Case 1:} When $|x|\leq 3\sqrt{t+1}$,
\begin{equation*}
A\lesssim \int_{0}^{t}\left( 1+t-\tau \right) ^{-m}e^{-c\tau }d\tau \lesssim
(1+t)^{-m}\lesssim \left( 1+t\right) ^{-m}\left( 1+\frac{|x|^{2}}{1+t}%
\right) ^{-N}\,.
\end{equation*}%
\textbf{Case 2:} When $|x|>3\sqrt{t+1}$,
\begin{align*}
A& =\int_{0}^{t}\int_{|y|<|x|/2}+\int_{0}^{t}\int_{|y|>|x|/2} \\
& \lesssim \left( 1+\frac{|x|^{2}}{1+t}\right) ^{-N}\int_{0}^{t}\left(
1+t-\tau \right) ^{-m}e^{-c\tau }d\tau \\
& \quad +\int_{0}^{t}\left( 1+t-\tau \right) ^{-m+\frac{3}{2}}e^{-c\tau
}d\tau \left( 1+|x|+\sqrt{t+1}\right) ^{-\hat{P}} \\
& \lesssim \left( 1+t\right) ^{-m}\left( 1+\frac{|x|^{2}}{1+t}\right)
^{-N}+\left( 1+t\right) ^{-m+\frac{3}{2}}\left( t+1+|x|^{2}\right) ^{-\frac{%
\hat{P}}{2}} \\
& \lesssim \left( 1+t\right) ^{-m}\left( 1+\frac{|x|^{2}}{1+t}\right)
^{-N}+\left( 1+t\right) ^{-m-\frac{\hat{P}}{2}+\frac{3}{2}}\left( 1+\frac{%
|x|^{2}}{1+t}\right) ^{-\frac{\hat{P}}{2}}\,.
\end{align*}
\end{proof}

\begin{lemma}
For any $m\geq 0$, $N>3/2$, $\hat{P}\geq \max \{N+\frac{5}{2}$, $2N$, $%
5\}$, \
\begin{eqnarray*}
I &=&\int_{0}^{t}\int_{\mathbb{R}^{3}}\left( 1+t-\tau \right) ^{-m}\left( 1+%
\frac{\left( \left \vert x-y\right \vert -c\left( t-\tau \right) \right) ^{2}%
}{1+t-\tau }\right) ^{-N}e^{-c_{0}\tau }\left( 1+\left \vert y\right \vert
\right) ^{-\hat{P}}dyd\tau \\
&\lesssim &\left( 1+t\right) ^{-m}\left( 1+\frac{\left( \left \vert x\right
\vert -t\right) ^{2}}{1+t}\right) ^{-N}\text{.}
\end{eqnarray*}
\end{lemma}

\begin{proof}
We may assume $c=1$ and write%
\begin{equation*}
I=\int_{0}^{t}\int_{\mathbb{R}^{3}}e^{-c_{0}\left( t-\tau \right) }\left(
1+\left\vert x-y\right\vert \right) ^{-\hat{P}}\left( 1+\tau \right)
^{-m}\left( 1+\frac{\left( \left\vert y\right\vert -\tau \right) ^{2}}{%
1+\tau }\right) ^{-N}dyd\tau \text{.}
\end{equation*}%
We divide the $xt$ space into five parts
\begin{eqnarray*}
\mathcal{D}_{1} &=&\{\left\vert x\right\vert \leq \sqrt{1+t}\}\text{, \ \ \
\ \ \ \ \ \ \ }\mathcal{D}_{2}=\{\left\vert x-ct\right\vert \leq \sqrt{1+t}\}%
\text{,}%
\vspace {3mm}
\\
\mathcal{D}_{3} &=&\{\left\vert x\right\vert \geq t+\sqrt{1+t}\}\text{, \ \
\ \ \ }\mathcal{D}_{4}=\{\sqrt{1+t}\leq \left\vert x\right\vert \leq \frac{t%
}{2}\}\text{,}%
\vspace {3mm}
\\
\mathcal{D}_{5} &=&\{\frac{t}{2}\leq \left\vert x\right\vert \leq t-\sqrt{1+t%
}\}\text{.}
\end{eqnarray*}

For $\left( x,t\right) \in \mathcal{D}_{1}$ and $t\leq 65$,
\begin{eqnarray*}
I &\lesssim &e^{-\frac{c_{0}t}{2}}\int_{0}^{\frac{t}{2}}\int_{\mathbb{R}%
^{3}}\left( 1+\tau \right) ^{-m}\left( 1+\frac{\left( \left\vert
y\right\vert -\tau \right) ^{2}}{1+\tau }\right) ^{-N}dyd\tau \\
&&+\left( 1+t\right) ^{-m}\int_{\frac{t}{2}}^{t}\int_{\mathbb{R}%
^{3}}e^{-c_{0}\left( t-\tau \right) }\left( 1+\left\vert x-y\right\vert
\right) ^{-\hat{P}}dyd\tau \\
&\lesssim &e^{-\frac{c_{0}t}{4}}+\left( 1+t\right) ^{-m}\lesssim \left(
1+t\right) ^{-m-N}\text{.}
\end{eqnarray*}%
For $\left( x,t\right) \in \mathcal{D}_{1}$ and $t\geq 65$, we have%
\begin{equation*}
\left\vert x-y\right\vert \geq \left\vert y\right\vert -\left\vert
x\right\vert \geq \frac{\tau }{2}-\sqrt{1+t}\geq \frac{t}{4}-\sqrt{1+t}\geq
\frac{t}{8}
\end{equation*}%
when $\left\vert y\right\vert \geq \frac{\tau }{2}$, $\frac{t}{2}\leq \tau
\leq t$, so that
\begin{eqnarray*}
I &\lesssim &e^{-\frac{c_{0}t}{2}}\int_{0}^{\frac{t}{2}}\int_{\mathbb{R}%
^{3}}\left( 1+\left\vert x-y\right\vert \right) ^{-\hat{P}}\left( 1+\tau
\right) ^{-m}dyd\tau \\
&&+\int_{\frac{t}{2}}^{t}\int_{\left\vert y\right\vert <\frac{\tau }{2}%
}e^{-c_{0}\left( t-\tau \right) }\left( 1+\left\vert x-y\right\vert \right)
^{-\hat{P}}\left( 1+t\right) ^{-m-N}dyd\tau \\
&&+\int_{\frac{t}{2}}^{t}\int_{\left\vert y\right\vert \geq \frac{\tau }{2}%
}e^{-c_{0}\left( t-\tau \right) }\left( 1+t\right) ^{-m-\hat{P}}\left( 1+%
\frac{\left( \left\vert y\right\vert -\tau \right) ^{2}}{1+\tau }\right)
^{-N}dyd\tau \\
&\lesssim &e^{-\frac{c_{0}t}{4}}+\left( 1+t\right) ^{-m-N}+\left( 1+t\right)
^{-m+\frac{5}{2}-\hat{P}}\lesssim \left( 1+t\right) ^{-m-N}\text{.}
\end{eqnarray*}%
Thus, for $\left( x,t\right) \in \mathcal{D}_{1}$,

\begin{equation*}
I\lesssim \left( 1+t\right) ^{-m-N}\lesssim \left( 1+t\right) ^{-m}\left[ 1+%
\frac{\left( \left \vert x\right \vert -t\right) ^{2}}{1+t}\right] ^{-N}%
\text{.}
\end{equation*}

For $\left( x,t\right) \in \mathcal{D}_{2}$,
\begin{eqnarray*}
I &\lesssim &\int_{0}^{\frac{t}{2}}\int_{\mathbb{R}^{3}}e^{-c_{0}\left(
t-\tau \right) }\left( 1+\left \vert x-y\right \vert \right) ^{-\hat{P}%
}dyd\tau +\int_{\frac{t}{2}}^{t}\int_{\mathbb{R}^{3}}e^{-c_{0}\left( t-\tau
\right) }\left( 1+\left \vert x-y\right \vert \right) ^{-\hat{P}}\left(
1+t\right) ^{-m}dyd\tau \\
&\lesssim &e^{-\frac{c_{0}t}{3}}+\left( 1+t\right) ^{-m}\lesssim \left(
1+t\right) ^{-m}\lesssim \left( 1+t\right) ^{-m}\left[ 1+\frac{\left( \left
\vert x\right \vert -t\right) ^{2}}{1+t}\right] ^{-N}\text{.}
\end{eqnarray*}

For $\left( x,t\right) \in \mathcal{D}_{3}$, we split the the space domain
into two parts $\{ \left \vert y\right \vert <\frac{\left \vert
x\right
\vert +\tau }{2}\}$ and $\{ \left \vert y\right \vert \geq \frac{%
\left \vert x\right \vert +\tau }{2}\}$. If $\left \vert y\right \vert <%
\frac{\left
\vert x\right \vert +\tau }{2}$, then
\begin{equation*}
\left \vert x-y\right \vert >\left \vert x\right \vert -\frac{\left \vert
x\right \vert +\tau }{2}=\frac{\left \vert x\right \vert -t}{2}+\frac{t-\tau
}{2}\text{;}
\end{equation*}%
if $\left \vert y\right \vert \geq \frac{\left \vert x\right \vert +\tau }{2}
$, then%
\begin{equation*}
\left \vert y\right \vert -\tau \geq \frac{\left \vert x\right \vert +\tau }{%
2}-\tau =\frac{\left \vert x\right \vert -t}{2}+\frac{t-\tau }{2}\text{.}
\end{equation*}%
Therefore,
\begin{eqnarray*}
I &\lesssim &\int_{0}^{\frac{t}{2}}\int_{\left \vert y\right \vert <\frac{%
\left \vert x\right \vert +\tau }{2}}e^{-c_{0}\left( t-\tau \right) }\left(
1+\frac{\left \vert x\right \vert -t}{2}+\frac{t-\tau }{2}\right) ^{-%
\hat{P}}\left( 1+\tau \right) ^{-m}\left( 1+\frac{\left( \left \vert
y\right \vert -\tau \right) ^{2}}{1+\tau }\right) ^{-N}dyd\tau \\
&&+\int_{0}^{\frac{t}{2}}\int_{\left \vert y\right \vert \geq \frac{\left
\vert x\right \vert +\tau }{2}}e^{-c_{0}\left( t-\tau \right) }\left(
1+\left \vert x-y\right \vert \right) ^{-\hat{P}}\left( 1+\tau \right)
^{-m}\left( 1+\frac{\left( \left \vert x\right \vert -t\right) ^{2}+\left(
t-\tau \right) ^{2}}{1+\tau }\right) ^{-N}dyd\tau \\
&&+\int_{\frac{t}{2}}^{t}\int_{\left \vert y\right \vert <\frac{\left \vert
x\right \vert +\tau }{2}}e^{-c_{0}\left( t-\tau \right) }\left( 1+\frac{%
\left \vert x\right \vert -t}{2}+\frac{t-\tau }{2}\right) ^{-\hat{P}%
}\left( 1+t\right) ^{-m}\left( 1+\frac{\left( \left \vert y\right \vert
-\tau \right) ^{2}}{1+\tau }\right) ^{-N}dyd\tau \\
&&+\int_{\frac{t}{2}}^{t}\int_{\left \vert y\right \vert \geq \frac{\left
\vert x\right \vert +\tau }{2}}e^{-c_{0}\left( t-\tau \right) }\left(
1+\left \vert x-y\right \vert \right) ^{-\hat{P}}\left( 1+t\right)
^{-m}\left( 1+\frac{\left( \left \vert x\right \vert -t\right) ^{2}+\left(
t-\tau \right) ^{2}}{1+t}\right) ^{-N}dyd\tau \\
&\lesssim &e^{-\frac{c_{0}t}{3}}\left( 1+\frac{\left( \left \vert x\right
\vert -t\right) ^{2}}{1+t}\right) ^{-\frac{\hat{P}}{2}}+e^{-\frac{c_{0}t%
}{3}}\left( 1+\frac{\left( \left \vert x\right \vert -t\right) ^{2}}{1+t}%
\right) ^{-N} \\
&&+\left( 1+t\right) ^{-m+\frac{5}{2}-\frac{\hat{P}}{2}}\left( 1+\frac{%
\left( \left \vert x\right \vert -t\right) ^{2}}{1+t}\right) ^{-\frac{%
\hat{P}}{2}}+\left( 1+t\right) ^{-m}\left( 1+\frac{\left( \left \vert
x\right \vert -t\right) ^{2}}{1+t}\right) ^{-N} \\
&\lesssim &\left( 1+t\right) ^{-m}\left( 1+\frac{\left( \left \vert x\right
\vert -t\right) ^{2}}{1+t}\right) ^{-N}\,.
\end{eqnarray*}

For $\left( x,t\right) \in \mathcal{D}_{4}$, we split the integral into
three parts%
\begin{equation*}
I=\left( \int_{0}^{\frac{2t}{3}}\int_{\mathbb{R}^{3}}+\int_{\frac{2t}{3}%
}^{t}\int_{\left \vert y\right \vert \leq \frac{\left \vert x\right \vert
+\tau }{2}}+\int_{\frac{2t}{3}}^{t}\int_{\left \vert y\right \vert >\frac{%
\left \vert x\right \vert +\tau }{2}}\right) \left( \cdots \right) dyd\tau
=I_{1}+I_{2}+I_{3}\text{.}
\end{equation*}%
It immediately follows that%
\begin{equation*}
I_{1}\lesssim e^{-\frac{c_{0}t}{3}}\lesssim \left( 1+t\right) ^{-m}\left( 1+%
\frac{\left( \left \vert x\right \vert -t\right) ^{2}}{1+t}\right) ^{-N}%
\text{.}
\end{equation*}%
If $\frac{2}{3}t\leq \tau \leq t$, $\left \vert y\right \vert \leq \frac{%
\left \vert x\right \vert +\tau }{2}$, then
\begin{equation*}
\tau -\left \vert y\right \vert \geq \tau -\frac{\left \vert x\right \vert
+\tau }{2}=\frac{\tau -\left \vert x\right \vert }{2}\geq \frac{t}{12}\geq
\frac{t-\left \vert x\right \vert }{12}\text{;}
\end{equation*}%
if $\frac{2}{3}t\leq \tau \leq t$, $\left \vert y\right \vert >\frac{%
\left
\vert x\right \vert +\tau }{2}$, then
\begin{equation*}
\left \vert x-y\right \vert \geq \frac{\left \vert x\right \vert +\tau }{2}%
-\left \vert x\right \vert =\frac{\tau -\left \vert x\right \vert }{2}\geq
\frac{t}{12}\geq \frac{t-\left \vert x\right \vert }{12}\text{.}
\end{equation*}%
Therefore,
\begin{eqnarray*}
I_{2} &\lesssim &\int_{\frac{2t}{3}}^{t}\int_{\mathbb{R}^{3}}e^{-c_{0}\left(
t-\tau \right) }\left( 1+\left \vert x-y\right \vert \right) ^{-\hat{P}%
}\left( 1+t\right) ^{-m}\left( 1+\frac{\left( \left \vert x\right \vert
-t\right) ^{2}}{1+t}\right) ^{-N}dyd\tau \\
&\lesssim &\left( 1+t\right) ^{-m}\left( 1+\frac{\left( \left \vert x\right
\vert -t\right) ^{2}}{1+t}\right) ^{-N}\text{,}
\end{eqnarray*}%
\begin{eqnarray*}
I_{3} &\lesssim &\int_{\frac{2t}{3}}^{t}\int_{\left \vert y\right \vert >%
\frac{\left \vert x\right \vert +\tau }{2}}e^{-c_{0}\left( t-\tau \right)
}\left( 1+t-\left \vert x\right \vert \right) ^{-\hat{P}}\left(
1+t\right) ^{-m}\left( 1+\frac{\left( \left \vert y\right \vert -\tau
\right) ^{2}}{1+\tau }\right) ^{-N}dyd\tau \\
&\lesssim &\left( 1+t\right) ^{-m+\frac{5}{2}}\left( 1+t-\left \vert x\right
\vert \right) ^{-\hat{P}}\lesssim \left( 1+t\right) ^{-m+\frac{5}{2}-%
\frac{\hat{P}}{2}}\left( 1+\frac{\left( \left \vert x\right \vert
-t\right) ^{2}}{1+t}\right) ^{-\frac{\hat{P}}{2}}\text{.}
\end{eqnarray*}%
Consequently,
\begin{equation*}
I\lesssim \left( 1+t\right) ^{-m}\left( 1+\frac{\left( \left \vert x\right
\vert -t\right) ^{2}}{1+t}\right) ^{-N}\text{.}
\end{equation*}

For $\left( x,t\right) \in \mathcal{D}_{5}$, we split the integral into
three parts%
\begin{equation*}
I=\left( \int_{0}^{\frac{t}{2}}+\int_{\frac{t}{2}}^{\frac{t+\left \vert
x\right \vert }{2}}+\int_{\frac{t+\left \vert x\right \vert }{2}}^{t}\right)
\int_{\mathbb{R}^{3}}\left( \cdots \right) dyd\tau =J_{1}+J_{2}+J_{3}\text{.}
\end{equation*}%
It immediately follows that
\begin{equation*}
J_{1}\lesssim e^{-\frac{c_{0}t}{2}}\int_{0}^{\frac{t}{2}}\int_{\mathbb{R}%
^{3}}\left( 1+\left \vert x-y\right \vert \right) ^{-\hat{P}}\left(
1+\tau \right) ^{-m}dyd\tau \lesssim e^{-\frac{c_{0}t}{3}}\text{,}
\end{equation*}%
\begin{eqnarray*}
J_{2} &\lesssim &\left( 1+t\right) ^{-m}e^{-\frac{c_{0}}{2}\left( \frac{%
t-\left \vert x\right \vert }{2}\right) }\int_{\frac{t}{2}}^{\frac{t+\left
\vert x\right \vert }{2}}\int_{\mathbb{R}^{3}}e^{-\frac{c_{0}\left( t-\tau
\right) }{2}}\left( 1+\left \vert x-y\right \vert \right) ^{-\hat{P}%
}dyd\tau \\
&\lesssim &\left( 1+t\right) ^{-m}\left( 1+\frac{\left( t-\left \vert
x\right \vert \right) ^{2}}{1+t}\right) ^{-N}\text{.}
\end{eqnarray*}%
For $J_{3}$, we further split the integral into two parts%
\begin{equation*}
J_{3}=\int_{\frac{t+\left \vert x\right \vert }{2}}^{t}\left( \int_{\left
\vert y\right \vert \leq \frac{\left \vert x\right \vert +\tau }{2}%
}+\int_{\left \vert y\right \vert >\frac{\left \vert x\right \vert +\tau }{2}%
}\right) \left( \cdots \right) dyd\tau \text{.}
\end{equation*}%
If $\frac{t+\left \vert x\right \vert }{2}\leq \tau \leq t$, $\left \vert
y\right \vert \leq \frac{\left \vert x\right \vert +\tau }{2}$, then
\begin{equation*}
\tau -\left \vert y\right \vert \geq \tau -\frac{\left \vert x\right \vert
+\tau }{2}=\frac{\tau -\left \vert x\right \vert }{2}\geq \frac{t-\left
\vert x\right \vert }{4}\text{;}
\end{equation*}%
if $\frac{t+\left \vert x\right \vert }{2}\leq \tau \leq t$, $\left \vert
y\right \vert >\frac{\left \vert x\right \vert +\tau }{2}$, then
\begin{equation*}
\left \vert x-y\right \vert \geq \frac{\left \vert x\right \vert +\tau }{2}%
-\left \vert x\right \vert =\frac{\tau -\left \vert x\right \vert }{2}\geq
\frac{t-\left \vert x\right \vert }{4}\text{.}
\end{equation*}%
Thus,
\begin{eqnarray*}
J_{3} &\lesssim &\int_{\frac{t+\left \vert x\right \vert }{2}%
}^{t}\int_{\left \vert y\right \vert \leq \frac{\left \vert x\right \vert
+\tau }{2}}e^{-c_{0}\left( t-\tau \right) }\left( 1+\left \vert x-y\right
\vert \right) ^{-\hat{P}}\left( 1+t\right) ^{-m}\left( 1+\frac{\left(
t-\left \vert x\right \vert \right) ^{2}}{1+t}\right) ^{-N}dyd\tau \\
&&+\int_{\frac{t+\left \vert x\right \vert }{2}}^{t}\int_{\left \vert
y\right \vert >\frac{\left \vert x\right \vert +\tau }{2}}e^{-c_{0}\left(
t-\tau \right) }\left( 1+t-\left \vert x\right \vert \right) ^{-\hat{P}%
}\left( 1+t\right) ^{-m}\left( 1+\frac{\left( \left \vert y\right \vert
-\tau \right) ^{2}}{1+\tau }\right) ^{-N}dyd\tau \\
&\lesssim &\left( 1+t\right) ^{-m}\left( 1+\frac{\left( t-\left \vert
x\right \vert \right) ^{2}}{1+t}\right) ^{-N}+\left( 1+t\right) ^{-m+\frac{5%
}{2}}\left( 1+t-\left \vert x\right \vert \right) ^{-\hat{P}} \\
&\lesssim &\left( 1+t\right) ^{-m}\left( 1+\frac{\left( t-\left \vert
x\right \vert \right) ^{2}}{1+t}\right) ^{-N}\text{.}
\end{eqnarray*}%
Combining $J_{1}$, $J_{2}$ and $J_{3}$ gives
\begin{equation*}
I\lesssim \left( 1+t\right) ^{-m}\left( 1+\frac{\left( t-\left \vert x\right
\vert \right) ^{2}}{1+t}\right) ^{-N}\text{.}
\end{equation*}%
The proof is completed.
\end{proof}

\hspace*{2em}

\noindent \textbf{Acknowledgment}\hspace*{1em}
This work is partially supported by the National Key R\&D Program of China under grant 2022YFA1007300. Y.-C. Lin is supported by the National Science and Technology Council under
the grant NSTC 112-2115-M-006-006-MY2. H.-T. Wang is supported by NSFC under Grant No.
12371220 and 12031013, the Strategic Priority Research Program of Chinese Academy of Sciences under Grant No. XDA25010403. K.-C. Wu is supported by the National Science and Technology Council under the grant NSTC 112-2628-M-006-006 -MY4 and National Center for Theoretical Sciences.
\\

\noindent \textbf{Data Availability}\hspace*{1em}
Data sharing is not applicable to this article as no datasets were generated or analyzed during the current study.
\\

\noindent \textbf{Declarations}
\\

\noindent \textbf{Conflict of interest}\hspace*{1em}
The authors declare that they have no Conflict of interest.

\end{document}